\documentclass[12pt]{amsart}
\usepackage{amsmath,amssymb,amsfonts,a4wide}
\usepackage{mathrsfs,euscript}
\usepackage{xcolor}
\usepackage{hyperref}
\hypersetup{
    colorlinks=true, 
    linktoc=all,     
    linkcolor=blue,  
}
\usepackage{url}

\def\sideremark#1{\ifvmode\leavevmode\fi\vadjust{\vbox to0pt{\vss
\hbox to 0pt{\hskip\hsize\hskip1em%
\vbox{\hsize2cm\tiny\raggedright\pretolerance10000%
\noindent {\color{red}{#1}}\hfill}\hss}\vbox to8pt{\vfil}\vss}}}%

\theoremstyle{plain}
\newtheorem{propn}{Proposition}[section]
\newtheorem{thm}[propn]{Theorem}
\newtheorem{lemma}[propn]{Lemma}
\newtheorem{cor}[propn]{Corollary}
\newtheorem{fct}[propn]{Fact}

\newtheorem{exo}[propn]{Example}

\theoremstyle{definition}
\newtheorem{defn}[propn]{Definition}

\theoremstyle{remark}
\newtheorem{rem}{Remark}

\newenvironment{rlist}
{

\begin{enumerate}}
{\end{enumerate}}

\def \bnabla{\overline{\nabla}}
\def \V{\mathcal{V}}
\def \bbV{\mathbb{V}}
\def \bbR{\mathbb{R}}
\def \bbC{\mathbb{C}}

\def \m{\mathfrak{m}}

\def \s {\mathfrak{s}}
\numberwithin{equation}{section}

\newcommand{\g}{\mathfrak{g}}
\newcommand{\h}{\mathfrak{h}}

\newcommand{\LL}{\mathfrak{l}}
\newcommand{\z}{\mathfrak{z}}
\newcommand{\n}{\mathfrak{n}}

\newcommand{\rad}{\mathfrak{r}}
\newcommand{\uh}{\underline{\h}}
\newcommand{\ug}{\underline{\g}}
\newcommand{\ur}{\underline{\rad}}
\newcommand{\Ll}{\mathfrak{l}}
\newcommand{\E}{\mathcal{E}}
\newcommand{\Kk}{\mathfrak{k}}
\newcommand{\gl}{\mathfrak{gl}}
\newcommand{\Sl}{\mathfrak{sl}}
\newcommand{\og}{\overline{\g}}
\newcommand{\oh}{\overline{\h}}
\newcommand{\so}{\mathfrak{so}}
\newcommand{\un}{\mathfrak{u}}
\newcommand{\spr}{\mathfrak{sp}}
\newcommand{\hol}{\mathfrak{hol}}

\newcommand{\stg}{\mathfrak{stab}(R^g)}

\DeclareMathOperator{\IM}{Im}
\DeclareMathOperator{\Der}{Der}
\DeclareMathOperator{\Hom}{Hom}
\DeclareMathOperator{\Aut}{Aut}
\DeclareMathOperator{\aut}{\mathfrak{aut}}
\DeclareMathOperator{\ung}{\underline{g}}
\DeclareMathOperator{\D}{D}

\DeclareMathOperator{\VV}{V}
\DeclareMathOperator{\mL}{L}
\DeclareMathOperator{\ad}{ad}
\DeclareMathOperator{\Ad}{Ad}
\DeclareMathOperator{\me}{Met}
\DeclareMathOperator{\Tr}{Tr}
\DeclareMathOperator{\Hol}{Hol}
\DeclareMathOperator{\Ric}{Ric}
\DeclareMathOperator{\ric}{ric}
\DeclareMathOperator{\di}{d}
\DeclareMathOperator{\li}{\mathscr{L}} 
\DeclareMathOperator{\spa}{span}
\DeclareMathOperator{\Sym}{Sym}
\DeclareMathOperator{\GL}{GL}
\DeclareMathOperator{\BK}{B} 
\def \H{\mathcal{H}}
\def \scrV{\mathscr{V}}
\def \scrH{\mathscr{H}}
\begin{document}
\title{3-symmetric spaces, Ricci solitons, and homogeneous structures}
\author[T. Murphy]{Thomas Murphy}
\address[T. Murphy]{Department of Mathematics, California State University Fullerton, 800 N. State College Blvd., Fullerton 92831 CA, USA}
\email{tmurphy@fullerton.edu}
\urladdr{http://www.fullerton.edu/math/faculty/tmurphy/}
\author[P.-A.Nagy]{Paul-Andi Nagy}
\address{Paul-Andi Nagy\\
Center for Complex Geometry, Institute for Basic Science (IBS)\\
55 Expo-ro, Yuseong-gu, 34126 Daejeon, South Korea
}
\email{paulandin@ibs.re.kr}

\begin{abstract}
The full classification of Riemannian $3$-symmetric spaces is presented. Up to Riemannian products the main building 
blocks consist in (possibly symmetric) spaces with semisimple isometry group, nilpotent Lie groups of step at most $2$ and spaces of type III and IV. 

For the most interesting family of examples, the Type III spaces, we produce an explicit description including results concerning the moduli space of all $3$-symmetric metrics living on a given Type III space. Each moduli space contains a unique distinguished point corresponding to an (almost-K\"ahler) expanding Ricci soliton metric.  
For certain classes of  3-symmetric metrics there are many different groups acting transitively and isometrically on a fixed Riemannian 3-symmetric space.  The construction of  expanding Ricci solitons on   spaces of Type III is also shown to generalize to \emph{any} effective representation of a simple Lie group of non-compact type, yielding a very general construction of homogeneous Ricci solitons. We also give a procedure to compute the isometry group of any Ambrose--Singer space. 
\end{abstract}

\maketitle 

\setcounter{tocdepth}{1}
\tableofcontents

\section{Introduction} \label{intro}
Arguably the most notable and widely studied class of homogeneous spaces after the Riemannian symmetric spaces are the  $3$-symmetric spaces. Taking a connected real Lie group $G$ and an automorphism $\theta$ of $G$ with order $3$, let $G^{\theta}$ denote the fixed point set of $\theta$ and $G^{\theta}_0$ its identity component.  Choosing the  closed  isotropy subgroup  $H  = G^{\theta}_0$
yields a  simply connected 3-symmetric space $M = G/H$. The object of study in the present paper are the  \emph{Riemannian $3$-symmetric spaces}, which are defined as the  3-symmetric spaces $(M,g)$ equipped with a $G$-invariant  Riemannian metric.  The corresponding homogeneous spaces arising from an automorphism  of order 2 are exactly the Riemannian symmetric spaces of Cartan.

Their study has a long history since the introduction of $k$-symmetric spaces by Graham--Leger \cite{gl} (see also \cite{kow}, \cite{jim}). As justification for their importance, we briefly record that they admit quasi-K\"ahler metrics \cite{gr1}, naturally occur in the study of nearly-K\"ahler metrics, and can be viewed as a generalization of Hermitian symmetric spaces, which are the most important family of symmetric spaces. Progress towards the classification of 3-symmetric spaces with  $G$ semisimple was achieved by Gray--Wolf \cite{gw} \cite{gw2}, assuming the transitive group action $G$ is holomorphic.  Even under this assumption  there remain issues in giving an explicit list of all  Riemannian 3-symmetric spaces in their classification, as the duality correspondence can affect the signature of the metric.  Examples with a solvable transitive group action  are far  more sporadic \cite{discala2} 
and the classification has remained largely open.

Let us begin by mentioning an example highlighting  some peculiarities of Riemannian 3-symmetric spaces which should persuade the reader that full classification results are far more subtle than the corresponding classification of Riemannian symmetric spaces. $\mathbb{C}P^{2n+1}$ equipped with its Fubini--Study metric is a Hermitian symmetric space and hence  is naturally endowed with a 3-symmetric structure as already mentioned.  At the same time, $\mathbb{C}P^{2n+1}$ arises as the twistor space of $\mathbb{H}P^n$ and thus may be equipped with a nearly K\"ahler, homogeneous (and hence 3-symmetric) metric. The upshot is that the same manifold can admit two different 3-symmetric metrics, a phenomena which does not happen in the classical theory of symmetric spaces. In addition, as explained in \cite{MuNa}, for the nearly--K\"ahler metric there are {\textit{two  different}  groups acting transitively and isometrically on the same Riemannian manifold, but one of these group actions does not preserve the 3-symmetric structure. This highlights key difficulties with the classical definitions: the transitive groups and metrics are intertwined together and it is not at all clear how to explain these phenomena  from this perspective.
\subsection{The main result} \label{mm1}
Our approach combines the Lie theoretic set-up and almost Hermitian geometry as follows.

In an analogous fashion to  the classical theory of Riemannian symmetric spaces to any Riemannian $3$-symmetric space one associates 
a Riemannian $3$-symmetric Lie algebra. This consists of a Lie algebra $\g$ together with an automorphism $\sigma \in \Aut(\g)$ with $\sigma^3=1$ and $\sigma \neq 1$. This entails the reductive decomposition $\g=\h \oplus \bbV$ where 
the subalgebra $\h:=\ker(\sigma-1)$ and the vector subspace $\bbV:=\mathrm{Im}(\sigma^2+\sigma+1)$ admits a linear complex structure 
$J = \frac{1}{\sqrt{3}}1_{\vert \bbV} + \frac{2}{\sqrt{3}} \sigma_{\vert \bbV}.$ The adjective Riemannian translates 
into the assumption that the subalgebra $\h:=\ker(\sigma-1)$ is compactly embedded, in the sense that the representation $(\h,\bbV)$ 
is tangent to a Lie group representation $(H,\bbV)$, with $H$ compact.

More geometrically, Riemannian $3$-symmetric spaces are almost complex manifolds $(M^{2m},J)$ which admit a Hermitian Ambrose-Singer connection $\mathrm{D}$ with torsion type  $\tau^{\mathrm{D}} \in \Lambda^1M \otimes \gl_{J}^{\perp}TM$; furthermore we require the holonomy group $\Hol(\mathrm{D})$ be compact. This assumption allows treating all compatible Riemannian metrics 
$$\mathrm{Met}(M,J,\D):=\{g : \mathrm{D}g=0 \ \mathrm{and} \ g(J \cdot ,J \cdot)=g \}
$$
simultaneously.  Building the infinitesimal model for $\mathrm{D}$ at a point $x \in M$ leads us to the study of the transvection Lie algebra $\ug=\mathbb{V} + [\mathbb{V}, \mathbb{V}]$ where $\bbV=T_xM$. Then $\ug$ is Riemannian $3$-symmetric 
in the sense above and the corresponding Lie group $\underline{G}$ always acts transitively \cite{kow} on $M$.

Our first main result, which classifies Riemannian $3$-symmetric spaces according to the algebraic structure of $\ug$ is as follows.

\begin{thm}\label{t1}
Every  connected, simply connected  Riemannian 3-symmetric space is  a Riemannian product with factors 
\begin{rlist}
\item  3-symmetric spaces with $\underline{\g}$ semisimple (spaces of Type I); 
\item  3-symmetric spaces  with $\underline{\g}$ nilpotent, $[\ug,[\ug,\ug]]=0$, (spaces of Type II);
\item 3-symmetric spaces of Type III; 
\item 3-symmetric spaces of Type IV. 
\end{rlist}
\end{thm}
We refer the reader to the text for precise definitions of all objects.  

\begin{rem} \label{r1} While Theorem \ref{t1} is stated at the Riemannian level,  the  splitting is independent of the metric chosen. Namely, the Riemannian homogeneous space $M$ may always be decomposed as a product of Riemannian 3-symmetric spaces in such a manner that the transitive group of each factor is either semisimple, nilpotent, or solvable. In general the classification of 3-symmetric spaces without a compatible Riemannian metric is beyond the reach of our techniques. See Proposition \ref{gen-exa} for further details.  
\end{rem}
We now explain the result in Theorem \ref{t1} in some detail.
As already mentioned, 3-symmetric spaces with $\underline{\g}$ semisimple have been intensively studied  \cite{gw}, \cite{gw2}, but 
their complete classification remains an open problem; indeed those references determined instances when $\Aut^0(M,g)$ is semisimple and {\it{moreover}} preserves the almost complex structure $J$; see Theorem \ref{t1.5} which establishes this is indeed the case unless $(M,g)$ is Riemannian symmetric.

Classifying examples whose transvection algebra is a  2-step nilpotent Lie algebra is still wide open ( see  Proposition \ref{gen-exa}), though we can reduce it to an algebraic problem concerning 2-step nilpotent Lie algebra.  These are 3-symmetric spaces with flat Chern connection $\mathrm{D}$, a topic studied in \cite{discala1}.  

The new occurrences in the classification Theorem \ref{t1} are thus spaces of type III and IV; for those instances the transvection algebra is neither semisimple nor nilpotent. As homogeneous spaces those spaces are of the form 
$$ M = \underline{G}/ K = (\VV\rtimes_{\pi} G)/K
$$
with data entering the construction given by 
\begin{itemize}
\item[$\bullet$]a simply connected Hermitian symmetric space $G/K$ where the Lie algebra $\mL$ of $G$ is simple and has Cartan decomposition  $\mL= \mathfrak{k} \oplus \mathcal{H}$
\item[$\bullet$] a linear representation $\pi:G \to \GL(\VV)$ on a real vector space such that the tangent representation $\rho=(d \pi)_e: \mL \rightarrow  \mathfrak{gl}(\VV)$ is admissible; see Definition \ref{adm-ff} for full details.
\end{itemize}
In the construction above the transvection algebra of $M$ is given by the semi-direct product 
$ \underline{\mathfrak{g}}=\VV\rtimes_{\rho} \mL$. 

Spaces of type III are obtained when requiring the Hermitian symmetric pair $(\mL,\Kk)$ have {\it{non-compact}} type whilst spaces 
of type IV correspond to having $(\mL,\Kk)$ of {\it{compact}} type. Riemannian 3-symmetric spaces of Type IV are, at metric level, Riemannian products of flat complex space and Hermitian symmetric spaces of noncompact type.  However they do \emph{not} overlap with the first two  cases in Theorem \ref{t1} as the almost-complex structure is not the product one.

Theorem \ref{t1} provides a full classification result, up to the case of Riemannian symmetric spaces discussed above. Indeed, admissible 
representations of simple Lie algebras of non-compact have been fully classified by Satake \cite{Sat1}, see Table \ref{table1} in the paper; note that Satake's motivation was somewhat different, namely
the study of totally geodesic, complex submanifolds of the Hermitian symmetric space $\frac{Sp(m,\bbR)}{U(m)}$. Admissible representations of Hermitian Lie algebras of compact type can be obtained from Satake's list by an appropriate notion of duality, see Table \ref{table2}.

To end this section we briefly streamline some of the arguments used to prove Theorem \ref{t1}.
We start from the geometric interpretation of Riemannian $3$-symmetric spaces as Ambrose-Singer almost complex manifolds $(M^{2n},\D,J)$. Next we consider the transvection algebra $\ug$, built from the infinitesimal model for the Ambrose-Singer connection 
$\mathrm{D}$ and determine its radical $\mathfrak{r}(\ug)$. Using holonomy reduction techniques we establish that this radical 
is determined by the curvature nullity of $\mathrm{D}$, namely $\{X \in TM : R^{\D}(X, \cdot)=0\}$. The splitting of $\ug$ into 
algebraic types 
then follows from the Levi-Malcev theorem by representation theoretic arguments. Note that we also need to show that the study of admissible representations reduces to the case when $\mL$ is simple. The factors in the splitting of $\ug$ are then showed to correspond geometrically to de Rham factors, observation which leads to the proof of Theorem \ref{t1}.
\subsection{Geometric invariants of spaces of Type III} \label{giIII}
Our second main theorem discusses the structure of compatible $3$-symmetric metrics on spaces of type III. Remarkably, the study 
of such metrics reduces to the study of various group actions on the centraliser 
$\mathfrak{c}(\mL,\rho,\VV) $ of the defining admissible representation $\rho:\mL \to \gl(\VV)$. For the notation $\mathfrak{c}^{-}(\mL,\rho,\VV)$ in the statement below the reader is referred to Section \ref{prep-1}. 
\begin{thm}\label{t2}
Let $M=(\VV \rtimes_{\pi} G) \slash K$ be a Riemannian $3$-symmetric space of Type III.
\begin{rlist}
\item any $\underline{G}=V \rtimes_{\pi} G$-invariant metric in $\mathrm{Met}(M,J,\mathrm{D})$ is de Rham irreducible 
\item the set of $\underline{G}$-invariant metrics in $\mathrm{Met}(M,J,\mathrm{D})$ is in $1:1$ correspondence with 
$ \mathfrak{c}^{-}(\mL,\rho,\VV) \times \bbR_{>0}$
\item there is a bijective correspondence between irreducible Type III spaces and admissible representations of Lie algebras of non-compact type, all of which are classified in Table \ref{table1}
\item if $\rho$ is of real type, the moduli space of Riemannian 3-symmetric metrics on $M$ is a point (i.e. there is a unique Riemannian 3-symmetric metric)
\item  if $\rho$ is of complex type, the moduli space of Riemannian 3-symmetric metrics on $M$  can have arbitrarily large dimension
\item $M$ admits a unique $\underline{G}$-invariant  metric $g$  which is almost-K\"ahler with respect to $J$. This metric is an expanding Ricci soliton.  
\end{rlist}
\end{thm}

Theorem \ref{t2}(iii) shows that, for any $d\in \mathbb{N}$, one can find a Type III space whose  moduli space of  compatible 3-symmetric metrics has dimension $d$. See Theorem \ref{cpx-ty2} for further details. This is remarkable as it means there are uncountably many $3$-symmetric metrics living on the same underlying manifold.   Another interesting consequence of the theorem is that  any admissible representation of a simple Lie group of non-compact type  at Lie algebra level yields an irreducible  Riemannian 3-symmetric space. This has the consequence that one can take direct sums of such admissible representations  to yield new irreducible Riemannian metrics.

The almost-K\"ahler metrics in part (vi) of Theorem \ref{t2} have additional curvature properties, in the sense they belong to 
the class $\mathcal{AK}_2$, studied in \cite{Ap1,Ap3,Chi-Na}. 
We have also discovered that almost-K\"ahler Type III spaces are also of central interest in geometric analysis, due to the fact that they furnish prototypical examples of a new, very general, construction of (homogeneous) Ricci solitons metrics which is given below. 
 
\begin{thm}\label{t3}
Let $\mL$ be a simple Lie algebra of noncompact type, $\mL= \mathfrak{k} \oplus \mathcal{H}$  be a Cartan decomposition, and $\rho: \mL\rightarrow \mathfrak{gl}(\VV)$ a faithful linear representation. Then  $(\VV\rtimes_{\pi} G)/K$ admits an expanding Ricci soliton metric. 
\end{thm}
Above $G$ is the simply connected Lie group with Lie algebra $\mL$ and $K \subseteq G$ is the connected Lie subgroup corresponding to the Lie algebra $K$; also $\pi:G \to \GL(\VV)$ is the representation exponentiating $\rho$.

 Every admissible representation is faithful, and the Ricci soliton metric this theorem produces on Type III spaces is exactly the almost-K\"ahler metric of Theorem \ref{t2}(iv).  It is worth noting the degree of generality in this construction: no assumption is made on the representation $\rho$. The theorem provides a systematic way to construct homogeneous Ricci solitons from \emph{arbitrary} representations of non-compact semisimple Lie groups. We refer the reader to \cite{LaLa,jab} for general facts regarding homogeneous 
Ricci solitons.
 
We also produce many examples of polar Riemannian  foliations on  Type III spaces. This might not seem as important as the other theorems, but it was our initial motivation to study such spaces and turns out to play a crucial  role in the proof of Theorem \ref{t2}, as well as inspiring the proof of Theorem \ref{t3} .  
 \begin{thm}\label{t4}
 Every  irreducible Riemannian 3-symmetric space of Type III  admits a polar Riemannian foliation. \end{thm}
\begin{rem} In the case that the Riemannian metric is the compatible almost K\"ahler one, a complex structure $I$ may be chosen so that  $g$ is K\"ahler and the foliation is complex with respect to $I$.  
In this case  we also produce a geometric interpretation of the theorem: namely $\VV = \text{ker}(\Ric^g)$. 
\end{rem}
\subsection{Non-uniqueness phenomena} \label{n-unq}
Another phenomena investigated is to  show that, in our setting, many different descriptions of $(M,g)$ as a homogeneous space are possible.   This echos the behavior in the semisimple case which we have already met in the example of the twistor space of $\mathbb{H}P^n.$  
Firstly, in Section \ref{iso-fin}, we explicitly produce an algorithm to compute the isometry group of any Riemannian Ambrose--Singer space, see Theorem \ref{embedd}. The latter theorem is perhaps more direct than other existing results, for instance \cite{GoWi}, since it takes into account the properties of the intrinsic torsion tensor. This algorithm is of independent interest, but we apply it to prove the following

\begin{thm}\label{t1.5}  If  an irreducible  Riemannian 3-symmetric space is not locally symmetric, then every isometry is holomorphic. 

\end{thm} 
A subtle  issue at this point is that a Riemannian 3-symmetric space can admit transitive isometric group actions which do not preserve the 3-symmetric structure (see \cite{MuNa} where this issue was first explored). Somewhat amusingly, the upshot is that transitive group actions on a Riemannian 3-symmetric space are now more poorly understood for the case of the  Hermitian symmetric spaces, despite them being the easiest examples of 3-symmetric spaces.   A major application of Theorem \ref{t1.5} is to show that the lists appearing in  Gray--Wolf \cite{gw} \cite{gw2} are the only examples of Riemannian 3-symmetric spaces of Type I (i.e. $\ug$ semisimple) which are not symmetric. Note that if the manifold is a compact locally symmetric space, one could apply the well-known work of Onishchik \cite{on} to complete the full classification by a case-by-case analysis. We leave this as a project for the future. 

Furthermore, recall that a  \emph{presentation} of $(M,g)$ is given by a transitive isometric  Lie group action  $G$ on $(M,g)$, with isotropy $H$, such that the pair $(G,H)$  satisfy the definition of a Riemannian 3-symmetric space. 
Given any presentation of $(M,g)$ as a Riemannian 3-symmetric space $G/H$, we characterize the possible choices of $G$ at Lie algebra level in Theorem \ref{thm-gen}. However it is not obvious how to determine all possible transitive group actions on $(M^n,g)$, even for Type III spaces.

We also record that for spaces of Type III,  Theorem \ref{t1.5} allows us to explicitly compute the full isometry group in terms of the centralizer of the representation, as explained in Theorem \ref{isoIII}. 

\subsection{Organization of the paper and further results} \label{org-p}
Owing to the length of this work and the interconnectedness between the main results, we offer a roadmap to the reader to guide their reading. We start with a prototypical example in Section \ref{sdp} and formulate a new interpretation of the definition of  Riemannian 3-symmetric spaces in Section \ref{riem3} which is more convenient for our purposes. In Theorem \ref{unq} we give several sufficient 
conditions for the uniqueness of the Ambrose-Singer connection $\D$; this extends to other torsion types the Skew-symmetric Holonomy 
Theorem proved in \cite{Na-prol}, see also \cite{O}.

We then move on to set notation and discuss Riemannian 3-symmetric Lie algebras, the issue of regularity of infinitesimal models, the definition of admissible representations, and the main splitting theorems at algebraic level  in Sections \ref{o3}--\ref{admr1}. The proof of Theorem \ref{t1} is given in Section  \ref{reginfpf1} and all occurrences on the list of admissible representations are discussed in Section  \ref{met-adm}. Theorem \ref{t3} is rephrased using  the language of background metrics  and proven in Theorem \ref{sol-1}. Next up is the proof of Theorem \ref{t1.5} which  appears in Section \ref{der-tv}, although a crucial part of the proof is postponed until Theorem \ref{split-iso}.

Theorem \ref{t2} (i) follows from (ii) in Theorem \ref{irred-R},  part (ii) is proved in Proposition \ref{iso-invm1}, part (iii) is Theorem \ref{claas-nc}, part (iv) appears in the text as Proposition \ref{r-type},
part (v) is Theorem \ref{cpx-ty} whilst part (vi) follows from combining Theorems  \ref{sol-1} and \ref{sol-2}. The proof of Theorem \ref{t3} is given in Theorem \ref{sol-1}. The proof of Theorem  \ref{t4} is restated and proven  in Section \ref{it-curv} as Theorem \ref{polar}. Finally the proof of Theorem \ref{t1.5} is given in Theorem \ref{split-iso} of the paper.

 
\subsection{Notations and conventions} \label{not-conv}
As this paper uses a blend of geometric and Lie algebraic techniques and notions 
we briefly set up notation to be used systematically in what follows. All Riemannian manifolds are simply connected, and hence connected, throughout.

All vector spaces are considered over the reals unless explicitly mentioned. 
When 
$\bbV$ is a vector space $\mathfrak{gl}(\bbV)$ denotes the space of linear endomorphisms of $\bbV$. If $J$ is a linear complex structure on $\bbV$ we consider the Lie algebra
$\mathfrak{gl}_J(\bbV):=\{f \in \mathfrak{gl}(\bbV):fJ=Jf\}$ together with its complement in $\gl(\bbV)$, given by
$\mathfrak{gl}^{\perp}_J(\bbV):=\{f \in \mathfrak{gl}(\bbV):fJ=-Jf\}$. If $g$ is a non-degenerate symmetric bilinear form on $\bbV$ we denote 
$$\mathfrak{so}(\bbV,g):=\{f \in \mathfrak{gl}(\bbV):g(fv,w)+g(v,fw)=0 \ \mbox{for all} \ v,w \in \bbV\}$$ and 
we also let $\Sym^2(\bbV,g):=\{S \in \mathfrak{gl}(\bbV):S \ \mbox{is symmetric with respect to} \ g \}$. We say that a linear complex structure $J$ is compatible with a pseudo-Riemannian metric $g$ provided that $g(J \cdot, J\cdot)=g$. If $\bbV$ carries such a structure  the Lie algebra preserving $(g,J)$ is $\mathfrak{u}
(\bbV,g,J):=\mathfrak{so}(\bbV,g) \cap \mathfrak{gl}_J(\bbV)$.

Whenever $\g$ is a Lie algebra  we indicate 
with $\z(\g)$ its center and   $\rad(\g)$ its radical.
The radical $\rad(\g)$ is the maximal solvable ideal in $\g$ and it is alternatively determined from $ B(\rad(\g),[\g,\g])=0$
where $B$ is the Killing form of $\g$. The radical of the Killing form, denoted 
by $\g^{\perp}$, is determined from 
$ B(\g^{\perp},\g)=0.$ Clearly $\g^{\perp}$ is an ideal in $\g$.

These ideals are related by the sequence of inclusions 

\begin{equation} \label{incl-1}
[\rad(\g),\rad(\g)] \subseteq \g^{\perp} \subseteq \rad(\g).
\end{equation}
(see e.g. \cite{Bourbaki}, page 49 for proof)

Whenever $\L$ is a Lie algebra and $\rho:\mL \to \mathfrak{gl}(\VV)$ is a linear representation we indicate with 
$t_{\rho}:\mL \times \mL \to \mathbb{R}$ its 
trace form, that is $t_{\rho}(F,G)=\Tr(\rho(F) \circ \rho(G))$. We also record that when $\mL$ is semisimple, any linear representation $\rho:\mL \to \mathfrak{gl}(\VV)$ satisfies  $\rho(\mL) \subseteq \Sl(\VV)$. This easily follows from having $\mL = [\mL, \mL]$. Furthermore, assuming that $\rho_i:\mL \to \gl(\VV_i), i=1,2$
are Lie algebra representations we denote 
$$\Hom_{\mL}(V_1,V_2)=\{ F:\VV_1 \to \VV_2 : F \rho_1(h)v=\rho_2(h)Fv, \ (h,v) \in \mL \times \VV_1 \}.$$
We will frequently use the analogous notation for Lie subalgebras $\Kk \subseteq \mL$.
\subsection*{Acknowledgements}
The research of Paul-Andi Nagy is supported by the Institute for Basic Science (IBS-R032-D1).

\section{Semi-direct products and homogeneous structures}\label{sdp}
We open with a discussion of   some basic material concerning semi-direct products. Such structures play a key r\^ole in the paper both in the classification results for 3-symmetric spaces.  Furthermore, in this section  they are also shown to be prototypical  examples of manifolds with many different homogeneous structures. 
\subsection{Semi-direct products of Lie algebras} \label{semi-Lie}
Let $\mL$ be a Lie algebra over $\mathbb{R}$ equipped with a linear representation $\rho:\mL \to \gl(\VV)$, where $\VV$ is some real vector space. We recall that $\rho$ is called faithful provided that $\ker(\rho)=0$.
A few notions from representation theory we will frequently deal with in this paper are the fixed point set 
respectively the centraliser of $\rho$, defined according to 
\begin{equation*}
\VV^{\mL}:=\{v \in \VV : \rho(\mL)v=0\} \ \mathrm{respectively} \ \mathfrak{c}(\mL,\rho,\VV):=\{f \in \gl(\VV):[\rho(\mL),f]=0\}.
\end{equation*}
Clearly the centraliser is a Lie subalgebra of $\gl(\VV)$. A canonical way to enlarge the representation $(\mL,\rho)$ is by means of its {\it{normal extension}} $\widetilde{\rho}:\widetilde{\mL} \to \gl(\VV)$ defined by 
\begin{equation*}
\widetilde{\mL}:=\mL \oplus \mathfrak{c}(\mL,\rho,\VV), \ \widetilde{\rho}(l+f):=\rho(l)+f.
\end{equation*}
Above $\widetilde{\mL}$ is equipped with the direct product Lie algebra structure. This construction is used below to describe the derivations of semi-direct products and will also appear naturally in the next section (see also Remark \ref{rmk-normiser} for more details).

The semi-direct product $\VV \rtimes_{\rho} \mL$ is the vector space $\VV \oplus \mL$ equipped with the Lie bracket 
\begin{equation} \label{semi-1}
[v_1,v_2]=0, \ [l,v_1]=\rho(l)v_1, \ [l_1,l_2]=[l_1,l_2]^{\mL}.
\end{equation}
Equivalently $\VV$ is an abelian ideal and $\mL$ is a subalgebra with $\VV \cap \mL=0$. 

From now on we assume that $\mL$ is semisimple which makes that the representation $\rho$ satisfies $\rho(\mL) \subseteq \Sl(\VV)$. We also assume that the representation $\rho$ is faithful and we describe some of the invariants of the Lie algebra 
$$\g:=\VV \rtimes_{\rho} \mL$$ of geometric relevance in this paper. Because $\mL$ is semisimple, the representation $\rho$ is completely reducible hence one may choose a $\mL$-invariant subspace $\VV_{\perp} \subseteq \VV$ such that 
\begin{equation} \label{weyl-1}
\VV=\VV^{\mL} \oplus \VV_{\perp}.
\end{equation}
Note that we have $\rho(\mL)\VV_{\perp}=\VV_{\perp}$ as well as  $(\VV_{\perp})^{\mL}=0$ and also that $\VV_{\perp}$ is uniquely determined from $\VV_{\perp}=\rho(\mL)\VV$.
\begin{propn} \label{cen-rad-p}
Let $\g:=\VV \rtimes_{\rho} \mL$ where $\mL$ is semisimple and $\rho:\mL \to \gl(\VV)$ is a faithful linear representation. The following hold
\begin{itemize}
\item[(i)] the radical $\rad(\g)=\VV$
\item[(ii)] the center $\z(\g)=\VV^{\mL}$ and $[\g,\g]=\mL \oplus \VV_{\perp}$
\item[(iii)] we have a Lie algebra isomorphism $\Der(\g) \cong \VV_{\!\perp} \rtimes_{\tilde{\rho}} \widetilde{\mL}$.
\end{itemize}
\end{propn}
\begin{proof}
(i) Since $\VV$ is an abelian ideal in $\g$ we must have $\VV \subseteq \rad(\g)$. It follows that $\rad(\g)=\VV \oplus \mL_0$ as vector spaces, where $\mL_0=\rad(\g) \cap \mL$. As the latter is an ideal in $\mL$ it must be semisimple; because $\rad(\g)$ is solvable it cannot contain non-zero semisimple subalgebras thus $\mL_0=0$ and the claim is proved.\\
(ii) follows by combining the general inclusion $\z(\g) \subseteq \rad(\g)$, part (i) and the definition of the Lie bracket in \eqref{semi-1}.\\
(iii) Pick $\zeta \in \Der(\g)$. Then $\zeta$ preserves the radical of $\g$ and since 
$\rad(\g)=\VV$ by (i) it follows that $\zeta \VV \subseteq \VV$. Thus, with respect to the splitting $\g=\VV \oplus \mL$ we have 
$\zeta=\left ( \begin{array}{cc} f & 0 \\ c & \zeta_0 \end{array} \right )$ with $(f,\zeta_0,c) \in \gl(\VV)\oplus \gl(\mL) \oplus \Hom(\mL,\VV)$. From \eqref{semi-1} having $\zeta$ a derivation is easily seen to be equivalent to having
$\zeta_0 \in \Der(\mL),c$ a $1$-cocycle with values in $\VV$, that is $c[l_1,l_2]=\rho(l_1)
c(l_2)-\rho(l_2)c(l_1)$ as well as $[f,\rho(l)]=\rho(\zeta_0l)$ for all $l \in \mL$. Using that $\mL$ is semisimple and Whitehead's Lemma leads to $\zeta_0=\ad_{l_0}, c(l)=\rho(l)v_0$ for some uniquely determined $(l_0,v_0) \in \mL \times \VV_{\perp}$. It follows that $f=\rho(l_0)+g$ with $g \in \mathfrak{c}(\mL,\rho,\VV)$. Summarising 
\begin{equation} \label{der-gen2}
\zeta_{\vert \VV}=\rho(l_0)+g, \ \zeta l=\ad_{l_0}l+\rho(l)v_0, \ l \in \mL.
\end{equation}
This entails, after a short calculation, the claim on the Lie algebra structure in 
$\Der(\g)$.
\end{proof}
We end this section by explaining, in the semisimple case, our choice of terminology for the extension $(\widetilde{\mL},\widetilde{\rho})$.  
\begin{rem} \label{rmk-normiser}
Assume that $\mL$ is semisimple and that $\rho$ is faithful.
\begin{itemize}
\item[(i)] In this case, $\widetilde{\rho}(\widetilde{\mL})$ is equivalently described as the normaliser 
$$\mathfrak{n}_{\gl(\VV)}(\rho(\mL)):=\{f \in \gl(\VV) : [f,\rho(\mL)] \subseteq \rho(\mL)\}$$ of $\rho(\mL)$ in 
$\gl(\VV)$. The proof of this well known fact goes as follows. Consider the representation $\mL \times \Sl(\VV) \to \Sl(\VV)$ given by 
$(l,f) \mapsto [\rho(l),f]$. Because $\mL$ is semisimple this representation is completely reducible thus 
the invariant subspace $\rho(\mL)$ admits an invariant complement $\mathrm{E}$, that is $\Sl(\VV)=\rho(\mL) \oplus \mathrm{E}$. Splitting accordingly elements in the normaliser leads to the equality 
$$\mathfrak{n}_{\gl(\VV)}(\rho(\mL))=\rho(\mL) \oplus \mathfrak{c}(\mL,\rho,\VV).$$ 
\item[(ii)]
There is a canonically defined matrix Lie group $G^{\widetilde{\mL}}$ equipped with a representation $\pi:G^{\widetilde{\mL}} \to \GL(\VV)$ such that 
$G^{\widetilde{\mL}}$ has Lie algebra $\widetilde{\mL}$ and $(d \pi)_e=\widetilde{\rho}$. To see this take the normaliser of the Lie algebra $\rho(\mL)$ in $\GL(\VV)$; explicitly
\begin{equation} \label{normaliser}
G^{\widetilde{\mL}}:=\{g \in \GL(\VV): g\rho(\mL)g^{-1}=\rho(\mL) \} \ \mbox{and}  \ \pi(g)=g.
\end{equation}
\item[(iii)]$\widetilde{\mL}$ is reductive; this fact can be checked by using the structure of the centraliser with respect to the splitting of $\VV$ into 
isotypical components.
\end{itemize}
\end{rem}

\subsection{Homogeneous structures on semi-direct products.} \label{grp}
Here we consider in full generality the group counterpart of the semi-product construction for Lie algebras. The construction is based on the following data, starting from Lie groups $G_1$ and $G_2$ with closed subgroups $K_i \subseteq G_i$ for $i=1,2$. Given a smooth action $\pi:G_2 \to \Aut(G_1)$ such that 
$$\pi(G_2) \subseteq \Aut_{K_1}(G_1):=\{f \in \Aut(G_1):f(K_1) \subseteq K_1\}$$
we consider the semi-direct product $G_1 \rtimes_{\pi} G_2$. This is the Cartesian product $G_2 \times G_1$ equipped with the group structure 
$$(g_2,g_1)(g_2^{\prime},g_1^{\prime})=(g_2g_2^{\prime}, \left (\pi((g_2^{\prime})^{-1})g_1 \right )g_1^{\prime}).
$$
Then $G_1$ respectively $G_2$ are a normal subgroup respectively a subgroup in $G_1 \rtimes_{\pi} G_2$.
\begin{propn} \label{hom-semig}
Let $G_i$ be Lie groups with closed subgroups $K_i \subseteq G_i, i=1,2$. Assume 
that $\pi:G_2 \to \Aut(G_1)$ is a smooth action such that 
$\pi(G_2) \subseteq \Aut_{K_1}(G_1)$ and write $M_i=G_i \slash K_i, i=1,2$. 
\begin{itemize}
\item[(i)]The map $\varphi : (G_1 \rtimes_{\pi} G_2) \slash K \to M_2 \times M_1$ given by 
\begin{equation*}
\varphi((g_2,g_1)K):=(g_2K_2, \left ( \pi(g_2)g_1 \right ) K_1)
\end{equation*}
is a smooth diffeomophism.
\item[(ii)] The semidirect product $G_1 \rtimes_{\pi} G_2$ acts transitively on 
$M_2 \times M_1$ with isotropy
$K:=K_2 \times K_1$. Explicitly, 
\begin{equation*}
(g_2,g_1)(g_2^{\prime}K_2, g_1^{\prime}K_1)=(g_2g_2^{\prime}K_2,\pi(g_2)(g_1g_1^{\prime})K_1).
\end{equation*} 
\end{itemize}
\end{propn}
\begin{proof}
(i) 
Due to the assumption $\pi(G_2)\subseteq \Aut_{K_1}(G_1)$ the map $\varphi$ is well defined. Its inverse is given by $(g_2K_2,g_1K_1) \mapsto (g_2K_2,\left ( \pi(g_2^{-1})g_1 \right )K_1)$. \\
(ii) follows from (i) by transporting the left action by multiplication of $G_1 \rtimes_{\pi} G_2$ on the quotient 
$(G_1 \rtimes_{\pi} G_2) \slash K$ to $M_2 \times M_1$ via $\varphi$.
\end{proof}
Below we discuss several instances of this general construction which are crucial for this work.

If 
$G$ is a Lie group with Lie algebra $\mL$ and $\pi:G \to \GL(\VV)$ is a linear representation with $(d\pi)_e=\rho$ one can define the semi-direct product Lie group $\VV \rtimes_{\pi} G$. This is the cartesian product $G \times \VV$ equipped with the group multiplication 
\begin{equation} \label{pp-rod1}
(g_1,v_1)(g_2,v_2)=(g_1g_2, (\pi(g_2^{-1})v_1)+v_2).
\end{equation} 
Here we identify $G$ with the subgroup $G \times \{0 \}$ and $\VV$ with the normal subgroup $\{e\} \times \VV$.
Clearly the Lie algebra of $\VV \rtimes_{\pi} G$ is $V \rtimes_{\rho} \mL$.
\begin{exo} \label{aff-int}
Let $A(V)=V \rtimes \GL(V)$ be the affine group. For convenience we record that 
$(f_1,v_1)(f_2,v_2)=(f_1 \circ f_2,f_2^{-1}(v_1)+v_2)$
in $A(V)$. The affine group acts transitively on the manifold $V$ via 
\begin{equation} \label{aff-act}
A(V) \times V \to V, (f_1,v_1)v:=f_1(v+v_1).
\end{equation}
In this way we can regard $V$ as the homogeneous manifold $A(V) \slash GL(V)$ by means of the evaluation map $A(V) \slash \GL(V) \to V, \ (f,v)\GL(V) \mapsto f(v)$; note that the latter map is equivariant with respect to the action \eqref{aff-act}. 
\end{exo}
Assume that $H$ is a closed subgroup of $G$ and consider 
the product manifold $M=(G \slash H) \times V$. This is diffeomorphic, by the remark above, to the homogeneous manifold $(G \times A(V)) \slash (H \times GL(V))$. 
We now explain how the representation $\pi$ gives rise to a second homogeneous 
structure on $M$ as follows. First we intepret $V \rtimes_{\pi} G$  as a closed subgroup of the product group $G \times A(V)$. Indeed 
$$\varepsilon:V \rtimes_{\pi} G \to G \times A(V), \ (g,v) \mapsto 
(g,(\pi(g),v)) $$
is easily seen to be an injective Lie group morphism. We denote by 
$G_{(\pi,V)}$ respectively $H_{(\pi,V)}$ the images of $V \rtimes_{\pi} G $ respectively $H$ under $\varepsilon$.

Secondly we have a natural diffeomorphism 
$$f:(V \rtimes_{\pi} G) \slash H \to M:=(G \slash H) \times V, \ 
(g,v)H \mapsto (gH, \pi(g)v).$$
Transporting, via $f$, the canonical left action of $V \rtimes_{\pi} G$ 
onto $M$ we obtain a transitive left action of $V \rtimes_{\pi} G $ on 
$M$ explicitly given by 
\begin{equation} \label{act-prod}
(g_1,v_1)(g_2H,v_2)=(g_1g_2H,\pi(g_1)(v_1+v_2)).
\end{equation}
Clearly the stabiliser of a point with respect to this action is isomorphic to $H$. Moreover, using \eqref{aff-act}, it follows that this action is the restriction to 
$G_{(\pi,V)}$ of the product action of $G\times A(V)$ on $M$. Summarizing, we have shown that the canonical inclusion 
$$ G_{(\pi,V)} \slash H_{{(\pi,V)}} \hookrightarrow (G \times A(V)) \slash (H \times \GL(V))
$$
is a diffeomorphism. The product manifold $M$ is, in this paper, the prototypical example of manifold admitting different homogeneous structures. Now we remark 
that depending on the structure of the representation $\pi$ the product manifold 
can be given many more {\it{different}} homogeneous structures, which are constructed as follows. The procedure simply consists in enlarging the representation $\pi:G \to \GL(V)$ by its centraliser. 

Indeed consider the centraliser 
\begin{equation} \label{cent-grp-d}
C(G,\pi,V):=
\{g \in \GL(V):g\pi(h)g^{-1}=\pi(h) \ \mbox{for all} \ h \in G\}.
\end{equation}
 Let $K \subseteq 
C(G,\pi,V)$ be a closed subgroup and consider the representation 
$$ \pi^K:G \times K \to \GL(V), (g,k) \to k\pi(g). $$
Note $\pi^K$ is effective if and only if the center $Z(G)=\{e\}$. By the previous discussion 
we have 
$$ (V \rtimes_{\pi^K} (G \times K))\slash (H \times K)=V \times (G \slash H).
$$
The explicit description of the transitive action of $V \rtimes_{\pi^K} (G \times K)$ follows directly from \eqref{act-prod}.
\begin{rem} 
We note the following:
\begin{itemize}
\item[(i)] In this set-up we have trivially enlarged the homogeneous structure on 
$G \slash H$ to $(G \times K) \slash (H \times K)$.
\item[(ii)] $V \times_{\pi} G$ is a normal subgroup in $V \rtimes_{\pi^K} (G \times K)$.
\end{itemize}
\end{rem}
In order to generalise the above set-up we make the following example.
\begin{exo} \label{flat-2}
Let $K \subseteq \mathrm{GL}(V)$ be a closed subgroup. Then $V \rtimes K$ acts transitively 
on $V$, via \eqref{aff-act}, with stabiliser $K$.
\end{exo}
In view of this example, the idea in the rest of the paper is to view the semidirect product $V \rtimes_{\pi^K} (G \times K)$ as $(V \rtimes K) \rtimes G$ with $G$ acting by automorphisms on the Lie group $V \rtimes K$. 

To conclude this section we explain the Riemannian counterpart of the considerations above. 
\begin{propn}\label{R-sem}
Let $G$ be a Lie group and let $H$ be a closed subgroup of $G$. Assume that $\pi:G \to \GL(V)$ is a linear representation on a real vector space $V$. Indicating with $p:G \to M:=G \slash H$ the canonical projection,  
the following hold
\begin{itemize}
\item[(i)] if $\mathbf{g}_M$ is a $G$-invariant metric on $M$ and $g_V$ is a linear scalar product on $V$ such that 
$\pi(H) \subseteq \mathrm{O}(V,g_V)$
the tensor 
\begin{equation} \label{met-ab}
(g,v) \mapsto (p^{\star}\mathbf{g}_{M})_g+(\pi(g^{-1}))^{\star}g_V
\end{equation}
defined on $G \times V$ projects onto a $V \rtimes _{\pi}G$-invariant metric on $M \times V$
\item[(ii)] any closed subgroup $K$ in 
\begin{equation} \label{met-cent}
{}^{g_V} C(G,\pi,V)):=C(G,\pi,V) \cap \mathrm{O}(V,g_V)
\end{equation}
 acts by isometries on $\mathbf{g}$
\item[(iii)] the metric $\mathbf{g}$ is invariant under the action of $V \rtimes_{\pi} (G \times {}^{g_V} C(G,\pi,V))$.
\end{itemize}
\end{propn}
\begin{proof}
(i) Start with a Riemannian metric $\mathbf{g}$ on $M \times V$ which is invariant under the left action $g \in V \rtimes _{\pi}G \mapsto l_g$ on $M \times V$ given in \eqref{act-prod}. Assume that at $m:=(o,0)$ the metric splits as $\mathbf{g}_m=(\mathbf{g}_M)_{o} \oplus g_V$, with respect to 
$T_m(M \times V)=T_oM \oplus V$; here $(\mathbf{g}_M)_{o}$ respectively $h_V$ are inner products on $T_oM$ respectively 
$V$. By \eqref{act-prod} these inner products must be invariant under the action of $H$ on $T_{o}M$ respectively $V$ given by the isotropy representation respectively the restriction $\pi_{\vert H}$; in particular $(g_M)_o$ induces a $G$ 
invariant metric $\mathbf{g}_M$ on $M$. 

Having $\mathbf{g}$ left invariant means $\mathbf{g}_{gm}=\mathbf{g}_m((d l_{g^{-1}})_{gm} \cdot, (d l_{g^{-1}})_{gm} \cdot)$ for all $g$ in $G \times V$. Now take $g=(g_1,v_1)$ so that $g^{-1}=(g_1^{-1}, -\pi(g_1)v_1)$ and indicate the left action of $G$ on $M$ 
with $g_1 \in G \mapsto L_{g_1}$. Then 
\begin{equation*}
\begin{split}
& l_{g^{-1}}(g_2H,v_2)=(g_1^{-1}g_2H,-v_1+\pi(g_1^{-1})v_2)\\
& (d l_{g^{-1}})_{gm}(A,w)=((d L_{g^{-1}})_{g_1o}A, \pi(g_1^{-1})w)
\end{split}
\end{equation*} 
for all $(A,w) \in T_{g_1o}M \times V$. These facts together with the invariance of $\mathbf{g}$ ensure that the latter is the projection onto $M \times V$ of the tensor in \eqref{met-ab}. The statement in (i) is thus fully proved.\\
(ii) An element $k \in K$ acts on $M \times V$ according to $k(gH,v)=(gH,kv)$. The claim follows from the definition in \eqref{met-ab} by observing that having $k$ in the centraliser forces  
$g_V ( \pi(g^{-1}) \circ k \cdot, \pi(g^{-1}) \circ k \cdot)=g_V (k \circ  \pi(g^{-1}) \cdot, k\circ \pi(g^{-1})\cdot )=
g_V ( \pi(g^{-1}) \cdot, \pi(g^{-1}) \cdot)$ since $k$ is an orthogonal transformation with respect to $h_V$.\\
(iii) follows from (ii) and having $\mathbf{g}$ invariant under $V \rtimes _{\pi}G$.
\end{proof}

\section{Riemannian $3$-symmetric spaces: definitions and propaganda. } \label{riem3}
\subsection{Elements of quasi-K\"ahler geometry} \label{el-qK}
Throughout this paper whenever $\D$ is a linear connection on a manifold $M$ we denote with 
$\tau^{\D}$ respectively $R^{\D}$ the torsion respectively the curvature tensor of $D$. Explicitly $\tau^{\D}(X,Y)=\D_X\!Y-\D_Y\!X-[X,Y]$ and $R^{\D}(X,Y)=\D^2_{Y,X}-\D^2_{X,Y}$ whenever $X,Y \in \Gamma(TM)$.
 
Let $(M^{2m},g,J)$ be almost-Hermitian in the sense that $g$ is a Riemannian metric on $M$ and $J$ is an almost complex structure compatible with $g$, that is $g(J \cdot, J\cdot)=g$. The intrinsic torsion tensor $\eta^g:=
\frac{1}{2}(\nabla^gJ)J$, where $\nabla^g$ is the Levi-Civita connection of $g$, satisfies 
$\eta_U^g \in \mathfrak{u}^{\perp}(TM,g,J)$ and 
enters the following 
\begin{defn} \label{d-qK}
The almost-Hermitian manifold $(M,g,J)$ is quasi-K\"ahler provided that 
$$ \eta_{JU}^g=\eta_U^g \circ J \ \mathrm{for \ all} \ U \in TM.
$$
\end{defn}
A useful way to describe this type of structure is as follows. The space of forms of real type $(3,0)+(0,3)$ with respect to the almost complex structure $J$ will be denoted with 
$\lambda^3_JM:=\{\alpha \in \Lambda^3M : \alpha(J\cdot,J\cdot, \cdot)=-\alpha\}$. Also recall that the Nijenhuis tensor of $J$ is defined according to 
\begin{equation*}
N^J(X,Y)=[X,Y]-[JX,JY]+J([JX,Y]+[X,JY])
\end{equation*}  
for vector fields $X,Y$ on $M$.
Below we gather a few well-known facts from almost-Hermitian geometry, see e.g. \cite{PGaud} for proofs.
\begin{propn}\label{des-qK}Let $(M,g,J)$ be almost Hermitian. Then 
\begin{itemize}
\item[(i)] the following are equivalent \begin{itemize}
\item[(a)] $(g,J)$ is quasi-K\"ahler 
\item[(b)] the K\"ahler form $\omega:=g(J \cdot,\cdot)$ satisfies $\di\omega \in \lambda^3_JM$
\item[(c)] there exists a linear connection $\D$ on $TM$ with $\D\!g=\D\!J=0$ and torsion tensor $\tau^{\D}$ satisfying 
$\tau^{\D}(U, \cdot) \in \gl_J^{\perp}TM$ for all $U \in TM$
\end{itemize}
\item[(ii)] if $(g,J)$ is quasi-K\"ahler a linear connection $\D$ as in (c) is unique and given by 
$\D=\nabla^g+\frac{1}{2}(\nabla^gJ)J$. In addition the torsion tensor of $D$ reads $\tau^{\D}=-\frac{1}{4}N^J$.
\end{itemize}
\end{propn}
When $(M,g,J)$ is quasi-K\"ahler the connection $\nabla^g+\frac{1}{2}(\nabla^gJ)J$ coincides with the Chern 
connection, which is the unique Hermitian connection whose torsion has vanishing $(1,1)$ component. We will simply write $\tau$ for the torsion of $\D$, as it only depends on the almost complex structure. There are two distinguished sub-classes of quasi-K\"ahler structures. The first is when $(g,J)$ is almost-K\"ahler, that is $\di\!\omega=0$; in this case the metric $g$ picks up Riemannian features of symplectic geometry. Note that when $m=2$ any quasi-K\"ahler structure is almost-K\"ahler; this is due to having 
$\lambda^3_JM=0$ when $m=2$ and (b) in Proposition \ref{des-qK}. The second subclass consists in nearly-K\"ahler structures when $(g,J)$ is required to satisfy $\eta^g_UU=0, U \in TM$. The explicit classification of homogeneous nearly-K\"ahler structures is contained in \cite{GoCa,bu}.

In what follows we look at the class of quasi-K\"ahler structures with parallel torsion with respect to the Chern connection, that is  
$\D\!N^J=0$. Since $\eta^g$ is determined by $N^J$ and the metric according to 
\begin{equation} \label{it-expr-n}
2g(\eta^g_{v_1}v_2,v_3)=g(\tau(v_1,v_2),v_3)-g(\tau(v_2,v_3),v_1)+g(\tau(v_3,v_1),v_2)
\end{equation}
it follows this is equivalent to having $\D\!\eta^g=0$.

This class generalises the class of Riemannian $3$-symmetric spaces, see Definition \ref{R3s}, but 
is not limited to those, in view of the well known twistor examples in \cite{Ivanov}. We briefly explain that in this set-up the Riemann curvature tensor of $(g,J)$ has special algebraic symmetries, known as Gray's curvature conditions. An almost-Hermitian structure $(g,J)$ is said to belong to the class $\mathcal{AH}_k$ where $1 \leq k \leq 3$ provided its Riemann curvature tensor satisfies the condition $(G_k)$. These curvature requirements are as follows
\begin{itemize}
\item[$(G_1)$] $R^g(J\cdot, J\cdot, \cdot, \cdot)=R^g$
\item[$(G_2)$] $[J,R^g]=0$
\item[$(G_3)$] $R^g(J \cdot, J\cdot, J \cdot, J\cdot)=R^g$.
\end{itemize}
The algebraic action in the second condition is defined according to 
\begin{equation*}
\begin{split}
&[J,R^g](v_1,v_2,v_3,v_4)\\
=&R^g(Jv_1,v_2,v_3,v_4)+R^g(v_1,Jv_2,v_3,v_4)+R^g(v_1,v_2,Jv_3,v_4)+
R^g(v_1,v_2,v_3,Jv_4).
\end{split}
\end{equation*}
We prove below that quasi-K\"ahler structures with parallel torsion satisfy condition $(G_2)$; this also helps to set up some notation that will be used later on in the paper. First we have the comparaison formula 
\begin{equation*}
R^{\D}(U_1,U_2)=R^g(U_1,U_2)+[\eta_{U_1},\eta_{U_2}]-\eta_{\tau(U_1,U_2)}
\end{equation*}
which is due to $\D\!\tau=0$. 
\begin{propn} \label{G2}
Assume that $(M,g,J)$ is quasi-K\"ahler with parallel torsion, that is $\D\! \eta^g=0$. The following hold
\begin{itemize}
\item[(i)] we have $R^g(JX,JY,JZ,JU)=R^{g}(X,Y,Z,U)$
\item[(ii)] the tensor $(X,Y,Z,U) \mapsto g(\eta_{\tau(X,Y)}Z,U)$ is symmetric in pairs
\item[(iii)] $[J,R^g]=0$. 
\end{itemize}
\end{propn}
\begin{proof}
(i)\&(ii) are proved at the same time, using the Hermitian symmetries of $\eta$ respectively $\tau$ which are 
\begin{equation} \label{H-sym}
\eta_{JX}=\eta_XJ=-J\eta_X \ \mathrm{respectively} \ \tau(JX, \cdot)=\tau(X,J\cdot)=-J\tau(X,\cdot).
\end{equation}
Because $\D\!J=0$ we have $[R^{\D}(U_1,U_2),J]=0$; since $[\eta_{U_1},\eta_{U_2}]$ commutes with $J$ we end up with 
\begin{equation} \label{star}
R^g(X,Y,JZ,U)+R^g(X,Y,Z,JU)=2g(\eta_{\tau(X,Y)}JZ,U)
\end{equation}
Operating $U \mapsto JU$ yields 
$R^g(X,Y,JZ,JU)-R^g(X,Y,Z,U)=-2g(\eta_{\tau(X,Y)}Z,U).$
The variable change $(X,Y,Z,U) \mapsto (Z,U,X,Y)$ shows that 
\begin{equation*}
R^g(Z,U,JX,JY)-R^g(Z,U,X,Y)=-2g(\eta_{\tau(Z,U)}X,Y).
\end{equation*}
Because the Riemann curvature tensor is symmetric in pairs we find, after substraction 
\begin{equation*}
R^g(X,Y,JZ,JU)-R^g(JX,JY,Z,U)=2(g(\eta_{\tau(Z,U)}X,Y)-g(\eta_{\tau(X,Y)}Z,U)).
\end{equation*}
Performing $(X,Y) \mapsto (JX,JY)$ and using \eqref{H-sym} finally entails 
$$ R^g(JX,JY,JZ,JU)-R^g(X,Y,Z,U)=2(g(\eta_{\tau(X,Y)}Z,U)-g(\eta_{\tau(Z,U)}X,Y)).
$$
Since the left hand side is symmetric in pairs whilst the right hand side is skew-symmetric in pairs they must vanish 
independently and the claims are proved.\\
(ii) From \eqref{star} we also get 
\begin{equation*}
R^g(JX,Y,Z,U)+R^g(X,JY,Z,U)=R^g(Z,U,JX,Y)+R^g(Z,U,X,JY)=2g(\eta_{\tau(Z,U)}JX,Y).
\end{equation*}
Thus we obtain 
\begin{equation*}
\begin{split}
\frac{1}{2}[J,R^g](X,Y,Z,U)=&g(\eta_{\tau(X,Y)}JZ,U)+g(\eta_{\tau(Z,U)}JX,Y)\\
=&g(\eta_{\tau(X,Y)}JZ,U)+g(\eta_{\tau(JX,Y)}Z,U)
\end{split}
\end{equation*}
 where 
in the last equality we have used the symmetry property in (ii). The vanishing of $[J,R^g]$ follows now from \eqref{H-sym}.
\end{proof}
When assuming that $(g,J)$ is almost-K\"ahler we have $g(\tau(X,Y),Z)=-g(\eta_ZX,Y)$ which easily entails that the symmetry property in (ii) of Proposition \ref{G2} is automatic.

Some additional motivation to study the almost-K\"ahler case is as follows. Almost-K\"ahler structures $(g,J)$ with Riemann curvature satisfying $(G_2)$ respectively $(G_3)$ are said to belong to the class $\mathcal{AK}_2$ respectively $\mathcal{AK}_3$. 

Concerning these classes, the relevance of Riemannian $3$-symmetric spaces of almost-K\"ahler type is thus 
to provide examples of complete metrics in the class $\mathcal{AK}_2$. In dimension $4$ structures 
in the class $\mathcal{AK}_2$ are symmetric by work in \cite{Ap1}; still in dimension $4$ the class $\mathcal{AK}_3$ has been fully described in \cite{Ap3}. In higher dimensions other examples of metrics 
in the class $\mathcal{AK}_3$ have been constructed in \cite{Chi-Na} but those are not necessarily complete.
\subsection{Intrinsic definition of $3$-symmetric spaces} \label{intr-defn}
Here we record for the reader's convenience the connection between  3-symmetric spaces, which are geometric objects, and their algebraic description, which are the primary object of study. We begin with some preliminary definitions. A linear connection $\D$ on a manifold 
$M$ is called Ambrose Singer if 
\begin{equation*}
\D\!\tau^{\D}=0, \D\!R^{\D}=0.
\end{equation*}

\begin{defn} \label{dd1}
Let $(M^{2m},J), m \geq 2$ be almost complex. It is called $3$-symmetric provided it admits an Ambrose-Singer connection $\D$ such that 
\begin{equation} \label{tor-def}
\D\!J=0 \ \mbox{and} \ \tau^{\D}(U, \cdot) \in \gl_J^{\perp}(TM), U \in TM.
\end{equation}
\end{defn}
This class of spaces is more general than we strictly need in this work. However examining issues such as uniqueness for $\D$ above will shed light on the structure of the various automorphism groups we will study in the Riemannian set-up. Below we gather a few observations in that direction before proving our first main results.
\begin{lemma}  \label{aut-gr}
The following hold
\begin{itemize}
\item[(i)] if $(M,J)$ is almost complex, a connection $\D$ as in \eqref{tor-def} satisfies 
$$\tau^{\D}=-\frac{1}{4}N^J$$
\item[(ii)] any other linear connection satisfying \eqref{tor-def} is of the form $\D+\xi$ where $\xi_X \in \gl_J(TM)$ and $\xi_XY=\xi_YX$ for all $X,Y \in TM$.
\end{itemize}
\end{lemma}
\proof 
(i) Since $[X,Y]=\D_X\!Y-\D_Y\!X-\tau(X,Y)$ the Nijenhuis tensor reads, in terms of $\D$ and its torsion  
\begin{equation*}
\begin{split}
N^J(X,Y)=&-(\D_{JX}\!J)Y+(\D_{JY}\!J)X+J(D_XJ)Y-J(\D_Y\!J)X\\
&-\tau(X,Y)+\tau(JX,JY)-J\left (\tau(JX,Y)+\tau(X,JY) \right ).
\end{split}
\end{equation*}
Since $\D\!J=0$ and $\tau(JX,JY)=-\tau(X,Y),\tau(X,JY)=\tau(JX,Y)=-J\tau(X,Y)$ it follows that $N^J=-4\tau^{\D}$. 
Note that this statement is the metric-free version of (ii) in Proposition \ref{des-qK}.\\
(ii) By (i) any other connection satisfying \eqref{tor-def} has the same torsion as $\D$ hence 
it is of the form $\D+\xi$ where $\xi_XY=\xi_YX$. The $J$-invariance property for $\xi$ follows from requiring $J$ be parallel with respect to $\D+\xi$. \\
\endproof
In light of part (i) above, the torsion tensor of connections involved in Definition \ref{dd1} does not depend 
on the choice of connection and thus will be simply denoted by $\tau$ in what follows. 

Next we examine the stabiliser of the Nijenhuis tensor within 
$\gl_J(TM)$, that is the Lie algebra $\mathfrak{t}:=\mathfrak{stab}_{\gl_J(TM)}\tau=\{ F\in \gl_{J}TM : F\tau(x,y)-\tau(Fx,y)-\tau(x,Fy)=0\}$. We consider the second symmetric prolongation of $\mathfrak{t}$ which is given by 
$$\mathfrak{t}^{(2)}:=\{\xi :TM \to \mathfrak{t} : \xi_xy=\xi_yx\}.$$
Due to the complex type of $\tau$ this can be determined explicitly in terms of the distributions 
$$\mathscr{K}:=\ker(v \mapsto N^J(v,\cdot)) \ \mathrm{ respectively} \ \mathscr{N}:=\spa \{N^J(v,w) : v,w \in TM\}.$$
\begin{lemma} \label{proL}
We have 
$$\mathfrak{t}^{(2)}=\{\xi:TM \to \gl_JTM : \xi_xy=\xi_yx \ \mathrm{and} \  \xi_{TM}\mathscr{N}=0, \ \xi_xy \in \mathscr{K} \ \mbox{for all} \ x,y \in TM \}$$
at every point of $M$.
\end{lemma}
\begin{proof}
See \cite{Kr0}, proof of Theorem 2.1,(i). 
\end{proof}
These preliminaries enable isolating conditions, intrinsically phrased in terms of $J$, under which the connection $\D$ is unique. When $(M,J)$ is $3$-symmetric or more generally when only $\D\!\tau=0$ is required the distributions 
$\mathscr{K} \ \mathrm{and} \ \mathscr{N}$ are parallel with respect to $\D$ and hence have constant rank.

We shall say that the torsion tensor or equivalently the Nijenhuis tensor is non-degenerate 
respectively has maximal rank provided that $\mathscr{K}=0$ respectively $\mathscr{N}=TM$ at some (hence every ) 
point in $M$.
\begin{thm}\label{unq} Let $(M,J)$ be $3$-symmetric such that the Nijenhuis tensor is either non-degenerate 
or has maximal rank. Then
\begin{itemize}
\item[(i)] the Ambrose Singer connection $\D$ from Definition \ref{dd1} is unique 
\item[(ii)] we have  
\begin{equation} \label{JvsD}
 \Aut^0(M,J)=\Aut^0(M,\D).
\end{equation}
\end{itemize}
\end{thm}
\begin{proof}
(i) Assume that $\D+\xi$ is another Ambrose-Singer connection satisfying \eqref{tor-def}. Then $\xi_xy=\xi_yx$ and 
$\xi_x \in \gl_J(TM)$ by part (ii) in Lemma \ref{aut-gr}. Since $\D+\xi$ and $\D$ have the same torsion tensor, having the latter parallel with respect to both connections amounts to saying that 
$ \xi \in \mathfrak{t}^{(2)}.$
Because $\tau^{\D}$ is proportional to $N^J$, Lemma \ref{proL} ensures that $\xi_{TM}\mathscr{N}=0$ and $\mathrm{Im} \xi \subseteq \mathscr{K}$ hence the assumptions on the Nijenhuis tensor make that $\xi=0$.\\
(ii) Choose $\varphi \in \Aut(M,J)$. Then $\varphi^{\star}\D$ satisfies all the conditions of Definition \ref{dd1}. Yet such connections are unique, hence $\varphi^{\star}\D =\D$, showing that $\Aut(M,J) \subseteq \Aut(\D)$. Now consider 
$X$ in the Lie-algebra $\mathfrak{aut}(\D)$ of $\Aut(\D)$. Since $X$ preserves $\D$ it preserves its torsion tensor, that is $\li_X\!\tau=0$. Differentiating in $\tau(J \cdot, \cdot)=-J\tau(\cdot, \cdot)$ it follows that 
the endomorphism $A:=\li_X\!J$ satisfies $\tau(Ax,y)=-A\tau(x,y)$ for all $x,y$ in $TM$. Since $AJ+JA=0$ and $\tau(Jx,Jy)=-\tau(x,y)$ the left hand side is $J$-invariant whilst the right hand side is $J$-anti-invariant.  It follows that 
$\tau(Ax,y)=A\tau(x,y)=0$ that is $A \mathscr{N}=0$ and $\mathrm{Im} A \subseteq \mathscr{K}$. The assumptions made on the Nijenhuis tensor then yield $A=0$, so $\mathfrak{aut}(M,\D) \subseteq \mathfrak{aut}(M,J)$ and the inclusion $\Aut^0(M,\D) \subseteq \Aut^0(M,J)$ is also proved. This shows the equality in \eqref{JvsD}.
\end{proof}
Examples of $3$-symmetric spaces with non-degenerate Nijenhuis tensor are given by the so-called class III and IV spaces, see Proposition \ref{int-s}, (ii), in section \ref{cpx-symp}; note however that those spaces do not have maximal rank. The equality in \eqref{JvsD} provides a quick way to determine explicitly, for these spaces, the group $\Aut(M,J)$.
\begin{rem} \label{p-tor}
Clearly Theorem \ref{unq} holds in the more general set-up when the assumption of having $\D$ an Ambrose-Singer connection is dropped. Indeed it is enough to require $\D$ satisfy \eqref{tor-def} and $\D\!\tau=0$. Other classes 
of connections with torsion which are uniquely determined by their holonomy and torsion type can be found in \cite{Na-prol}. Note however that no assumption on the holonomy representation of $\D$ is made in Theorem \ref{unq}. 
\end{rem}
\subsection{The Riemannian set-up} \label{su-R}
To move on to the Riemannian counterpart of Definition \ref{dd1} we recall that a connection $\D$ such that 
$\D\!J=0$ is called metrisable provided there exists a Riemannian metric $g$ such that $\D\!g=0$ and $g(J \cdot,J \cdot)=g$. As it is well known this is phrased intrinsically as follows.
\begin{lemma} \label{hol-t}
Let $(M,J)$ be almost complex with $M$ simply connected and assume that $\D\!$ is a linear connection on $M$ with $\D\!J=0$. Then $\D$ is metrisable iff its holonomy group is compact in $\GL_J(TM)$ at some(and hence every) point in $M$.
\end{lemma}
\begin{proof} 
Pick $x \in M$ and let $\bbV=T_xM$. If the connection $D$ is metrisable from $\D\!g=0$ and $\D\!J=0$ we get $\Hol(\D)\subset \mathrm{U}(\bbV,g,J)$ and hence $\Hol(\D)$ is compact. Conversely, when $\Hol(\D)$ is compact, it preserves a metric at the point $x$. By parallel transport this is extended to a Riemannian metric $g$ on $M$ such that $\D\!g=0$. Since $\D\!J=0$, via $g \mapsto g+g(J \cdot, \cdot)$ we can assume that $J$ is orthogonal with respect to $g$. 
\end{proof}
In light of this fact the following definition makes sense. 
\begin{defn} \label{R3s}
Let $(M,J,\D)$ be $3$-symmetric, where $M$ is simply connected. It is called Riemannian $3$-symmetric provided 
$\mathrm{Hol}(\D)$ is compact in $\GL_J(TM)$. If this is the case, the space of compatible metrics, that is metrics with $g(J \cdot, J\cdot)=g$ and $\D\!g=0$, will be indicated with $\me(M,J,\D)$.
\end{defn}
When $J$ is integrable, that is $N^J=0$, we simply recover Hermitian symmetric spaces.
In the Riemannian set-up the counterpart of the previous results for $3$-symmetric spaces is as follows.
\begin{propn} \label{p-311}
Let $(M,J,\D)$ be Riemannian $3$-symmetric and let $g \in \me(M,J,\D)$. Then 
\begin{itemize}
\item[(i)] $(g,J)$ is quasi-K\"ahler and $\D=\nabla^g+\frac{1}{2}(\nabla^gJ)J$ where $\nabla^g$ denotes the Levi-Civita connection of $g$
\item[(ii)] we have 
\begin{equation} \label{aut-grD}
\Aut(M,g,J):=\Aut(M,g) \cap \Aut(M,J) \subseteq \Aut(M,\D)
\end{equation}
\item[(iii)] if $\me(M,J,\D)$ contains a locally irreducible almost K\"ahler metric $g$ then 
either $N^J=0$ or $N^J$ is non-degenerate. In the latter case we moreover have that $\D$ is the unique connection satisfying \eqref{tor-def} and $\Aut^0(M,J)=\Aut^0(M,\D)$.
\end{itemize}
\end{propn}
\begin{proof}
(i) follows from Proposition \ref{des-qK}, (i).\\
(ii) follows from having $\D$ determined as in (i).\\
(iii) From $g(\tau(x,y),z)=-g(\eta_zx,y)$ we get that $\mathscr{K}=\{x \in TM : \eta_{TM}x=0\}$. It follows 
that $\D$ and $\nabla^g$ coincide on $\mathscr{K}$. As the latter is parallel with respect to $\D$ it thus must be parallel with respect to the Levi-Civita connection as well. Because $g$ is locally irreducible we either have $\mathscr{K}=0$ hence 
$N^J$ is non-degenerate or $\mathscr{K}=TM$, when $N^J=0$. The last part of the claim follows from Theorem \ref{unq}.
\end{proof}
Summarising, we record the various different descriptions of Riemannian 3-symmetric spaces possible, viz.;
\begin{itemize}
\item[(i)] The classical definition as given at the start of this paper. 
\item[(ii)] $(M, g, J) $ is a quasi-K\"ahler  manifold whose Chern connection $\D$ is Ambrose--Singer 
\item[(iii)] A Riemannian 3-symmetric space $(M,J,\D)$ in the sense of Definition \ref{R3s}.
\end{itemize}
We refer to \cite{bu} for further details explaining the equivalence between the first two definitions. The equivalence of (ii) and (iii) is contained in Proposition \ref{p-311}.
\begin{rem}
 In this work we choose to not use the classical definition (given first in the above list).  Attempts to obtain classifications by firstly obtaining a classification of transitive groups $G$ acting on a fixed Riemannian 3-symmetric space is too cumbersome, as is  evident from  Theorem \ref{t1.5}.  The second definition has the disadvantage of not separating the differential-geometric level from the Riemannian (i.e. a metric is bound up in the definition). As such we will    take the third definition as our definition of a (Riemannian) 3-symmetric space.    The advantage of this approach is that the metric is not intrinsic to the definition of a 3-symmetric space, and as such it works for every metric simultaneously. 
 \end{rem}
\subsection{Metric uniqueness and foliation aspects} \label{u-fol}
Next we wish to discuss, for a given Riemannian manifold $(M,g)$ uniqueness for almost complex structures 
$J$ compatible with $g$, such that $(g,J)$ is Riemannian $3$-symmetric in the sense that the Chern connection is Ambrose-Singer. 

Uniqueness with respect to the metric is not true on some Hermitian symmetric spaces, 
by work in \cite{MuNa}. In what follows we generalise the foliation set-up in that work in order to include 
spaces of type III and IV and also to explain the result therein in the present framework.
\begin{rem}
The motivation for addressing this question stems from the need to decide, later on in the paper, 
which are the classes of Riemannian $3$-symmetric spaces such that the inclusion  
\begin{equation} \label{incl-14}
\Aut(M,g,J) \subseteq \Aut(M,g)
\end{equation}
is strict. Indeed, if $\varphi \in \Aut(M,g)$ then $(g,\varphi_{\star}J)$ is again $3$-symmetric. In other words we have an injection 
$$ \Aut(M,g) \slash \Aut(M,g,J) \to \{J : (g,J) \ \mathrm{is \ Riemannian \ 3-symmetric }\}.
$$
For some symmetric spaces, \eqref{incl-14} is strict, see \cite{MuNa}; as mentioned in the introduction, this is the first example of how a fixed Riemannian 3-symmetric space can admit different presentations as a homogeneous Riemannian manifold and was the starting point of our investigations. Using very different techniques we will show in Theorem \ref{split-iso} that equality however holds, provided $g$ is not locally symmetric. 
\end{rem}
The following is a general recipe for constructing quasi-K\"ahler structures from K\"ahler manifolds with split tangent bundle in the sense made precised below. We recall that a sub-bundle $E \subseteq TM$ where $(M,I)$ is some complex manifold is called complex if $IE=E$.

Now assume that $TM=E \oplus F$ orthogonally with respect to $g$ where $IE=E$ and $IF=F$. Consider the projection $\bnabla$ of the Levi-Civita connection onto this splitting. Explicitly
$$ \bnabla=\nabla^g+\pi_E \circ \nabla^g \pi_E+\pi_F \circ \nabla^g \pi_F
$$ 
where $\pi_E:TM \to E$ and $\pi_F:TM \to F$ are the projection maps. Since $E$ and $F$ are complex and $g$-orthogonal, this connection is metric and Hermitian, $\bnabla g=0$ and $\bnabla I=0$. 
By analogy with the case of foliations we shall call this connection the Bott connection of the splitting.
\begin{propn} \label{fol-11}
Let $(M,g,I)$ be K\"ahler and equipped with a $g$-orthogonal and $I$-invariant splitting 
$$TM=E \oplus F$$
and consider the $g$-orthogonal almost complex structure $J:=-I_{\vert E}+I_{\vert F}$.
\begin{itemize}
\item[(i)] The following are equivalent 
\begin{itemize}
\item[(a)] $(g,J)$ is quasi-K\"ahler
\item[(b)] we have 
\begin{equation} \label{T-split}
(\li_{\Gamma E}I)\Gamma E \subseteq \Gamma E \ \mathrm{ and}  \ (\li_{\Gamma F}I)\Gamma F \subseteq \Gamma F
\end{equation} 
\end{itemize}
\item[(ii)] if $(g,J)$ is quasi-K\"ahler its Chern connection coincides with the Bott connection of the splitting $TM=E \oplus F$
\item[(iii)] $(g,J)$ is K\"ahler iff $E$ is parallel with respect to $\nabla^g$
\item[(iv)] $(g,J)$ is almost-K\"ahler iff both $E$ and $F$ are integrable.
\end{itemize}
\end{propn}
\begin{proof}
First we prove a few general facts related to the splitting $TM=E \oplus F$. Below we indicate generic 
section of $E$ respectively $F$ with $E_1,E_2$ respectively $F_1,F_2$. The diference $\zeta:=\nabla^g-\bnabla$ belongs to $\Lambda^1 \otimes \so(TM)$ since 
both connections preserve $g$. From the definition of the Bott connection we have 
$$\zeta_{E}E \subseteq F, \ \zeta_{E}F \subseteq E, \ \zeta_{F}E \subseteq F, \zeta_{F}F \subseteq E $$
and in addition $g(\zeta_{E_1}F_1,F_2)=
-g(F_1,\zeta_{E_1}E_2)$ respectively $g(\zeta_{F_1}E_1,F_2)=-g(E_1,\zeta_{F_1}F_2)$. Since both $E$ and $F$ are complex and $\nabla^gI=0$ we get $\zeta_U \in \gl_I(TM)$ for all $U \in TM$. Straightforward computation based on the algebraic properties of $\zeta$ reveals 
that the intrinsic torsion tensor of $(g,J)$ is determined from  
\begin{equation} \label{t-flip}
\frac{1}{2}(\nabla^gJ)J=-\zeta.
\end{equation}
Because $(g,I)$ is K\"ahler we have $(\li_XI)Y=-\nabla^g_{IY}X+I\nabla^g_YX$ for all $X,Y \in \Gamma(TM)$. It follows that 
\begin{equation} \label{hol-s}
\begin{split}
&\pi_F (\li_{E_1}I)E_2 =-\zeta_{IE_2}E_1+I\zeta_{E_2}E_1, \ \left ( (\li_{E_1}I)F_1 \right)_F=-\zeta_{IF_1}E_1+I\zeta_{F_1}E_1\\
&\pi_E (\li_{F_1}I)E_1=-\zeta_{IE_1}F_1+I\zeta_{E_1}F_1, \ \left ( (\li_{F_1}I)F_2 \right)_E=-\zeta_{IF_2}F_1+I\zeta_{F_2}F_1
\end{split}
\end{equation}
(i) 
Start from the quasi-K\"ahler condition $\eta^g_{JU_1}JU_2=-\eta^g_{U_1}U_2$ with $U_1,U_2$ in $TM$. Using \eqref{t-flip} we see this is equivalent with the pairs of requirements
\begin{equation} \label{temp22}
\zeta_{IE_1}IE_2=-\zeta_{E_1}E_2, \ \zeta_{IF_1}IF_2=-\zeta_{F_1}F_2
\end{equation}
and 
\begin{equation} \label{temp23}
\zeta_{IE_1}IF_1=\zeta_{E_1}F_1, \ \zeta_{IF_1}IE_1=\zeta_{F_1}E_1.
\end{equation}
However, an orthogonality argument based on having $\zeta \in \Lambda^1 \otimes \mathfrak{u}(TM,g,J)$ and $g(E,F)=0$ as well as $IE=E, IF=F$ 
shows that \eqref{temp22} and \eqref{temp23} are equivalent. By \eqref{hol-s} we see that \eqref{temp22} is equivalent to the condition stated in (b). 
(ii)\&(iii) are direct consequences of \eqref{t-flip}.\\
(iv) is well-known, see \cite{Apo4}.
\end{proof}
Now we explain briefly how the results from \cite{MuNa} fit in the construction in Proposition \ref{fol-11}; in fact 
the latter applies to two important general classes of foliations on K\"ahler manifolds. Recall that 
on a complex manifold $(M,I)$ a complex sub-bundle $E \subseteq TM$ is called  
holomorphic if $(\li_XI)TM \subseteq E$ whenever $X \in \Gamma(E)$. 
\begin{propn} \label{tams-int}
Let $(M,g,I)$ be K\"ahler and let $\mathscr{V}$ be tangent to a foliation with
complex leaves. 
\begin{itemize}
\item[(i)] if $\scrV$ induces a totally geodesic Riemannian foliation we have
\begin{itemize}
\item[(a)] $\mathscr{H}:=\mathscr{V}^{\perp}$ is holomorphic and $(g,J:=-I_{\vert \mathscr{V}}+I_{\mathscr{\H}})$ is quasi-K\"ahler with parallel torsion
\item[(b)] the nearly-K\"ahler structure $(g_{nk}:=\frac{1}{2}g_{\scrV}+g_{\mathscr{H}},J)$ is $3$-symmetric iff $(g,J)$ is $3$-symmetric
\item[(c)] if $\nabla^gR^g=0$ then $(g,J)$ is Riemannian $3$-symmetric.
\end{itemize}
\item[(ii)] if $\scrV$ is a Riemannian foliation which is polar in the sense that $\scrH:=\scrV^{\perp}$ is integrable then $\scrV$ is holomorphic, $\scrH$ is totally geodesic and $(g,J)$ is almost-K\"ahler
\end{itemize} 
\end{propn}
\begin{proof}
(i) In terms of the intrinsic torsion of the Bott connection the assumption on the foliation means that $\zeta_{\mathscr{V}}=0$ and $\zeta_XY+\zeta_YX=0$ for $X,Y \in \mathscr{H}$. It follows that $\zeta_{IX}IY=-\zeta_XY$ thus by using \eqref{hol-s} we see that the splitting $TM=\scrV\oplus \scrH$ satisfies the requirements in \eqref{T-split} hence $(g,J)$ is quasi-K\"ahler. According to Proposition \ref{fol-11} the Chern connection $\D$ of $(g,J)$ coincides with the Bott connection of the foliations. As $\zeta$ is parallel with respect to the Bott connection by \cite{MuNa} we get $D\eta^g=0$.\\
(ii) has been proved in \cite{MuNa}[Lemma 4.2].
\end{proof}

Instances as in (c) above have been described in \cite{MuNa}, even locally; for those the locally symmetric metric $g$ admits two distinct compatible $3$-symmetric structures namely $I$ and $J$. The set-up in (ii) occurs 
for the class of 
$3$-symmetric spaces of type $III$ which are studied in detail in section 
\ref{it-curv}.
\section{Automorphisms of order $3$} \label{o3}

\subsection{Riemannian 3-symmetric algebras} \label{o3R}We now define the main object of study in the paper, Riemannian 3-symmetric Lie algebras. We remind the reader of our convention that $\g$ always refers to such algebras. Suppose there exists $\sigma \in \Aut(\g)$ such that $\sigma^3=1$.We assume 
that $\sigma \neq 1_{\g}$ in what follows. Then $\g$ splits as 
\begin{equation} \label{can-red}
\g=\h \oplus \bbV
\end{equation}
where $\h:=\ker(\sigma-1)$ and $\bbV:=\ker(\sigma^2+\sigma+1)$. In particular 
$\h$ is a subalgebra and $\bbV$ admits a canonical almost complex structure $J$ determined from 
\begin{equation} \label{J-cons} 
\sigma_{\vert \bbV}=-\frac{1}{2}1_{\bbV}+\frac{\sqrt{3}}{2}J.
\end{equation}
We will refer to $\h$ 
as the isotropy Lie algebra and to \eqref{can-red} as the canonical reductive decomposition of $\g$. Using that $\sigma$ is a Lie algebra automorphism it is easy to see that $[\h,\bbV] \subseteq \bbV$ and that $J$ is $\ad_{\h}$-invariant. This yields a linear representation 
\begin{equation*}
\iota:\h \to \mathfrak{gl}_J(\bbV), \iota(F):=(\ad_F)_{\vert \bbV}
\end{equation*}
to be referred to as the isotropy representation. Let $\BK$ be the Killing form of $\g$. Since $\sigma$ is an automorphism 
$\sigma^{\star}\BK=\BK$ thus 
\begin{equation} \label{B-inv}
\BK(\h,\bbV)=0, \ \BK(J \cdot, J\cdot)=\BK \ \mbox{on} \ \bbV
\end{equation}
by using the structure of $\sigma$ outlined above. 
 \begin{defn} \label{3sL}
A Riemannian $3$-symmetric Lie algebra $(\g,\sigma)$ is a Lie algebra $\g$ such that 
\begin{itemize}
\item[(i)] $\sigma \in \Aut(\g)$ has order $3$
\item[(ii)] there exists a compact connected Lie group $\mathrm{H}$ with Lie algebra $\h$ together with an effective representation 
$\pi:\mathrm{H} \to \GL_{J}(\bbV)$ such that $(d\pi)_e=\iota$.
\end{itemize}
\end{defn}
We will refer to property (ii) above simply by saying that the isotropy representation $(\h, \iota)$ is {\it{compactly embedded}}; also a compact group $\mathrm{H}$ as above will be said to {\it{integrate}} the isotropy representation $\iota$.

A few well known consequences, purely at the Lie algebraic level, of having $(\h,\iota)$ compactly embedded are summarised below. The first is to have the isotropy representation {\it{metrisable}} in the sense of admitting an isotropy invariant metric.

\begin{defn} \label{metr}
An isotropy invariant metric on $(\g, \sigma)$, which by an abuse of notation is written  $g$,   is an $\Ad_{\mathrm{H}}$-invariant inner product on $\bbV$ such that 
$g(J \cdot, J \cdot)=g$.
\end{defn}

We will frequently write $(\g, \sigma, g)$ to indicate a Riemannian 3-symmetric Lie algebra endowed with an isotropy invariant metric.

Part (ii) in Definition \ref{3sL} provides an intrinsic way to pin down metrisability for the isotropy representation $(\h,\iota)$ as the following simple observation shows.
\begin{propn} \label{met-c}
Let $(\g,\sigma)$ be Riemannian $3$-symmetric. The following hold  
\begin{itemize}
\item[(i)] the isotropy representation of $\h$ is faithful and metrisable
\item[(ii)] the restriction $\BK_{\vert \h}:\h \times \h \to \mathbb{R}$ is negative definite.
\end{itemize}
\end{propn}
\begin{proof}
(i) That $\iota$ is faithful follows by exponentiating and using that $\pi$ is effective. The second part of the claim is proved in the standard way by averaging any $J$-invariant innner product on $\bbV$ over the compact group $\mathrm{H}$.\\ 
(ii) We prove in fact a little more, namely that the claim follows from having $(\h,\iota)$ metrisable. This is a standard fact, see e.g. \cite{KN}. We indicate with $\BK_{\h}$ the Killing form of the Lie algebra $\h$. Then $\BK(F,G)=\BK_{\h}(F,G)+t_{\iota}(F,G)$ whenever 
$F,G \in \h$. 
 Since  $\h$ is identified with 
a subalgebra of $\mathfrak{so}(\mathbb{V})$ we have $\BK_{\h}(F,F) \leq 0$ for all 
$F \in \h$; equality holds if and only if $F \in \mathfrak{z}(\h)$, where the latter indicates the centre of $\h$. At the same time the trace form $t_{\iota}$ is negative definite since 
$\iota(\h) \subseteq \mathfrak{so}(\bbV,g)$ for some $g \in \mathrm{Met}_{\mathrm{H}}\bbV$ and the isotropy representation is faithful. The result follows. 
\end{proof}
 
\begin{defn} \label{met-not1}
We denote by $\me_{\mathrm{H}}(\bbV)$ the space of isotropy invariant metrics on $\bbV$. 
\end{defn}

At this stage we comment that the  main idea in this work is to determine the structure of 
$(\g,\sigma,\me_{H}(\bbV))$ simultaneously. In particular, fixing one metric will allow us to compute $\underline{\g}$ and from here determine the product structure of $M$. Knowing $(\underline{\g}, \sigma, g)$ will then be shown, in many cases, to determine all possible presentations of $M$ as a Riemannian 3-symmetric space.  The fact that  $\BK_{\vert \h}$ is non-degenerate will play a key role in the computation of the radical of $\g$(see Theorem \ref{rad11} in section \ref{rad}); this fact also allows to construct in a canonical way complements for ideals in $\h$.

\begin{lemma}\label{pass2local}  For $(M,g, \D,J)$ a Riemannian 3-symmetric space, choose $p\in M$ and let $T_pM = \bbV$. Letting $\uh = \mathrm{span}\{ R^{\D}(v_1,v_2) : v_i\in \bbV\}$ and $\ug = \uh \oplus \bbV$ define  $\sigma $ by
$$\sigma_{\vert \uh}=1_{\h}, \ \sigma_{\vert \bbV}=z_01_{\H}+z_1J$$
 where $(z_0 + \sqrt{-1} z_1)^3=1$. 
Then $(\ug, \sigma )$ is a Riemannian 3-symmetric Lie algebra.
\end{lemma}

\proof 
The first step is to build the infinitesimal model of $(M,\D)$.  Define the Lie bracket on $\ug$ via
\begin{equation} \label{bra-inf}
\begin{split}
&[F,G]=F \circ G-G \circ F, \ \ [F,v_1]=Fv_1\\
&[v_1,v_2]=\tau^{\D}(v_1,v_2)+R^{\D}(v_1,v_2).
\end{split}
\end{equation}
where $F,G\in \uh$ and $v_i \in \bbV$.
The claim is that the  infinitesimal model  $\ug = \uh \oplus \bbV$ is 3-symmetric.  This follows from the Ambrose--Singer Theorem since
$$ \uh = \text{span} \lbrace R^{\D}(X,Y) : X,Y\in \bbV\rbrace  = \hol(\D).$$
Yet $\Hol(\D)$ is compact since $D$ is metrisable. This implies $(\g, \sigma)$ is a Riemannian 3-symmetric algebra as required, since $\uh$ clearly acts via the isotropy representation on $\bbV$.  \endproof

This  Lemma allows us pass from the differential-geometric level to an algebraic model. The reader should observe that, in general, there are many more 3-symmetric Lie algebras associated to one 3-symmetric space aside from  $\ug$: we will later undertake a  study of this issue. The salient point here is that starting with a 3-symmetric space this Proposition produces a 3-symmetric Lie algebra $\ug$ and this is essential for the proof of Theorem \ref{t1}.

To have complete equivalence in our definitions, it is necessary to go in the converse direction and, starting from an abstract Riemannian 3-symmetric algebra $(\g, \sigma, g)$, construct a Riemannian 3--symmetric space. It is obvious how to construct $J$ from $\sigma$, but care is needed to construct $\D$ and associate a Riemannian homogeneous space $M$ to $(\g, \sigma, g)$.    This will only be fully achieved after a long march through proving the main splitting Theorem of the paper at Lie algebra level (Theorem \ref{last-split} and then computing all possible 3-symmetric Lie algebras in Theorem \ref{thm-gen}).  However the special case when $\ug$ is semisimple is presented in the next section as it needed for the proof of the main splitting theorem (Theorem \ref{last-split}). 

\subsection{Properties of  $(\g, \sigma, g)$} \label{el-pp}

In this subsection we shall take a somewhat different starting point, and begin with $(\g, \sigma, g)$ a Riemannian 3-symmetric algebra.  Let  the skew-symmetric tensors $\tau : \bbV \times \bbV \to \bbV$ respectively 
$\mathfrak{R}:\bbV \times \bbV \to \h$ be defined by 
\begin{equation*}
\tau(v_1,v_2)=[v_1,v_2]_{\bbV} \ \mbox{respectively} \ \mathfrak{R}(v,w):=[v,w]_{\h}.
\end{equation*}
We record the following elementary (and well known) fact which essentially stems from the fact that $\sigma$ is a Lie algebra automorphism. See e.g. \cite[Propn.5.5]{gr1} for proof.
\begin{lemma} \label{basic1}
We have 
\begin{equation} \label{TJ1}
\tau(Jv,Jw)=-\tau(v,w), \ \tau(v,Jw)=-J\tau(v,w),
\end{equation}
as well as 
\begin{equation} \label{RJ1}
\mathfrak{R}(Jv,Jw)=\mathfrak{R}(v,w)
\end{equation}
whenever $v,w \in \bbV$. 
\end{lemma}
The skew-symmetric map $R^{\D}:\bbV \times \bbV \to \mathfrak{gl}_J(\bbV)$
$$ R^{\D}(v_1,v_2)v_3:=[\mathfrak{R}(v_1,v_2),v_3]$$
satisfies the algebraic Bianchi identity 
\begin{equation} \label{B1}
\begin{split}
\mathfrak{S}_{abc}R^{\D}(v_a,v_b)v_c=\mathfrak{S}_{abc}\tau(v_a,\tau(v_b,v_c))
\end{split}
\end{equation}
with cyclic permutation on the indices $abc$. This is a straightforward consequence of the Jacobi identity in $\g$.
\begin{rem}
Again from the Jacobi identity in $\g$ it follows that $R$ satisfies the differential Bianchi identity 
\begin{equation} \label{B2}
R^{\D}(\tau(v_1,v_2),v_3)+R^{\D}(\tau(v_2,v_3),v_1)+R^{\D}(\tau(v_3,v_1),v_2)=0.
\end{equation}
This fact will be however little used. We also record that the infinitesimal model for $3$-symmetric spaces obtained above is a substantial generalisation of nearly-K\"ahler infinitesimal models \cite{Na-NK,Na-ch}.
\end{rem}
The tensor 
$R^{\D}$ takes values in the subalgebra $\tilde{\h}:=\iota(\h) \subseteq \mathfrak{gl}_{J}(\bbV)$; that is 
\begin{equation*}
R^{\D}(v_1,v_2) \in \tilde{\h}
\end{equation*}
for all $v_1,v_2 \in \bbV$. Note $\tilde{\h}$ and $\h$ are isomorphic Lie algebras, 
since the isotropy representation is faithful. Using the Jacobi identity it is easy to see that $\tilde{\h}$ preserves both of the tensors $R^{\D}$ and $\tau$, in other words 
\begin{equation*}
\widetilde{\h} \subseteq \mathfrak{stab}(R^{\D}) \cap \mathfrak{stab}(\tau^{\D}) = : \overline{\h}.
\end{equation*}

Furthermore, define 
\begin{equation} \label{tans-v}
\mathfrak{t}(\bbV):=\bbV+[\bbV, \bbV] \subseteq \g.
\end{equation}
 From the definitions 
this is an ideal in $\g$; it will be referred to as the {\it{transvection}} algebra 
of $\g$. Since $\sigma \in \Aut(\g)$ preserves $\bbV$ we also have $\sigma(\mathfrak{t}(\bbV))=\mathfrak{t}(\bbV)$. 
\begin{lemma} \label{basic2}
Let $\ug:=\uh \oplus \bbV$ where 
\begin{equation*}
\uh=\spa \{\mathfrak{R}(v_1,v_2):v_1,v_2 \in \bbV\}.
\end{equation*}
We have $\mathfrak{t}(\bbV)=\ug$.
\end{lemma}
\begin{proof}
Pick $v_1,v_2 \in \bbV$. Then $[v_1,v_2]=\mathcal{R}(v_1,v_2)+\tau(v_1,v_2)$ belongs to $\uh \oplus \bbV$ thus $[\bbV,\bbV]$ (and hence $\mathfrak{t}(\bbV)$) is contained in $\ug$. Conversely, $\mathcal{R}(v_1,v_2)=[v_1,v_2]-\tau(v_1,v_2)$ belongs to $[\bbV,\bbV]+\bbV$ showing that $\uh$(and hence $\ug$) is contained in 
$\mathfrak{t}(\bbV)$.
\end{proof}
In particular $\uh$ is an ideal in $\h$; it will be referred to as the {\it{holonomy}} 
algebra of $(\g,\sigma)$. As the notation suggests, this algebra $\ug$ agrees with 
the transvection algebra $\ug$ we constructed in the preceding subsection. 
Moreover the algebraic tensors $R^{\D}$ and $\tau$ agree with their geometric counterparts constructed from a 3-symmetric space $(M,J,\D)$.

We now move on to discuss Riemannian 3-symmetric algebras endowed with an isotropy invariant metric. Choose $g \in \me_{H}(\bbV)$. Then $R^{\D}(v_1,v_2) \in \so(\bbV,g)$ for all $v_1,v_2 \in \bbV$. The metric $g$ can be used to define further objects of geometric 
relevance as follows. The intrinsic torsion tensor $\eta^g, v\in \bbV \mapsto \eta^g_v$ 
is uniquely determined from the requirements
$$ \eta^g:\bbV \to \mathfrak{so}(\bbV,g), \ \eta^g_vw-\eta^g_wv=\tau(v,w).
$$
By Lemma \ref{basic1} we have 
\begin{equation} \label{iti}
\eta^g_{Jv}=\eta^g_vJ=-J\eta^g_v.
\end{equation}
The following chain of inclusions has special geometric relevance. 
\begin{equation} \label{sandwich}
\uh \subseteq \h \subseteq {}^g\oh:=\mathfrak{stab}_{\mathfrak{so}(\bbV,g)}(R^{\D}) \cap \mathfrak{stab}_{\so(\bbV,g)}(\tau).
\end{equation}

Note that, from the definitions, $\uh$ is an {\it{ideal}} in both $\h$ and ${}^g\oh$.
The inclusions in \eqref{sandwich} make precise the degree of freedom in choosing the isotropy algebra 
$\h$. We also consider the Nomizu algebra  
\begin{equation} \label{Nom}
{}^g\og={}^g\oh \oplus \bbV
\end{equation}
with Lie bracket given by \eqref{bra-inf}. Clearly $\ug$ is an ideal in ${}^g\og$.
\begin{rem}
The notation $\uh, {}^g\oh$ is a bit misleading since these algebras do not depend on $\h$. Also note that $\uh$ does not depend on the metric.
\end{rem}
\subsection{Some important examples} \label{sym-eg}
Here we record that basic examples of Riemannian $3$-symmetric algebras are provided by the following well known class of Lie algebras.
\begin{defn} \label{hsd}
A semisimple Lie algebra $\mL$ is called Hermitian symmetric provided that 
\begin{itemize}
\item[(i)] it admits a reductive decomposition $\mL=\Kk \oplus \mathcal{H}$ where 
$$[\Kk,\Kk] \subseteq \Kk, \ [\Kk,\mathcal{H}] \subseteq \mathcal{H}, \ [\mathcal{H},\mathcal{H}] \subseteq \Kk$$
\item[(ii)] the isotropy representation $\Kk \to \gl_{J_{\mathcal{H}}}(\mathcal{H})$ is compactly embedded with corresponding isotropy group $K$
\item[(iii)] there exists an $\ad_{\Kk}$-invariant linear complex structure $J_{\mathcal{H}}$ on $\mathcal{H}$ with the property that 
\begin{equation} \label{jinv}
[J_{\mathcal{H}}x,J_{\mathcal{H}}y]=[x,y] \ \mathrm{ whenever} \ x,y \in \mathcal{H}.
\end{equation}
\end{itemize}
\end{defn}
Recall the following well known
\begin{fct}Let $\mL=\Kk \oplus \mathcal{H}$ be an Hermitian symmetric Lie algebra. We have 
\begin{equation} \label{z-HSL}
(\ad_z)_{\vert \mathcal{H}}=J_{\H}
\end{equation}
for some uniquely determined $z\in \z(\Kk)$. When $\mL$ is simple we moreover have $\z(\Kk)=\mathbb{R}z $.
\end{fct}
See \cite{He}, page 378 for details. The above normalisation of the center of $\Kk$ for $\mL$ simple
will be used throughout this paper. Property \eqref{z-HSL} ensures that \eqref{jinv} follows from the $\ad_{\Kk}$-invariance of $J_{\H}$ and the Jacobi identity in $\mL$.

The set of isotropy invariant 
metrics on $\H$ will be denoted, in accordance with \eqref{met-not1}, by 
$ \me_{K}(\H).$ See also Definition \ref{metr}.

Any Hermitian symmetric Lie algebra $\mL$ is naturally equipped with an automorphism $\sigma^{\mL}$ of order $3$ defined by 
\begin{equation} \label{aut-symm}
\sigma^{\mL}_{\vert \h}=1_{\h}, \ \sigma^{\mL}_{\vert \H}=z_01_{\H}+z_1J
\end{equation}
where $(z_0+iz_1)^3=1, z_0 \neq 1$.
Thus $(L,\sigma^L)$ becomes a Riemannian $3$-symmetric algebra in the sense of Definition \ref{3sL}. These examples describe $3$-symmetric Lie algebras of Riemannian type satisfying $\tau=0$. For the vanishing of the torsion is equivalent with $[\bbV,\bbV] \subseteq \h$, in the notation of section \ref{o3}.

 \section{Regularity of models with $\ug$ semisimple}\label{fix-s} 

\subsection{Preliminaries} 

Our definition of a Riemannian 3-symmetric Lie algebra reproduces exactly the algebraic properties of the transvection algebra of a Riemannian $3$-symmetric space $(M,D,J, g)$ at a given point in $M$.  To establish this, we will need to recall the  concept of regularity for infinitesimal models. These are designed to measure up to which extent an homogeneous space can be assigned to an algebraic model. In fact, they were originally introduced by Cartan to pass from his algebraic description of a symmetric space to the geometric level.  In this paper we only look at this notion in the context of Riemannian $3$-symmetric algebras where we show this can be done in a canonical way.  

\begin{defn}\label{def-nomizu}  Given $\uh\subseteq \h\subseteq {}^g\oh$, the infinitesimal model of  $\g = \h \oplus \bbV$ is defined by restricting the Lie brackets of the Nomizu algebra (defined in Equation \ref{Nom}) to $\g$. \end{defn}

\begin{defn} \label{def-reg1}
A  Lie algebra $\g$ with $\sigma\in \text{Aut}(\g): \sigma^3 =1$  is said to be regular if there exists a Lie group 
$G$ with Lie algebra $\g$ and a closed subgroup $\mathrm{H} \subseteq G$ which is compact and connected and such that $\bbV$ is $\Ad(\mathrm{H})$-invariant. Moreover we require $\Ad(\mathrm{H})$ to act effectively on $\bbV$.
\end{defn}
If it exists, a pair $(G,H)$ as above will be called a regular pair for $(\g,\sigma)$. In this setting one can conclude that the infinitesimal model $\g = \h \oplus \bbV$ arises as the tangent space to a Riemannian 3-symmetric space.   

\subsection{Semisimple regularity}\label{ss-reg}

Whenever $G$ is a Lie group and $x \in G$ we indicate with $G_x$ the fixed point set of the inner automorphism $g \mapsto xgx^{-1}$.
\begin{propn} \label{reg1}
Let $\g$ be semisimple and $\sigma\in \Aut (\g)$ satisfy $\sigma^3=1$. Consider the splitting $\g = \h \oplus \bbV$ induced by $\sigma$. The following are equivalent: 
\begin{itemize}
\item[(i)] $(\g, \sigma)$ is a Riemannian 3-symmetric algebra, and there exists a regular pair $(G,H)$ for $(\g,\h)$ such that $H=G_x$ for some $x \in G$ with $x^3=1$. 
\item[(ii)] $\bbV$ admits an $\ad_{\h}$-invariant inner product. 
\end{itemize}

\end{propn}
\begin{proof}
Assuming (i) holds, (ii) is a trivial consequence of the definitions.\\
Assume now that (ii) holds. Consider the Lie group $G:=\Aut(\g)$; its Lie algebra $\Der(\g)$ is isomorphic, via the adjoint representation $\ad$ to $\g$ since $\g$ is semisimple. 
Letting 
\begin{equation*}
H_{\sigma}:=\{f\in G:f \circ \sigma=\sigma \circ f\}
\end{equation*}
certainly yields a closed subgroup 
in $G$. Its Lie algebra is isomorphic, via $\ad$, to 
$$\mathfrak{h}_{\sigma}=\{X \in \g:\ad_X  \circ \sigma=\sigma \circ \ad_X \}.$$ 
Equivalently $X \in \mathfrak{h}_{\sigma}$ if and only if $[X,\sigma(Y)]=\sigma[X,Y]$ for all $Y \in \g$. This is equivalent to $[X-\sigma(X), \sigma(\g)]=0$; since 
$\g$ is semisimple it has vanishing center thus $\sigma(X)-X=0$. We have shown that $\mathfrak{h}_{\sigma}=\ad_{\h}$. Because $\bbV=\ker(\sigma^2+\sigma+1)$ any
$h \in H_{\sigma}$ satisfies $h(\bbV)=\bbV$. Moreover $h \circ \ad_v \circ h^{-1}=\ad_{hv}$
for all $v \in \bbV$. The adjoint representation $\Ad(H_{\sigma})$ of $H_{\sigma}$ 
on $\bbV$ is thus identified (via the isomorphism $\ad:\g \to \Der(\g)$) to $(h,v) \mapsto hv$. To see this is effective let 
$h\in H_{\sigma}$ satisfy $h_{\vert \bbV}=1_{\bbV}$. Since $h$ is an automorphism it follows that $[hF-F,v]=0$ for all $(F,v) \in \h \times \bbV$.Thus $h=1_{\g}$ by using that $\ad_{\h}$ acts faithfully on $\bbV$.

Now consider the map $H_{\sigma}^0 \to O(\bbV,g), h \mapsto h_{\vert \bbV}$. This is well defined since the Lie algebra $\ad_{\h}$ of $\h_{\sigma}$ preserves $g$; it is moreover injective since $H_{\sigma}$ acts effectively on $\bbV$. Thus $H_{\sigma}^0$ is isomorphic with 
a closed and  therefore compact subgroup of $SO(\bbV,g)$ and the result follows. 
\end{proof}
We close this section with the following 
\begin{defn} \label{typeI}
A Riemannian $3$-symmetric space $(M,g,J)$ is said to be of type I, provided its transvection algebra $\ug$, computed at some 
point $x \in M$, is semisimple.
\end{defn}
Since $\D\!R^{\D}=0$ it is easy to see that this definition is independent of the choice of point in $M$.

\section{Structure of the radicals} \label{rad}
Following our convention, let $(\g,\sigma)$ denote  a Riemannian $3$-symmetric Lie algebra.

Let $\underline{\BK}$ be the Killing form of $\ug$; since the latter is an ideal 
in $\g$ we have $\underline{\BK}=\BK_{\vert \ug}$. As $\sigma$ clearly preserves $\ug$ we let $\underline{\sigma}:=\sigma_{\vert \ug} \in \Aut(\ug)$.
The first step towards understanding the structure of the radicals of $\g$ respectively $\ug$ is the following 
\begin{lemma} \label{rad-10}
The following hold
\begin{itemize}
\item[(i)] the restriction 
\begin{equation} \label{ngu}
\underline{\BK}_{\vert \uh}: \uh \times \uh \to \mathbb{R} \ \mbox{is non-degenerate},
\end{equation}
 
\item[(ii)] we have 
\begin{equation*}
\g^{\perp}=\ug^{\perp}=W
\end{equation*}
where the $\h$-invariant subspace $W \subseteq \bbV$ is given by 
\begin{equation*}
W:=\{v \in \bbV: \BK(v, \cdot)=0\}.
\end{equation*}
\end{itemize}
\end{lemma}
\begin{proof}
(i) one uses that $\underline{\BK}_{\vert \uh}=\BK_{\vert \uh}$ is negative definite, thus non-degenerate.\\
(ii) from the definitions $W \subseteq \g^{\perp}$. Pick $v+F \in \g^{\perp}$; from 
$\BK(v+F,\h)=0$ we get $\BK(F,\h)=0$ by using $\BK(\h,\bbV)=0$. Since $\BK_{\vert \h}$ is non-degenerate by \eqref{rad-1n} it follows that $F=0$. We are left with $\BK(v,\bbV)=0$ which yields $v \in W$ thus $\g^{\perp} \subseteq W$. The proof of $\ug^{\perp}=W$ is entirely similar and based solely on \eqref{ngu}.
\end{proof}

Let $\underline{\rad}$ be the radical of $\ug$; as the latter is an ideal in $\g$ we have 
$$ \underline{\rad}=\rad \cap \ug.
$$
In particular $\underline{\rad}$ is an ideal in $\g$ as well. 
\begin{lemma} \label{tr1}
We have $\Tr(\tau_v)=0$ for all $v \in \bbV$.
\end{lemma}
\begin{proof}
Using the properties of the trace and \eqref{TJ1} we have 
$$\Tr(\tau_v)=\Tr(J^{-1}\tau_v J)=-\Tr(J\tau_vJ)=-\Tr(\tau_v)
$$
and the claim follows.
\end{proof}
First structure results for the radical $\ur$ are derived below, based on the observation that $W$ can be interpreted geometrically, 
namely as
the nullity of the non-Riemannian curvature tensor $R^D$. \begin{defn} \label{nil-def} 
A $3$-symmetric Lie algebra $(\g,\sigma)$ is of type II provided that $\h=0$.
\end{defn}
\begin{thm} \label{rad11}Let $(\g,\sigma)$ be Riemannian $3$-symmetric. Then we have the following:
\begin{itemize}
\item[(i)]The radical $W$ satisfies \begin{equation} \label{geo-int}
W=\{v \in \bbV: R^{\D}(v,\bbV)=0\}.
\end{equation}
\item[(ii)] The radical $\ur$ of the transvection algebra $\ug$ satisfies 
\begin{equation} \label{equa}
\underline{\rad}=W \subseteq \bbV.
\end{equation}
Moreover $W$ is $3$-symmetric of type II.
\end{itemize}
\end{thm}
\begin{proof}
(i) Pick $v \in \bbV$ such that $R^{\D}(v,\bbV)=0$.
Therefore, by the algebraic Bianchi identity \eqref{B1} for $R^D$ we obtain 
\begin{equation}
\label{B1null}
R^{\D}(v_1,v_2)v=\tau(v,\tau(v_1,v_2))+\tau(v_1,\tau(v_2,v))+\tau(v_2,\tau(v,v_1))
\end{equation}
whenever $v_1,v_2 \in \bbV$. By Lemma \ref{basic1} 
the second summand is $J$-anti-invariant in $(v_1,v_2)$ whilst the remaining summands 
are $J$-invariant in the same variables. It follows that 
\begin{equation} \label{compt}
\tau_{v} \circ \tau_{v_1}=0.
\end{equation}
Let $v_1,v_2 \in \bbV$; we compute 
\begin{equation*}
\begin{split}
(\ad_v \circ \ad_{v_1})v_2=&\ad_{v}(\tau(v_1,v_2)+R^{\D}(v_1,v_2))=\tau(v,\tau(v_1,v_2))+R^{\D}(v,\tau(v_1,v_2))-R^{\D}(v_1,v_2)v\\
=&-R^{\D}(v_1,v_2)v
\end{split}
\end{equation*} 
by using \eqref{compt} and $R^{\D}(v,\bbV)=0$. Updating \eqref{B1null} by means of \eqref{compt} we get further that 
\begin{equation*}
(\ad_v \circ \ad_{v_1})_{\vert \bbV}=\tau_{v_1} \circ \tau_v+\tau_{\tau(v,v_1)}.
\end{equation*}
At the same time, using again that $R^D(v,\bbV)=0$ makes it easy to check that 
$$(\ad_{v} \circ \ad_{v_1})\h \subseteq \bbV.$$ Since in particular $(\ad_{v} \circ \ad_{v_1})\bbV \subseteq \bbV$ it follows that $$\mathrm{B}(v,v_1)=\Tr(\tau_{v_1} \circ \tau_{v})+\Tr(\tau_{\tau(v,v_1)})=\Tr(\tau_{v_1} \circ \tau_{v})$$
by taking Lemma \ref{tr1} into account. 
But $\Tr(\tau_{v_1} \circ \tau_{v})=\Tr(\tau_{v} \circ \tau_{v_1})=0$ by \eqref{compt}  
thus $\mathrm{B}(v,\bbV)=0$. We have shown that 
\begin{equation} \label{incl1} 
\{v \in \bbV : R^{\D}(v,\bbV)=0\} \subseteq W.
\end{equation}
Equality will be proved during the proof of (ii) below.\\
(ii) Pick $v+F \in \ur$, with $v \in \bbV$ and $F \in \uh$.  
From $0=\underline{\mathrm{B}}(v+F,[\uh,\bbV])=\mathrm{B}(v,[\uh,\bbV])$ we get by successive use of $\mathrm{B}(\bbV,\h)=0$ and the $\h$-invariance of $B$ that
$\mathrm{B}([v,\bbV],\uh)=0$. Thus $\mathcal{R}(v,\bbV)=0$ by \eqref{ngu}. Taking $v_1,v_2 \in \bbV$ we have 
$$0=\underline{\mathrm{B}}(v+F,[v_1,v_2])=\mathrm{B}(v,\tau(v_1,v_2))+\mathrm{B}(F,\mathcal{R}(v_1,v_2))
$$
after using once more that $\mathrm{B}(\h,\bbV)=0$. Taking into account the $J$-invariance 
properties of $\mathcal{R}$ respectively $\tau$ from Lemma 
\ref{basic1} this decouples as $\mathrm{B}(v,\tau(v_1,v_2))=0$ respectively $\mathrm{B}(F,\mathcal{R}(v_1,v_2))=0$. In other words 
$\mathrm{B}(F,\uh)=0$ hence \eqref{ngu} yields $F=0$. 
The upshot is that   
\begin{equation} \label{gdes1} \ur \subseteq \{v \in \bbV : \mathcal{R}(v,\bbV)=0 \}.
\end{equation}
Equality in both of \eqref{geo-int} and \eqref{equa} follows from \eqref{incl1} and $W \subseteq \ur$.
That $W$ is nilpotent as indicated in Definition \ref{nil-def} follows from \eqref{compt} and the first part of the claim.
\end{proof}
As a direct consequence we have the following 
\begin{cor} \label{II-nil}
Assume that $(\g,\sigma)$ is $3$-symmetric of type II. Then $\g$ is nilpotent and moreover $[[\g,\g],\g]=0$.
\end{cor}
\begin{proof}
Since $\h=0$ we must have $R^{\D}=0$. Thus the same argument as the one used for deriving \eqref{compt} shows that $\tau_{v_1} \circ \tau_{v_2}=0$ for all $v_1,v_2 \in \bbV$ and the claim is proved.
\end{proof}
For future use we derive a few consequences of Theorem \ref{rad11} pertaining to the structure of the space fixed vectors $\bbV^{\uh}:=\{v \in \bbV : [\uh, v]=0\}$. The following Corollary will be needed for establishing structure results in section \ref{split-r1}. We indicate with $[\h,\bbV]^{\perp}$ the orthogonal complement in $\g$ of $[\h,\bbV]$ with respect to the Killing form $\mathrm{B}$.
\begin{cor} \label{null-simp}
Let $(\g,\sigma)$ be Riemannian $3$-symmetric. The following hold
\begin{itemize}
\item[(i)] $\bbV^{\h} \subseteq W$
\item[(ii)] $[\h,\bbV]^{\perp}=\h \oplus W$ 
\item[(iii)] if $\g$ is semisimple we have $[\h,\bbV]=\bbV$ and $\bbV^{\h}=0$.
\end{itemize}
\end{cor}
\begin{proof}
(i) using that the Killing form is invariant under $\g$ we obtain the general identity 
\begin{equation} \label{van0}
\BK(\mathcal{R}(v_1,v_2),F)=\BK(v_2,[F,v_1])
\end{equation}
with $F \in \h$ and $v_1,v_2 \in \bbV$. If $[\h,v_0]=0$ it follows that $\BK(\mathcal{R}(v_0,\bbV),\h)=0$ 
thus $v_0 \in \{v \in \bbV : R^{\D}(v,\bbV)=0\}=W$.\\
(ii) again from \eqref{van0} it follows that $[\h,\bbV]^{\perp}=\h \oplus \{v \in \bbV : R^{\D}(v,\bbV)=0\}$ and we conclude as above. \\
(iii)since $\g$ is semisimple we have $W=0$ thus $\bbV^{\h}=0$ and $[\h,\bbV]^{\perp}=\h$. In other words the map 
$\bbV \to \Lambda^1[\h,\bbV], v \mapsto \BK(v,\cdot)_{\vert [\h,\bbV]}$ is injective, thus $\bbV=[\h,\bbV]$ by a dimension argument.
\end{proof}

\section{Splitting the transvection algebra} \label{split-r1}
Let $\g$ be a Lie algebra. A Levi subalgebra is a a semisimple subalgebra $\Ll \subseteq\g$ such that $\g=\Ll \oplus \rad(\g)$, as a direct sum of vector spaces. Levi subalgebras exist, according to the Levi-Malcev theorem, however they are not necessarily unique. Their main use 
is to identify 
\begin{equation} \label{LE-S}
\g=\rad(\g) \rtimes_{\rho} \Ll
\end{equation}
where 
\begin{equation} \label{re-lev}
\rho:\Ll \to \Der(\rad(\g)), \ \  \rho(l)w:=[l,w]
\end{equation}
is the representation given by the Lie bracket.
\subsection{General observations} \label{gen-obs}
In this section we gather a few general facts and examples regarding $3$-symmetric Lie algebras $(\g,\sigma)$.These will in particular work for the subclass of Riemannian $3$-symmetric algebras.
Consider the finite(hence compact) subgroup $\mathbb{Z}_3:=\{1,\sigma,\sigma^2\}$ of $\Aut(\g)$. 
According to general results in \cite{taft1,taft2}(see also \cite{KN}, page 327) there exists a Levi subalgebra 
$\Ll \subseteq \g$ invariant under $\mathbb{Z}_3$ that is $\sigma(\Ll)=\Ll.$ 
Uniqueness for $\Ll$ will be addressed later on in the paper. Choose such a Levi subalgebra and let $\sigma_{\Ll}:=\sigma_{\vert \Ll}$. Since 
$\sigma \in \Aut(\g)$ it must preserve the radical thus $\sigma(W)=W$. Denote $\sigma_W:=\sigma_{\vert W}$; because $\sigma$ is an automorphism, using \eqref{re-lev} leads to 
\begin{equation} \label{c1-rep}
\rho \circ \sigma_{\Ll}=\sigma_W\rho \sigma_{W}^{-1}
\end{equation}
that is $\rho(\sigma_{\Ll}l)=\sigma_W \circ \rho(l) \circ \sigma_W^{-1}$ whenever 
$l \in \Ll$. In Lie theoretic language this can be interpreted as follows. Let $\Sigma_{W}$ be the automorphism of $\Der(W)$ defined according to
$\Sigma_W(f):=\sigma_W \circ f \circ \sigma_W^{-1}$. This satisfies $\Sigma_W^3=1_{\Der(W)}$ and \eqref{c1-rep} simply says that $\rho$ is a morphism of $3$-symmetric Lie algebras
\begin{equation} \label{3-m}
\rho \circ \sigma_{\Ll}=\Sigma_W \circ \rho.
\end{equation}
This is easily seen to be equivalent to 
\begin{equation} \label{type-3}
\begin{split}
&\rho(\ker(\sigma_{\Ll}-1)) \subseteq \ker(\Sigma_W-1)\ \mbox{and} \ \rho(\ker(\sigma_{\Ll}^2+\sigma_{\Ll}^2+1)) \subseteq \ker(\Sigma_W^2+
\Sigma_W+1)\\
& \rho(J_{\sigma_{\Ll}}x)=J_{\Sigma_W}(\rho(x)), \ x \in \ker(\Sigma_W^2+
\Sigma_W+1).
\end{split}
\end{equation}
Here the linear complex structures $J_{\sigma_{\LL}}$ respectively $J_{\Sigma_W}$ are constructed from the automorphisms $\sigma_{\Ll}$ respectively $\Sigma_W$ according to \eqref{J-cons}. 

Indicate with $T_W$ the torsion tensor of $(\Der(W),\Sigma_W)$. According to \eqref{type-3} the map $\rho$ preserves 
the canonical reductive decomposition of $(\Ll,\sigma_{\Ll})$ respectively $(\Der(W),\Sigma_W)$ thus having $\rho$ a morphism of $3$-symmetric Lie algebras entails 
\begin{equation} \label{TT3}
\rho(\tau(x,y))=T_W(\rho(x),\rho(y))
\end{equation} 
for all $x,y \in \ker(\sigma_{\Ll}^2+\sigma_{\Ll}^2+1)$.

Without making further assumptions on the radical \eqref{3-m} is not a particularly demanding constraint on the representation $\rho$ as the following  very important example  makes clear.
\begin{propn} \label{gen-exa}
Let $\pi:G \to \Aut(S)$ be a group representation, where $G$ is semisimple and $S$ solvable; indicate with $\rho:\g \to \Der(\s)$ the induced tangent representation.  Then 
\begin{equation*}
(\s \rtimes_{\rho} \g, \Ad(g)+\pi(g))
\end{equation*}
is $3$-symmetric whenever $g \in G$ satisfies $g^3=1$. 
\end{propn}
\begin{proof}
We have   $\rho \circ \Ad(g)=\pi(g) \rho \pi(g^{-1})$ for all $g \in G$. If $g \in G$ is such that $g^3=1$ equation \eqref{c1-rep} is satisfied when taking $\sigma_{\Ll}=\Ad(g)$ and $\sigma_W=\pi(g)$.
\end{proof}

\begin{rem}
The significance of Proposition \ref{gen-exa} is as follows. As the Killing form restricted to $\h$ is degenerate for all these examples, it is not possible to equip them with any compatible Riemannian 3-symmetric metric. We conclude that the class of 3-symmetric spaces is far more general  than Riemannian 3-symmetric spaces and that the techniques developed in this paper will not be applicable to tackle the classification problem for 3-symmetric spaces in general. \end{rem}

\subsection{Algebras of type II} \label{typII}
As established in Theorem \ref{rad11} the radical $W$ of a Riemannian $3$-symmetric algebra is an algebra of type II; in particular it is nilpotent by Corollary \ref{II-nil}. Thus letting $\g$ be an algebra of type II, we first recall that having $\h=0$ yields
$\g=\bbV=W$. 

The aim in this section is to describe the $3$-symmetric algebra $(\Der(W),\Sigma_W)$ in some detail. In turn, this is needed to understand the main properties of the representation $\rho$ above. A crucial piece of information, which follows from having $\ker(\sigma_W-1)=0$ is that 
\begin{equation} \label{sigW}
\sigma_W=z_0+z_1J_W.
\end{equation}

Define 
\begin{equation*}
\begin{split}
&\Der_{J_W}(W):=\Der(W) \cap \gl_{J_W}(W), \ \Der_{J_W}^{\perp}(W):=\Der(W) \cap \gl_{J_W}^{\perp}(W).\\ 
\end{split}
\end{equation*}
A key ingredient needed for proving the main splitting result in this section is the following 

\begin{lemma} \label{DER0} 
Let $W$ be $3$-symmetric of type II. We have
\begin{equation*}
\begin{split}
&\ker(\Sigma_W-1)=\Der_{J_W}W\\
&\ker(\Sigma_W^2+
\Sigma_W+1)= \Der^{\perp}_{J_W}W\\
&T_W=0.
\end{split}
\end{equation*} 
\end{lemma}
\begin{proof}
The first two equalities above follow by direct algebraic computation based on \eqref{sigW}. The vanishing of the torsion is granted 
by $[\Der^{\perp}_{J_W}W, \Der^{\perp}_{J_W }W] \subseteq  \Der_{J_W}W$.
\end{proof}
In particular the canonical reductive decomposition of $(\Der(W),\Sigma_W)$ reads 
\begin{equation} \label{der-nnJ}
\Der(W)=\Der_{J_W}W\oplus \Der^{\perp}_{J_W}W.
\end{equation}
Yet another piece of information needed in what follows is contained in the following 
\begin{lemma} \label{DER} 
Let $W$ be $3$-symmetric of type II. We have
\begin{equation} \label{der-nnJ1}
\Der^{\perp}_{J_W}W=\{f \in \gl^{\perp}_{J_W}W : fW \subseteq \mathfrak{z}(W) \ \mbox{and} \ f[W,W]=0\}.
\end{equation}
\end{lemma}
\begin{proof}
Pick $f \in \Der_{J_W}^{\perp}(W)$. Because $\tau(Jw_1,Jw_2)=-\tau(w_1,w_2)$ and $fJ_W+J_Wf=0$ the equation 
$f\tau(w_1,w_2)=\tau(fw_1,w_2)+\tau(w_1,fw_2)$ decouples as $f\tau(w_1,w_2)=0$ and $\tau(fw_1,w_2)+\tau(w_1,fw_2)=0$. The first equation means  
that $f[W,W]=0$.  Using that $\tau(w_1,J_Ww_2)=-J_W\tau(w_1,w_2)$ and again $fJ_W+J_Wf=0$ shows that the second equation is equivalent with 
$\tau(fw_1,w_2)=0$ and the claim is proved. 
\end{proof}
This will also be used in section \ref{der-tv} to show how $3$-symmetric Riemannian algebras can be recovered from their transvection algebra.
\begin{rem} \label{split-nilc}
Assume that $\n$ is a Lie algebra with $[\n,[\n,\n]]=0$. If $[\n,\n]=\z(\n)$ we have $\Der(\n,g_{\n})=\Der(\n,J_{\n},g_{\n})$ by \eqref{der-nnJ}. Choosing
a $J_{\n}$-invariant Riemannian metric $g_{\n}$ on $\n$, one can always split $\n=\tilde{\n} \oplus \mathfrak{a}$, orthogonally with respect to $g_{\n}$, where $[\tilde{\n},\tilde{\n}]=\z(\tilde{\n})$ and $\mathfrak{a}$ is abelian with $[\tilde{\n},\mathfrak{a}]=0$. However such a splitting depends on the choice of the metric so it is not canonical.
\end{rem}
\subsection{Abelian radical} \label{abrad}
At this stage it is convenient to summarise the main algebraic properties of representations $\rho$ satisfying \eqref{c1-rep} 
under the additional assumptions that the radical $W$ is abelian with $\sigma_W=z_0+z_1J_W$ and 
$L$ is Hermitian symmetric with order $3$-automorphism given by \eqref{aut-symm}. In this situation $\Der(W)=\gl(W)$ thus the canonical reductive decomposition of $(\Der(W),\Sigma_W)$ is simply the usual Cartan decomposition 
$\gl(W)=\gl_{J_W}(W) \oplus \gl_{J_W}^{\perp}(W)$ with respect to which $J_{\Sigma_W}(f)=f \circ J_W $ for $f \in \gl_{J_W}^{\perp}(W)$. 
From \eqref{type-3} it thus follows that in this set-up, representations satisfying \eqref{type-3} are in $1:1$ correspondence with linear representations satisfying (ii) and (iii) in the Definition below.
\begin{defn} \label{adm-ff}
Let $\mL$ be a Hermitian symmetric Lie algebra with reductive decomposition $\mL=\Kk \oplus \mathcal{H}$ and let $\VV$ be a real vector space equipped with a linear complex structure $J_{\VV}$. A faithful linear representation $\rho:\mL \to \mathfrak{sl}(\VV)$ is called admissible provided that 
\begin{itemize}
\item[(i)] the set of fixed vectors $\VV^{\mL}=0,$
\item[(ii)] we have 
\begin{equation} \label{adm-inv}
\rho(\Kk) \subseteq \gl_{J_{\VV}}(\VV), \rho(\mathcal{H}) \subseteq \gl^{\perp}_{J_{\VV}}(\VV), 
\end{equation}
and
\item[(iii)]  
\begin{equation} \label{ad-ma}
\rho(J_{\mathcal{H}}x)=\rho(x)\circ J_{\VV}
\end{equation}
for all $x \in \mathcal{H}$. 
\end{itemize}
\end{defn}
The reason for making the technical assumption $\VV^{\mL}=0$ will be clarified later on. Admissible representations will play a pivotal r\^ole in this paper. One of their important features is to split in a canonical way according to the splitting of $\mL$ into simple ideals. This will be proved below based on the following preliminary Lemma which is proved by direct computation based on \ref{ad-ma}.

\begin{lemma} 
Let $(\mL, \rho,\VV)$ be admissible. We have 
\begin{equation} \label{jordan}
2\rho(x)\circ \rho(y)=J_V \circ \rho([J_{\mathcal{H}}x,y])+\rho([x,y])
\end{equation}
for all $x,y \in \mathcal{H}$.
\end{lemma}
We also record the following standard facts. 
\begin{lemma} \label{ricci}
Let $\mL=\Kk \oplus \mathcal{H}$ be a Hermitian symmetric Lie algebra. The set $\mathcal{H}^{\Kk}:=\{x \in \mathcal{H} : [\Kk,x]=0\}$ vanishes identically. Moreover  $[\Kk,\mathcal{H}]=\mathcal{H}$ and  $[\mathcal{H},\mathcal{H}]=\Kk$.
\end{lemma}
\begin{proof}
See \cite{He}, Chap.V or use Corollary \ref{null-simp},(iii) with $\sigma$ given by \eqref{aut-symm}.
\end{proof}
We can now make the following splitting result which will be used in the proof of the main structure theorem in the next section.
\begin{thm} \label{split-lasty}
Let $(\mL,\rho,\VV)$ be admissible. Then 
$(\mL,\rho,\VV)$ decomposes as a direct sum representation 
\begin{equation*}
(\mL,\rho,\VV)=(\mL_1,\rho_1,\VV_1) \oplus \ldots \oplus (\mL_r,\rho_r,\VV_r)
\end{equation*}
where $(\mL_i,\rho_i,\VV_i), 1 \leq i \leq r$ are admissible representations of simple 
Hermitian symmetric Lie algebras.
\end{thm}
\begin{proof}
It is enough to show that if the isotropy representation $(\Kk,\mathcal{H})$ of the symmetric Lie algebra $\mL$ splits so does $\rho$. Therefore  assume that $\mathcal{H}_1 \subseteq \mathcal{H}$ is a $\Kk$-invariant subspace such that $(\Kk,\mathcal{H}_1)$ irreducible. From the general theory of Hermitian symmetric algebras we have an $\Kk$-invariant splitting $\mathcal{H}=\mathcal{H}_1 \oplus \mathcal{H}^{\prime}$ where $\mathcal{H}^{\prime}$ is the orthogonal complement of $\mathcal{H}_1$, with respect to the restriction of the Killing form $B_{\mL}$ of $\mL$ to $\mathcal{H}$. 
 
In addition we have that $\mathcal{H}_1$ and $\mathcal{H}^{\prime}$ are $J_{\mathcal{H}}$-invariant and satisfy $[\mathcal{H}_1,\mathcal{H}^{\prime}]=0$; moreover $\Kk=\Kk_1 \oplus \Kk^{\prime}$, a direct sum of ideals satisfying $\Kk_1=[\mathcal{H}_1,\mathcal{H}_1], \Kk^{\prime}=[\mathcal{H}^{\prime},\mathcal{H}^{\prime}]$ as well as 
$[\Kk_1,\mathcal{H}^{\prime}]=[\Kk^{\prime},\mathcal{H}_1]=0$. 

Define linear subspaces 
of $\VV$ according to  
\begin{equation} \label{dec-1}
\VV_1=\rho(\mathcal{H}_1)\VV, \ \VV^{\prime}=\rho(\mathcal{H}^{\prime})\VV.
\end{equation}
Because $\rho(\mathcal{H}) \subseteq \mathfrak{gl}^{\perp}_{J_{\VV}}(\VV)$ the subspaces $\VV_1$ and $\VV^{\prime}$ are $J_{\VV}$-invariant and moreover 
\begin{equation*} \label{dec-2}
\rho(\mathcal{H}_1)\VV_1 \subseteq \VV_1, \ \rho(\mathcal{H}^{\prime})\VV^{\prime}\subseteq \VV^{\prime}
\end{equation*}
by \eqref{dec-1}. Since $[\mathcal{H}_1,\mathcal{H}^{\prime}]=0$ formula \eqref{jordan} ensures that $\rho(x) \circ \rho(y)=0$
for all $x \in \mathcal{H}_1, y \in \mathcal{H}^{\prime}$. Thus 
\begin{equation*} \label{dec-3}
\rho(\mathcal{H}_1)\VV^{\prime}=\rho(\mathcal{H}^{\prime})\VV_1=0.
\end{equation*}
Because $\rho$ is a representation and $[\mathcal{H}_1,\mathcal{H}_1]=\mathfrak{k}_1, [\mathcal{H}^{\prime},\mathcal{H}^{\prime}]=\mathfrak{k}^{\prime}$ these relations lead to 
$$ \rho(\Kk_1)\VV_1 \subseteq \VV_1, \ \rho(\Kk^{\prime})\VV^{\prime} \subseteq \VV^{\prime}
, \rho(\Kk_1)\VV^{\prime}=\rho(\Kk^{\prime})\VV_1=0.$$
Summarising, 
$$ \rho(\mL_1)\VV_1 \subseteq \VV_1, \rho(\mL^{\prime} )\VV^{\prime} \subseteq \VV^{\prime} , \ \rho(\mL_1)\VV^{\prime} =\rho(\mL^{\prime} )\VV_1=0
$$
where $\mL_1:=\mathfrak{k}_1 \oplus \mathcal{H}_1$ and $\mL^{\prime}:=\Kk^{\prime} \oplus \mathcal{H}^{\prime}$. Taking into account that 
$\mL=\mL_1 \oplus \mL^{\prime}$, a direct sum of Lie algebras it follows that $\VV_1 \cap V^{\prime} \subseteq \VV^{\mL}=0$. From $\rho(\mL)\VV=\VV$ we get that $\VV=\VV_1 \oplus \VV^{\prime}$ hence 
$\rho=\rho_1 \oplus \rho^{\prime}$ where $\rho_1:=\rho_{\vert \mL_1}$ and 
$\rho^{\prime}:=\rho_{\vert \mL^{\prime}}$. As it is well known, the Hermitian symmetric Lie algebra $\mL_1$ is simple since its isotropy representation $(\Kk_1,\mathcal{H}_1)$ is irreducible. The claim follows now by induction.
\end{proof}
\subsection{The splitting result} \label{rre}
In this section we assume throughout that $\g$ is a transvection Riemannian $3$-symmetric algebra, i.e. $\g=\underline{\g}$. 
\begin{thm} \label{split-11}
Let $(\g,\sigma)$ be Riemannian $3$-symmetric with $\ug=\g$. We have a direct product splitting of Lie algebras 
\begin{equation} \label{split-dir}
\g=\mL_0 \oplus \n \oplus (\VV \rtimes_{\rho} L) 
\end{equation}
where 
\begin{itemize}
\item[(i)]$(\mL_0,\sigma_0)$
is a semisimple Riemannian $3$-symmetric algebra 
\item[(ii)] $\n$ is a type II algebra
\item[(iii)] $\mL_i=\Kk_i \oplus \mathcal{H}_i$ are simple Hermitian symmetric Lie algebras, $\rho_i:\mL_i \to \Sl(\VV_i)$ with $1 \leq i \leq r$ are admissible representations 
and $(\mL,\rho,\VV)=\bigoplus \limits_{i=1}^r (\mL_i,\rho_i,\VV_i)$ is the product representation.

\end{itemize}
With respect to \eqref{split-dir} we have 
\begin{equation} \label{split-hv}
\begin{split}
&\h=\Kk_0 \oplus \Kk, \ \bbV=\mathcal{H}_0 \oplus \n \oplus (\VV \oplus \mathcal{H})\\
&W=\n \oplus \VV
\end{split}
\end{equation}
with $\Kk_0=\ker(\sigma_{0}-1_{L_0})$ and $\mathcal{H}_0=\ker(\sigma_{0}^2+\sigma_{0}+1_{L_0})$ as well as $\Kk=\bigoplus \limits_{i=1}^r \Kk_i, \mathcal{H}=\bigoplus \limits_{i=1}^{r}\mathcal{H}_i$ and 
$\VV=\bigoplus \limits_{i=1}^r \VV_i$.
\end{thm}
\begin{proof}
Choose a $\sigma$-invariant Levi factor $\Ll$ as explained in section \ref{gen-obs}. By Theorem \ref{rad11} the radical of $\g$ then satisfies $\mathfrak{r}(\g)=W$, see also \eqref{geo-int} for the definition of the latter.
The proof continues in several steps as follows.\\
{\it{Step 1: Reduction to $(\Ll,\rho,W)$ faithful.}}
 Consider the ideal $\mL_0:=\ker(\rho)$ of $\Ll$; since $\Ll$ is semisimple we can split 
$ \Ll=\mL_0 \oplus \mL$ orthogonally with respect to the Killing form $\BK_{\Ll}$ of $\Ll$. Then $\mL$ is an ideal in $\Ll$ in particular $[\mL_0,\mL]=0, \ [\mL,\mL]=\mL. $
By \eqref{c1-rep} we have $ \sigma_{\Ll}(\mL_0) \subseteq \mL_0. $ 
Since $\sigma_{\Ll} \in \Aut(\Ll)$ we have $\sigma_{\Ll}^{\star}\BK_{\Ll}=\BK_{\Ll}$ thus $ \sigma_{\Ll}(\mL) \subseteq \mL$
as well. Consider the automorphisms $\sigma_{0} \in \Aut(\mL_0)$ respectively $\sigma_{\mL} \in \Aut(\mL)$ given by $\sigma_{0}:=(\sigma_{\Ll})_{\vert \mL_0}$ respectively $\sigma_{\mL}:=(\sigma_{\Ll})_{\vert \mL}$. Since $\sigma_{0}^3=1_{\mL_0}$ and $\sigma_{\mL}^3=1_{\mL}$ we can further split 
$$ \mL_0=\Kk_0 \oplus \mathcal{H}_0, \ \mL=\Kk \oplus \mathcal{H}
$$
where the isotropy algebras $\Kk_0:=\ker(\sigma_{0}-1_{\mL_0}), \Kk:=\ker(\sigma_{\mL}-1_{\mL})$ and the spaces $\mathcal{H}_0:=\ker(\sigma_{0}^2+\sigma_{0}+1_{\mL_0}), \mathcal{H}:=\ker(\sigma_{\mL}^2+\sigma_{\mL}+1_{\mL})$. Thus 
\begin{equation} \label{ppsp}
\h=\mathfrak{k}_0 \oplus \Kk \ \mathrm{and} \ \bbV=W \oplus \mathcal{H}_0 \oplus \mathcal{H}.
\end{equation}
The claimed Lie algebra splitting is now straightforward as $[\mL_0,\mL]=0$ and $\rho(\mL_0)W=0$. Furthermore, the restriction to $\mathcal{H}_0$ of any metric in $\me_{\mathrm{H}}(\bbV)$ is $\Kk_0$-invariant; since $\mL_0$ is semisimple, Proposition \ref{reg1} thus allows concluding that the isotropy representation $(\Kk_0, \mathcal{H}_0)$ is compactly embedded. It follows that the $3$-symmetric 
Lie algebra $\mL_0$ has Riemannian type. See also Corollary \ref{cor-grp} below for a more constructive proof.\\
%
{\it{Step 2: $\mL$ is Hermitian symmetric.}} As explained above $\rho:(\mL,\sigma_{\mL}) \to (\Der(W),\Sigma_{W})$ is a morphism of $3$-symmetric algebras. 
 
Because $W$ is a type II algebra the $3$-symmetric algebra $(\Der(W),\Sigma_W)$ has vanishing torsion by Lemma \ref{DER}. Combining this fact with \eqref{TT3} 
yields $\rho(\tau(x,y))=0$ for all $x,y \in \mathcal{H}$. As $\rho$ is faithful on $\mL$ it follows that $\tau(\mathcal{H},\mathcal{H})=0$ that is 
$[\mathcal{H},\mathcal{H}] \subseteq \Kk$. Since $W$ is an ideal in $\g$ from \eqref{ppsp} we get $\uh=[\mathcal{H},\mathcal{H}]\oplus [\mathcal{H}_0,\mathcal{H}_0]_{\Kk_0}$. The assumption 
$\uh=\h$ thus forces $[\mathcal{H},\mathcal{H}]=\Kk$; since $\mL$ is semisimple this is easily seen to be equivalent with having the isotropy 
representation $(\Kk,\mathcal{H})$ faithful.

In particular the cubic automorphism $\sigma_{\mL}$ is of the form \eqref{aut-symm} with respect to the reductive splitting $\mL=\Kk \oplus \mathcal{H}$. The restriction $J_{\mathcal{H}}:=J_{\vert \mathcal{H}}$ is clearly $\Kk$-invariant and moreover $[J_{\mathcal{H}}x,J_{\mathcal{H}}y]=
[x,y]$ whenever $x,y \in \mathcal{H}$, by \eqref{RJ1}. Clearly, the restriction to $\mathcal{H}$ of any metric in $\me_{\mathrm{H}}(\bbV)$ is $\Kk$-invariant 
thus the isotropy representation $(\Kk, \mathcal{H})$ is compactly embedded, which finishes the proof of the claim. \\
{\it{Step 3: Splitting of $W$.}} Using the notation in section \ref{semi-Lie} let $\n:=W^{\mL}$ be the subspace of $W$ on which $L$ acts trivially; this is a subalgebra in $W$ since $\mL$ acts on the latter by derivations.   
Because $\mL$ is semisimple the representation $(\mL,\rho,W)$ is completely reducible by Weyl's Theorem thus 
$ W=\mathfrak{n} \oplus \VV$
a direct sum of vector spaces with $\rho(\mL)V \subseteq \VV$. Note that $\VV$ is uniquely determined from $\VV=\rho(\mL)W$ and satisfies $\VV^{\mL}=0$; see also \eqref{weyl-1} for considerations of this type. 
Combining \eqref{type-3} and Lemma \ref{DER0} yields 
\begin{equation} \label{pp-rr}
\rho(\h) \subseteq \Der_{J_W}W, \ \rho(\mathcal{H}) \subseteq \Der_{J_W}^{\perp}W.
\end{equation}
Thus Lemma \ref{DER} ensures that $\rho(x)W \subseteq \mathfrak{z}(W)$ for all $x \in \H$. Since $[\mathcal{H},\mathcal{H}]=\Kk$ this leads, after taking commutators, to $\rho(\mL)W \subseteq \mathfrak{z}(W)$, that is $\VV \subseteq \mathfrak{z}(W)$. It follows that $W=\n \oplus \VV$ is a Lie algebra splitting with $\VV$ abelian and $\n$ a Lie algebra of type $II$ . As $\mL$ acts trivially on $\n$ the representation $(\mL,\rho,\VV)$ is easily seen to be 
admissible by using \eqref{pp-rr} together with \eqref{c1-rep}.\\
{\it{Step 4: Splitting of $(\mL,\rho,\VV)$.}} This is done according to Theorem \ref{split-lasty}.\\
\end{proof}
We may also prove directly that any compact subgroup integrating the isotropy representation splits canonically; this provides an alternate proof of the fact that the factors $\mL_0$ respectively $\VV_i \rtimes_{\rho_i} \mL_i, 1 \leq i \leq r$ in Theorem 
\ref{split-11} are Riemannian $3$-symmetric. This argument 
is based on the following elementary
\begin{lemma} \label{grp-rep}
Let $\mathrm{H}$ be a connected and compact Lie group equipped with an effective representation $\pi:\mathrm{H} \to \GL(\VV)$ onto 
a vector space $\VV$. Assume that the representation $\rho=(d \pi)_e:\h\to \gl(V)$ of the Lie algebra $\h$ of $\mathrm{H}$ is a direct product representation, that is $\h=\Kk_0 \oplus \Kk$ is a direct product of Lie algebras and $\VV=\VV_1 \oplus \VV_2$ is a direct sum with 
$\rho(\Kk_0)\VV_2=0$ and $\rho(\Kk)\VV_1=0$. Then we have a Lie group isomorphism 
$$ \mathrm{H} \cong K_0 \times K_1
$$
where the Lie groups $K_0$ and $K_1$ are compact and connected with Lie algebras $\Kk_i, i=0,1$. Moreover $\pi$ is a direct product
$\pi=\pi_1 \times \pi_2$ where the representations $\pi_i : K_i \to \GL(\VV_i)$ are effective and satisfy $(d \pi_i)_e=\rho_i$ for $i=0,1$.
\end{lemma}
\begin{proof}
 Consider the closed(hence compact) Lie subgroups of $\mathrm{H}$ given by 
\begin{equation*}
\begin{split}
&\mathrm{H_1}:=\{h \in \mathrm{H} : \pi(h)v=v \ \mathrm{for \ all} \ v \ \in \VV_2\},\ \mathrm{H_2}:=\{h \in \mathrm{H} : \pi(h)v=v \ \mathrm{for \ all} \ v \ \in \VV_1 \}.
\end{split}
\end{equation*}
First record that the representations $(\Kk_i,V_i), i=0,1$ induced by $\rho$ are faithful; this can be established by exponentiating and using that $\pi$ is effective.  It follows 
that the Lie algebras of $H_i$ equal $\Kk_i$ for $i=0,1$. Now let $K_0:=H_1^0$ and $K=H_2^0$ where the superscript indicates the connected component through the unit element. Elements of $K_0$ and $K$ mutually commute since $[\Kk_0,\Kk_1]=0$ and 
$K_0$ and $K$ are connected. 
Since $\mathrm{H}$ acts effectively on $V$ it follows that $K_0 \cap K=\{e\}$. Therefore $K_0K$ is a connected subgroup of $\mathrm{H}$ with the same Lie algebra hence $K_0K=\mathrm{H}$. It follows that $\mathrm{H}$ is isomorphic to 
$K_0 \times K$ and since $K_0$ respectively $K$ are compact and act effectively on $\VV_0$ respectively $\VV_1$ the claim is proved
\end{proof}
Note that having $\mathrm{H}$ connected is essential above; if $\mathrm{H}$ is not connected we may have an irreducible 
representation of $\mathrm{H}$ that is reducible as a representation of its Lie algebra $\h$; see \cite{CS} for more details.
\begin{cor} \label{cor-grp}Let $(\g,\sigma)$ be Riemannian $3$-symmetric and let $\mathrm{H}$ be a connected and compact Lie group 
integrating the isotropy representation. Then we have a Lie group isomorphism
$$ \mathrm{H} \cong K_0 \times K_1 \times \ldots \times K_r
$$
where the Lie groups $K_0, K_i$ are compact and connected and integrate the isotropy representations of the Riemannian $3$-symmetric Lie algebras $\mL_0$ respectively 
$\VV_i \rtimes_{\rho_i} \mL_i$ for $1 \leq i \leq r$.
\end{cor}
\begin{proof}
By Theorem \ref{split-11} the isotropy representation $\h$ on $\bbV$ is a direct product with $\h=\Kk_0 \oplus \bigoplus_i \Kk_i$ and 
$\bbV=\H_0 \oplus \n \oplus \bigoplus \limits_{i=1}^r (\H_i \oplus \VV_i)$. The claim follows now from Lemma \ref{grp-rep}.
\end{proof}
For future reference, we note that $\mathrm{H}$ acts trivially on the nilpotent factor $\n$ in \eqref{split-dir}.

Type distinction between the factors in Theorem \ref{split-11} which are neither semisimple nor nilpotent of type II is made according to the following 
\begin{defn} \label{3+4}
A Lie algebra of the form $\g=\VV \rtimes_{\rho} \mL$ where $(\mL, \rho,\VV)$ is an admissible representation with 
$\mL$ simple is of type III provided $\mL$ has non-compact type. It is called of type IV provided $\mL$ has compact type.
\end{defn}
An implicit consequence of Theorem \ref{split-11} is to have the representation $(\mL,\rho,\VV)$ metrisable; that is it admits isotropy invariant metrics in the sense of the following 
\begin{defn}\label{def-iso-m} Let $(\mL,\rho,\VV)$ be an admissible representation. An isotropy invariant metric is an inner product on $\VV$ such that 
$\rho(\Kk) \subseteq \un(\VV,g_{\VV},J_{\VV})$.
\end{defn}
Their existence is in fact granted by general representation theoretic considerations, see section \ref{backgm}.
The set of isotropy invariant metrics will be denoted by
$\me_{\Kk}(\VV).$ How $\me_{\mathrm{H}}(\bbV)$ relates to this is explained in the next section.
\section{Metrisability and regularity} \label{admr1}
\subsection{Background metrics and Satake's work} \label{backgm} 
We outline here the specific types of isotropy invariant metrics supported by linear representations of semisimple Lie algebras. Existence, together with specific properties in the admissible case, is established with the aid of a classical result of Mostow.
\begin{defn} \label{bgr-m}
Let $(\mL,\rho,\VV)$ be a faithful linear representation where $\mL=\Kk \oplus \mathcal{H}$ is a semisimple 
symmetric Lie algebra. An inner product $h_{\VV}$ on $\VV$ is called a background metric provided that 
\begin{itemize}
\item[(i)] $h_{\VV}$ is isotropy invariant, $\rho(\Kk) \subseteq \so(\VV,h_{\VV})$
\item[(ii)] we have $\rho(\mathcal{H}) \subseteq \so(\VV,h_{\VV})$ (respectively 
$\rho(\mathcal{H}) \subseteq \Sym^2(\VV,h_{\VV})$) if $\mL$ has compact  (respectively non-compact) type. 
\end{itemize}
\end{defn}
The proof of existence is outlined below.
\begin{propn} \label{exist-b}
Let $(\mL,\rho,\VV)$ be a faithful representation where $\mL$ is a simple Lie algebra. 
Then $\VV$ carries a background metric.
\end{propn}
\begin{proof}
When $\mL$ has compact type this follows by the standard averaging trick over the compact connected group $G$ with Lie algebra $\mL$.
When $\mL$ has non-compact type the simple Lie subalgebra $\rho(\mL) \subseteq \mathfrak{sl}(\VV)$ has a Cartan involution given by 
$\rho(1_{\Kk}-1_{\mathcal{H}})$. By a classical result of Mostow(see \cite{Mo}, Thm.6) there exists a Cartan involution $\sigma^{\prime}$ of $\mathfrak{sl}(\VV)$ with $\sigma^{\prime}_{\vert 
\rho(\mL)}=\rho(1_{\Kk}-1_{\H})$. However any Cartan involution of $\mathfrak{sl}(V)$, in particular $\sigma^{\prime}$,
is of the form $f \mapsto -f^T$ where the transpose is taken with respect to some Riemannian 
metric $h_{\VV}$ on $\VV$. The latter metric thus satisfies $\rho(\Kk) \subseteq \so(\VV,h_{\VV}), 
\rho(\mathcal{H}) \subseteq \Sym^2(\VV,h_{\VV})$. 
\end{proof}
The space of isotropy invariant metrics in the sense of Definition \ref{bgr-m},(i) is thus non-empty and shall be denoted by $\me_{\Kk}(\VV).$
For the subclass of admissible representations we need the following analogous 
\begin{defn} \label{bgr-madm}
Let $(\mL,\rho,\VV)$ be an admissible representation where $\mL=\Kk \oplus \mathcal{H}$ is simple. 
An inner product $h_{\VV}$ on $\VV$, compatible with $J_{\VV}$ is called a background metric provided that 
\begin{itemize}
\item[(i)] $h_{\VV}$ is isotropy invariant, $\rho(\Kk) \subseteq \un(\VV,h_{\VV},J_{\VV})$
\item[(ii)] we have $\rho(\mathcal{H}) \subseteq \un(\VV,h_{\VV},J_{\VV})$ (respectively 
$\rho( \mathcal{H}) \subseteq \un^{\perp}(\VV,h_{\VV},J_{\VV})$) if $\mL$ has compact (respectively non-compact) type. 
\end{itemize}
\end{defn}
Existence follows from replacing a background metric $h_V$ as provided by Proposition \ref{exist-b} with $h_V+h_V(J_V\cdot, J_V\cdot)$. Thus the space of isotropy invariant metrics is not empty. We will always assume, without loss of generality, that if $\rho$ is admissible we are working with a background metric  in this modified sense, i.e. one which satisfies these additional properties with respect to $J_V$. We will lightly abuse notation by continuing to denote this set as $\me_{\Kk}(V)$.

Having thus recalled how metrisability works for representations of semisimple Lie algebras we can relate, at least partially, 
admissible representations, to a class of representations studied in \cite{Sat1}. 

Let $(\VV,J_{\VV})$ be a real vector space with linear complex structure $J_V$. With respect to some $J_{\VV}$-invariant inner product 
$h_{\VV}$ on $\VV$ we consider the symplectic form given by $\omega_{\VV}=h_{\VV}(J_{\VV} \cdot, \cdot)$. The Cartan splitting of the corresponding symplectic Lie algebra is 
\begin{equation} \label{cart-symp}
\mathfrak{sp}(\VV,\omega_{\VV})=\mathfrak{u}(\VV,h_{\VV},J_{\VV}) \oplus \mathfrak{u}^{\perp}(\VV,h_{\VV},J_{\VV})
\end{equation}
where $\mathfrak{u}^{\perp}(\VV,h_{\VV},J_{\VV})=
\Sym^2(\VV,h_{\VV}) \cap \gl^{\perp}(\VV,J_{\VV})$. In particular $z^{\prime}=-\frac{1}{2}J_{\VV}$ is a normalised generator of the center of $\mathfrak{u}(\VV,h_{\VV},J_{\VV})$.

\begin{propn} \label{symp1}
Let $(\mL,\rho,\VV)$ be an admissible representation where $\mL$ is of non-compact type. Whenever $h_{\VV}$ is a background metric on $\VV$ denote $\omega_{\VV}=h_{\VV}(J_{\VV} \cdot, \cdot)$. We have 
\begin{equation} \label{symp2}
\rho(\mL) \subseteq \mathfrak{sp}(\VV,\omega_{\VV}).
\end{equation}
The induced Lie algebra morphism satisfies 
\begin{equation} \label{symp3}
\rho \circ \ad_z=\ad_{z^{\prime}} \circ \rho 
\end{equation}
where $z^{\prime}$ denotes the normalised generator of the center of $\mathfrak{u}(\VV,h_{\VV},J_{\VV})$.
\end{propn}
\begin{proof}
Property \eqref{symp2} follows by putting together \eqref{adm-inv}, \eqref{cart-symp} and $\rho(\mathcal{H}) \subseteq \Sym^2(\VV,h_{\VV})$. Equation \eqref{symp3} is clearly equivalent to 
\eqref{ad-ma}. 
\end{proof}
Lie algebra morphisms $\rho:\mL \to \mathfrak{sp}(\VV,\omega_{\VV})$(or equivalently symplectic representations) satisfying \eqref{symp3} were fully classified by Satake in \cite{Sat1} in case $\mL$ is simple of non-compact type and $\rho$ is irreducible; note that therein \eqref{symp3} is referred to as property $(H_1)$. Satake's classification of admissible representations will be explained in section \ref{met-adm} of the paper. 
\begin{rem} \label{dual1}
\begin{itemize}
\item[(i)]Note that Satake does not treat admissible representations where $\mL$ has compact type. Instead he considers $\mL=\mL_0 \oplus  \mL_1$ where 
$\mL_0$ has compact type, the Hermitian symmetric Lie algebra $\mL_1=\Kk_1 \oplus \mathcal{H}_1$ has only simple factors of non-compact type and the representation $(\mL,\rho,V)$ satisfies 
$\rho(\mL_0 \oplus \Kk_1) \subseteq \mathfrak{u}(V,g_V,J_V), \ \rho(\mathcal{H}_1) \subseteq \mathfrak{u}^{\perp}(V,g_V,J_V)$ together with 
\eqref{ad-ma} on $\mathcal{H}_1$. Adding in the factor $\mL_0$ is irrelevant geometrically for then the isotropy representation of $\mL=(\mL_0 \oplus \Kk_1) \oplus \mathcal{H}_1$ is no longer faithful.
\item[(ii)] A weaker version of Theorem \ref{split-lasty} has been established in \cite{Sat1}[Theorem 2]. When $\mL$ is non-compact Theorem 
\ref{split-lasty} has been proved in \cite{Iha2}[Proposition1]; however the proof heavily uses root systems as well as previous work from \cite{Iha} and is particularly involved.
\item[(iii)] In section \ref{sat} we will show how to recover explicitely admissible representations of Lie algebras of compact type from those of non-compact, via a specific kind of duality principle.
\end{itemize}
\end{rem}
\subsection{Structure of the metrics} \label{split-r}
Our next aim is to determine the structure of metrics $g \in \me_{\mathrm{H}}(\bbV)$ with respect to the splitting of $\bbV$ in \eqref{split-hv}. We open by establishing the following preliminary 
\begin{lemma} \label{fix-2}
Whenever  $(\mL=\Kk\oplus \mathcal{H}, \rho,\VV)$ is admissible we have $\rho(\Kk)\VV=\VV$.
\end{lemma}
\begin{proof}
First we show that $\VV^{\Kk}:=\{v \in \VV: \rho(\Kk)v=0\}$ vanishes identically. Clearly $J_{\VV} \VV^{\Kk}=\VV^{\Kk}$, thus \eqref{jordan} leads to 
$$ \rho(x)\rho(y)\VV^{\Kk}=0
$$
for all $x,y \in \mathcal{H}$. To interpret this in an invariant way consider the linear space $\VV_1:=\rho(\mathcal{H})\VV^{\Kk}$.
From $\rho(\mathcal{H})\VV_1=0$ and $[\mathcal{H},\mathcal{H}]=\Kk$(see Lemma \ref{ricci}, (iii)) it follows that $\rho(\Kk)\VV_1=0$ hence $\VV_1$ is acted on trivially by $\mL$. That is $\VV_1=0$ since $\VV^{\mL}=0$. According to the definition of $\VV_1$ we thus have 
$\rho(\mathcal{H})\VV^{\Kk}=0$; repeating the above argument leads to $\rho(\mL)\VV^{\Kk}=0$ and further to
$\VV^{\Kk} \subseteq \VV^{\mL}=0$. The claim follows by a standard orthogonality argument with respect to a background metric.\\
\end{proof}
The next technical ingredient we need is related to the eigenvalue structure of $\rho(z)J_{\VV}$, as follows.
\begin{lemma} \label{rho-z}
Assume that $(\mL,\rho,\VV)$ is admissible and denote $\VV_{\lambda}=\ker(\rho(z)+(\tfrac{1}{2}-\lambda)J_{\VV})$ whenever 
$\lambda \in \bbR$. Then 
$$ 2\lambda(\dim \VV_{\lambda}+\dim \VV_{-\lambda})=\dim \VV_{\lambda}-\dim \VV_{-\lambda}.
$$ 
\end{lemma}
\begin{proof}
When $\mL$ has non-compact type this has been proved in \cite{Sat1}, see end of proof of Proposition 1. Since the argument applies to any Hermitian algebra $\mL$ it will not be reproduced here. For further use we just record that the main ingredient in the proof consists in  
showing that $\VV_{\lambda} \oplus \VV_{-\lambda}$ is $\mL$-invariant and $\rho(\H) \VV_{\pm \lambda} \subseteq \VV_{\mp \lambda}$.
\end{proof}
We also need to make the following 
\begin{lemma} \label{met-mix}
Whenever $(\mL=\Kk\oplus \mathcal{H},\rho,\VV)$ is an admissible representation the $\Kk$-module 
$\Hom_{\Kk}^{\prime}(\mathcal{H},\VV):=\{F \in \Hom_{\Kk}(\mathcal{H},\VV):FJ_{\mathcal{H}}=J_{\VV}F\}$ 
vanishes.
\end{lemma}
\begin{proof}
First consider the space $\VV_{\lambda}, \lambda=\tfrac{3}{2}$, in the notation of Lemma \ref{rho-z}, which guarantees that 
$3(\dim \VV_{\lambda}+\VV_{-\lambda})=\dim \VV_{\lambda}-\VV_{-\lambda}$. It follows that 
$\dim \VV_{\lambda}=0$ hence $\VV_{\lambda}=0$. 
Whenever $F \in \Hom_{\Kk}^{\prime}(\mathcal{H},\VV)$ we have $F[h,x]=\rho(h)Fx$ for $(h,x) \in \Kk \times \mathcal{H}$.  Taking $h=z$ yields  
$FJ_{\mathcal{H}}=\rho(z)F$. As $FJ_{\mathcal{H}}=J_{\VV}F$ we obtain that 
$\IM(F) \subseteq \VV_{\lambda}=0$
and the claim is proved.
\end{proof}
 
When $\mL=\Kk\oplus \mathcal{H}$ is a semisimple $3$-symmetric Lie algebra we denote with $\me_{\Kk}(\mathcal{H},J_{\mathcal{H}})$ the set of Riemannian inner products $g_{\mathcal{H}}$ on $\mathcal{H}$, compatible with $J_{\mathcal{H}}$ and such that 
$\ad_{\Kk}(\mathcal{H}) \subseteq \un(\mathcal{H},g_{\mathcal{H}},J_{\mathcal{H}})$.  

\begin{propn} \label{fix-2} Let $g$ belong to $\me_{\mathrm{H}}(\bbV)$. With respect to the splitting \eqref{split-hv},
\begin{equation*}
g=g_{\mathcal{H}_0}+g_{\n}+\bigoplus \limits_{i=1}^r(g_{V_i}+g_{\mathcal{H}_i}),
\end{equation*}
where 
\begin{itemize}
\item[(i)]  $g_{\mathcal{H}_0}\in \me_{\Kk_0}(\mathcal{H}_0)$
\item[(ii)]$g_{\n}$ is a $J_{\n}$-invariant metric on $\n$ 
\item[(iii)]$(g_{V_i},g_{\mathcal{H}_i})\in \me_{\Kk_i}(\VV_i) \times \me_{\Kk_i}(\mathcal{H}_i)$ for $1 \le i \le r$.
\end{itemize}
\end{propn}
\begin{proof}
From the structure of the Lie bracket in $\g$ with respect to \eqref{split-dir} we have 
$$ [\Kk, \mathcal{H}_0]=[\Kk_0,W]=[\Kk,\n]=0.
$$
The $\h$-invariance of $g$ thus leads to 
\begin{equation*} 
\begin{split}
&g(\mathcal{H}_0,[\Kk, \mathcal{H}])=g(W,[\Kk_0, \mathcal{H}_0])=g(\n, [\Kk,\mathcal{H}])=0\\
&g(\n,\rho(\Kk)\VV)=0.
\end{split}
\end{equation*}
Because both of $\mL=\Kk\oplus \H$ and $\mL_0=\Kk_0 \oplus \mathcal{H}_0$ are semisimple and $3$-symmetric we have $[\Kk,\mathcal{H}]=\mathcal{H}$ and $[\Kk_0, \mathcal{H}_0]=\mathcal{H}_0$ by (iii) in Corollary \ref{null-simp}; at the same time we have $\rho(\Kk)\VV=\VV$ by Lemma \ref{fix-2}. We conclude that  
$$g(\mathcal{H}_0,\mathcal{H})=g(W,\mathcal{H}_0)=g(\n,\mathcal{H})=g(\n,\VV)=0.$$
In an entirely similar fashion, from $[\Kk_i,\mathcal{H}_i]=\mathcal{H}_i$ and $[\Kk_i,\mathcal{H}_j]=0$ we get 
$$g(\mathcal{H}_i,\mathcal{H}_j)=0, 1 \leq i \neq j \leq r.$$
The restriction $g_{\VV}$ of $g$ to $\VV$ clearly belongs to $\me_{\Kk}(\VV)$. Thus we may parametrise 
$g_{\vert \mathcal{H} \times \VV}=g_{\VV}(F \cdot, \cdot)$ with $F \in \Hom_{\Kk}^{\prime}(\mathcal{H},\VV)$. As the latter space vanishes by Lemma \ref{met-mix}, we get 
$$g(\mathcal{H},\VV)=0.$$
Finally consider the splitting $\VV=\bigoplus \limits_{i=1}^r\VV_i$ provided by Theorem \ref{split-lasty} and pick $1 \leq j \leq r$. 
Since the metric $g$ is $\Kk$-invariant and 
$\mathfrak{k}_j$ acts trivially on $\widehat{\VV}_j:=\bigoplus \limits_{i=1,i \neq j }^r\VV_i$ we obtain $g(\rho(\mathfrak{k}_j)\VV_j,\widehat{\VV}_j)=0$. As Lemma \ref{fix-2} grants that $\rho(\mathfrak{k}_j)\VV_j=\VV_j$ we have shown that $g(\VV_j,\widehat{\VV}_j)=0$, and the claim is fully proved. 
\end{proof}
In addition, we record that after splitting 
$$ \mathrm{H} \cong K_0 \times K_1 \times \ldots \times K_r
$$
according to Corollary \ref{cor-grp} and going through the properties of that splitting we obtain after exponentiating 
$$ g_{\mathcal{H}_0}\in \me_{K _0}(\mathcal{H}_0) \ \mathrm{and} \ (g_{\VV_i},g_{\mathcal{H}_i})\in \me_{K_i}(\VV_i) \times \me_{K_i}(\mathcal{H}_i)$$
for $1 \leq i \leq r$. This fully describes the product structure of the space $\mathrm{Met}_{\mathrm{H}}\bbV$.

In order to pin down further properties of the splitting of metrics $g \in \me_{\Kk}(\bbV)$ from Proposition \ref{fix-2} above 
we make the following 
\begin{defn} \label{inv-it}
A linear subspace $\mathbb{L} \subseteq \bbV$ is invariant under the intrinsic torsion tensor $\eta^g$ for some $ g \in \me_{\mathrm{H}}(\bbV)$ provided that $\eta^g_{\bbV}\mathbb{L} \subseteq \mathbb{L}$.
\end{defn}
The geometric content of this definition will be clarified in section \ref{loc-red}. 
\begin{lemma} \label{ideals}
Assume that $\bbV=\bbV_1 \oplus \bbV_2$ orthogonally with respect to $g \in \me_{H}(\bbV)$ and that we moreover have $\tau(\bbV_1,\bbV_1) \subseteq \bbV_1, \ \tau(\bbV_2,\bbV_2) \subseteq \bbV_2, \ \tau(\bbV_1,\bbV_2)=0$. Then $\bbV_1$ and $\bbV_2$ are invariant under the intrinsic torsion tensor $\eta^g$. 
\end{lemma}
\begin{proof}
From the definitions we have 
\begin{equation} \label{it-expr}
2g(\eta^g_{v_1}v_2,v_3)=g(\tau(v_1,v_2),v_3)-g(\tau(v_2,v_3),v_1)+g(\tau(v_3,v_1),v_2)
\end{equation}
for $v_i \in \bbV, 1 \leq i \leq 3$. Pick now $v_1,v_2 \in \bbV_1, v_3 \in \bbV_2$; 
since $\tau(\bbV_1,\bbV_2)=0$ we get $2g(\eta^g_{v_1}v_2,v_3)=g(\tau(v_1,v_2),v_3)$. The latter summand vanishes since $\tau(\bbV_1,\bbV_1) \subseteq V_1$ and $g(\bbV_1,\bbV_2)=0$ thus $\eta^g_{\bbV_1}\bbV_1 \subseteq \bbV_1$. When 
$v_1 \in \bbV_2, v_2 \in \bbV_1, v_3 \in \bbV_2$ formula \eqref{it-expr} reads 
$2g(\eta^g_{v_1}v_2,v_3)=g(\tau(v_3,v_1),v_2)$ by means of $\tau(\bbV_1,\bbV_2)=0$. That $\tau(\bbV_2,\bbV_2) \subseteq \bbV_2$ leads to $\eta^g_{\bbV_2}\bbV_1 \subseteq \bbV_1$. Therefore $\eta^g_{\bbV}\bbV_1 \subseteq \bbV_1$ and the claim on $\bbV_2$ 
follows by orthogonality. Notice that actually one can show that $\eta^g_{\bbV_2}\bbV_1=\eta^g_{\bbV_1}\bbV_2=0$.
\end{proof}
These facts allow proving the main splitting result for Riemannian 
$3$-symmetric algebras in this work, which is summarised below.
\begin{thm} \label{last-split} 
Let $(\g,\sigma)$ be Riemannian $3$-symmetric. Then 
\begin{equation} \label{smain-g}
\ug=\mL_0 \oplus \n \oplus \bigoplus \limits_{i=1}^r(\VV_i \rtimes_{\rho_{i}} \mL_i)
\end{equation}
a direct product of Lie algebras. The factors of the corresponding splitting 
\begin{equation} \label{can-bbV}
\bbV=\mathcal{H}_0 \oplus \n \oplus \bigoplus \limits_{i=1}^r (\VV_i \oplus \mathcal{H}_i)
\end{equation}
are
\begin{itemize}
\item[(i)] $J$-invariant and $g$-orthogonal whenever $g \in \me_{\mathrm{H}}(\bbV)$
\item[(ii)] 
invariant under the intrinsic torsion tensor $\eta^g$ with $g \in \me_{\mathrm{H}}(\bbV)$.
\end{itemize}
\end{thm} 
\begin{proof}
The splitting at the Lie algebra level as well as (i) have been established previously.\\
(ii) since \eqref{smain-g} is a direct product of Lie algebras the subspaces $\mathcal{H}_0, \n, \VV_i \oplus \mathcal{H}_i, 1 \leq i \leq r$ of $\bbV$ satisfy the assumptions in Lemma \ref{ideals} which grants the claim.
\end{proof}
\subsection{Regularity of infinitesimal models for $\ug$ and the proof of Theorem \ref{t1}. }\label{reginfpf1}

Finally we can explain the regularity of each factor coming from Theorem \ref{last-split} in the following manner.
\begin{enumerate} 
\item (Type III and IV)
As a consequence of the above result 
one can canonically associate a Riemannian homogeneous space to the 
transvection algebra $\ug$ above. 

If $\ug=\VV \rtimes_{\rho} \mL$ we proceed as follows. Let $G$ be the simply connected Lie group with Lie algebra $\mL$; from the general theory of symmetric 
spaces it has a canonical involution with connected fixed point set $K$; recall 
that the latter is a compact subgroup in $G$. Because $G$ is simply connected there exists a unique Lie group representation $\pi:G \to GL(\VV)$ 
with $(d\pi)_e=\rho$. The desired homogeneous manifold is $M=(\VV \rtimes_{\pi} G \slash K$. 

\item (Semisimple Case) See Section \ref{fix-s}, where this was already established. 

\item (Type II) This is immediate, as $\uh$ vanishes. 
\end{enumerate}

We can now present  the Proof of Theorem \ref{t1}} as follows.
\proof Consider a connected and simply connected Riemannian $3$-symmetric space $(M,J)$ together with its canonical connection $D$. We indicate with $H$ the holonomy group of $D$ at a fixed point $x \in M$, which is therefore connected. Moreover, consider the infinitesimal model algebra $\g$ at $x$ and split the transvection algebra $\ug$ according to 
\eqref{smain-g}. Pick $g \in \mathrm{Met}(M,J,D)$; because the splitting of $T_xM$ in \eqref{can-bbV} given in 
\eqref{can-bbV} is $\mathrm{H}$ invariant and also intrinsic torsion invariant it is invariant under the connected component of the holonomy group of $g$ computed at $x$. By parallel transport and using deRham's splitting theorem 
we obtain a Riemannian splitting 
$$ (M,g) =(M_0,g_0) \times (N,g_N) \times (M_1, g_1) \times \ldots \times (M_r, g_r).
$$
The factors of this splitting are Riemannian $3$-symmetric spaces with corresponding transvection algebras 
at $x$ isomorphic to $\mL_0, \n$ respectively $\VV_i \rtimes_{\rho_i} \mL_i$ for $1 \leq i \leq r$. 

To conclude, we proceed as follows. Given a simply connected $3$-symmetric space $(M,D)$ it is well known, see \cite{kow}[Proposition II.33], that 
$M=\underline{G} \slash \mathrm{H}$ where the connected Lie group $\underline{G}$ has Lie algebra $\ug$ and 
$\mathrm{H}=\mathrm{Hol}(D)$. In fact $\mathrm{\underline{G}}$ occurs as the holonomy bundle of the linear connection $D$. Therefore we may write $M_i=\mathrm{\underline{G}}^i \slash \mathrm{H}_i$ where the Lie algebra of $\mathrm{\underline{G}}^i$ is isomorphic 
to $\VV_i \rtimes_{\rho_i} \mL_i$, for $1 \le i \leq r$. Exponentiating yields a Lie group isomorphism $\varphi_i :\mathrm{\underline{G}}^i \to \VV_i \rtimes_{\pi_i} \mathrm{G}_i$ such that $\varphi_i(\mathrm{H}_i) \subseteq K_i$; thus 
$M_i$ is diffeomorphic to $(\VV_i \rtimes_{\pi_i} \mathrm{G}_i) \slash K_i$ for $1 \leq i \leq r$  
and the preceding arguments suffice to establish the splitting at manifold level. \endproof


\section{Structure of admissible representations} \label{met-adm}
The aim of this section is to obtain a full classification of admissible 
representations $(\mL,\rho,\VV)$ together with the set of isotropy invariant metrics. 
The factorisation in Theorem \ref{split-lasty} ensures that we may, without loss of generality, assume that $\mL$ is simple. In order to be able to classify admissible 
we first recall  that in Section \ref{admr1} it was proven that $V$ admits background metrics. The fact these are isotropy invariant is then an automatic consequence.  

\subsection{Standard examples and classification} \label{sat}
First we present, in some detail, the examples of admissible representations, starting with the non-compact case. The following conventions will be used below.
The canonical metric on $\mathbb{R}^{2n}$ will be denoted by $g_{2n}$; with respect to 
the canonical splitting $\mathbb{R}^{2p+2q}=\mathbb{R}^{2p} \oplus \mathbb{R}^{2q}$, 
we write $g_{2p,2q}=-g_{2p}+g_{2q}$. Note that for brevity we write $n=p+q$.  The standard complex structure on $\mathbb{R}^{2n}$ is denoted by $J_n$. We also consider $I_{p,q}:=J_p-J_q$ on $\mathbb{R}^{2n}$.

Below we outline the main features of the Hermitian symmetric Lie algebras of relevance for this work.

\begin{itemize}
\item[1)] Let $\mathfrak{sp}(m,\mathbb{R})$ be the Lie subalgebra of $\gl(\mathbb{R}^{2m})$ which preserves  
the symplectic form $\omega_m:=g_{2m}(J_m \cdot, \cdot)$. Then  $\mathfrak{sp}(m,\mathbb{R})=\Kk \oplus \H$ with $\Kk=\mathfrak{u}(m)$ and $\H=\Sym^2\mathbb{R}^{2m} 
\cap \gl^{-}_{J_m}(\mathbb{R}^{2m})$. A normalised generator for $\z(\Kk)$ is 
$z=-\frac{1}{2}J_m$ thus $J_{\H}(S)=S \circ J_m$. Therefore the representation 
$\rho^{\spr}:\spr(m,\bbR) \to \gl_m(\bbR)$ obtained by matrix multiplication is admissible with respect to $J_m$.
\bigskip

\item[2)] We let $\mathfrak{su}(p,q)$ be the Lie subalgebra of $\gl(\mathbb{R}^{2(p+q)})$ which preserves the pair $(g_{2p,2q},J_{p+q})$. 
The isotropy algebra is $\Kk=\mathfrak{su}(p) \oplus \mathfrak{su}(q) \oplus \bbR z$ with center generated by $z=\tfrac{1}{p+q}(pJ_q-qJ_p)$ and where 
$ \mathfrak{su}(p) \oplus \mathfrak{su}(q)$ is embedded diagonally, via 
$(A,B) \mapsto \left(\begin{array}{cc} A & 0 \\ 0 & B \end{array}\right)$. Then we have $\mathfrak{su}(p,q)=\Kk \oplus \H$ where $\H=\{f \in \mathfrak{su}(p,q) : f(\bbR^{2p}) \subseteq \bbR^{2q}\}$. This entails, after also taking into account that $\H \subseteq \mathfrak{gl}_{J_{p+q}}\bbR^{2p+2q}$, that the complex structure $J_{\H}=\ad_z$ reads 
$J_{\H}f=-I_{p,q} \circ f=f \circ I_{p,q}$ whenever $f \in \H$. Since $\H \subseteq \Sym^2(\bbR^{2(p+q)},g_{2p+2q})$ it is easy to conclude that the representation $(\mathfrak{su}(p,q),\bbR^{2p+2q})$ is admissible with respect to $(g_{2(p+q)},I_{p,q})$. 
%
\bigskip

\item[(3)] One of  the most interesting examples on Satake's list are the skew-symmetric representations of $\mathfrak{su}(1,n)$.  Consider the representation 
  $\rho: \mathfrak{su}(1,n)\rightarrow \mathfrak{gl}(\mathbb{R}^{2n+2})$ given by matrix multiplication. We split $\bbR^{2n+2}=\bbR^2 \oplus \bbR^{2n}$ where the factor $\bbR^2=\mathrm{span} \{e_1,e_2=J_2e_1\}$. With respect to $g_{2,n}$ we obtain a linear isomorphism 
  \begin{equation} \label{id-sun}
  \mathfrak{su}(1,n) \cong \mathfrak{su}(n) \oplus \bbR z \oplus \H 
  \end{equation}
 where $z=\tfrac{1}{n+1}(J_n-nJ_2)$ and $\H=\bbR^{2n}$. The isotropy algebra $\mathfrak{k}:=\mathfrak{su}(n) \oplus \bbR z$ acts on $\bbR^{2n+2}$ via matrix multiplication, whereas elements 
 $v \in \H$ act via 
$$v \in \H \mapsto e^1 \otimes v+e^2 \otimes Jv+v^{\flat} \otimes e_1-Jv^{\flat} \otimes e_2.$$ Here the metric dual is taken with respect to $g_{2n}$. The complex structure $J_{\H}=\ad_z$ on $\H$ then reads $J_{\H}=J_n$.

Consider the space $\lambda^p\mathbb{R}^{2n+2}=\{\alpha \in \Lambda^p\bbR^{2n+2} : \alpha(J_{n+1} \cdot ,J_{n+1}\cdot, \ldots, \cdot)=-\alpha\}$ for $p \geq 2$; when $p=1$ we let $\lambda^1\mathbb{R}^{2n+2}:=\Lambda^1\bbR^{2n+2}.$ These spaces admit an $\mathfrak{su}(1,n)$-invariant complex structure, still to be noted by $J_{n+1}$ and given by $\alpha \in \lambda^p\mathbb{R}^{2n+2} \mapsto J_{n+1}\alpha:=\alpha(J_{n+1} \cdot, \ldots, \cdot)$. We have a linear isomorphism 
\begin{equation} \label{ext-split}
\lambda^p\mathbb{R}^{2n+2} \cong \lambda^{p}\mathbb{R}^{2n} \oplus \lambda^{p-1}\mathbb{R}^{2n}
\end{equation}
where the first factor is embedded via inclusion. The second factor is embedded according to $\gamma \in \lambda^{p-1}\mathbb{R}^{2n}\mapsto 
e^1 \wedge \gamma+e^2 \wedge J_n\gamma$.
Note that the canonical form inner product on $\lambda^p\mathbb{R}^{2n+2}$ becomes $g \oplus 2g$, after the identification in \eqref{ext-split}. Here we indicate with $g$, by a slight abuse of notation, the canonical form inner product on $\lambda^{\star}\bbR^{2n}$ induced by $g_{2n}$.

Let now $\rho$ be the representation by derivations of $\mathfrak{su}(1,n)$ on $\VV:=\lambda^p\mathbb{R}^{2n+2}$.
Taking into account the identification in \eqref{id-sun}, a straightforward computation shows that with respect to \eqref{ext-split} we have
\begin{equation*}
\begin{split}
\rho(v)\beta=&v \lrcorner \beta \ \mathrm{when} \ \beta \in \lambda^{p}\mathbb{R}^{2n}\\
\rho(v) \gamma=&v^{\flat} \wedge \gamma-J_nv^{\flat} \wedge J_n \gamma \ \mathrm{when} \ \gamma \in \lambda^{p-1}\mathbb{R}^{2n}
\end{split}
\end{equation*}
and whenever $v \in \mathcal{H}=\mathbb{R}^{2n}$; in particular $\rho(\mathcal{H}) \subseteq \Sym^2(V, g\oplus 2g)$. Similarly, with respect to \eqref{ext-split} we have 
$$ \rho(A)(\beta+\gamma)=[A,\beta]+[A,\gamma] \ \mathrm{and} \ \rho(z)(\beta+\gamma)=\tfrac{p}{n+1}J_n\beta+\tfrac{n+1-p}{n+1}J_n\gamma
$$
where $[A,\cdot]$ is the standard action by derivations on $\lambda^{\star}\bbR^{2n}$; in particular we have $\rho(\mathfrak{k}) \subseteq \mathfrak{so}(\VV, g \oplus 2g)$. Note $\rho$ is faithful if and only if $p \leq n$.

We define a linear complex structure $J_{\VV}: \VV \to \VV$ 
$$ J_{\VV}=J_n \ \mathrm{on} \ \lambda^p\bbR^{2n} \ \mathrm{respectively} \ J_{\VV}=-J_n \ \mathrm{on} \ \lambda^{p-1}\bbR^{2n}
$$
with respect to the splitting \eqref{ext-split} 
It follows that $\rho(J_nv)=\rho(v) \circ J_{\VV}=-J_{\VV} \circ \rho(v)$ for all $v \in \bbR^{2n}$. Also, we have 
$\rho(\mathfrak{k}) \subseteq \gl_{J_{\VV}}\VV$. 

Summarising, we have showed that the representation 
$(\mathfrak{su}(1,n),\rho, \lambda^p\bbR^{2n+2})$ where $1 \leq p \leq n$ is admissible with respect to $J_{\VV}$. Further recording that $\rho$ commutes with $J_{n+1}$ shows that this representation is an example of an admissible representation of complex type.

\bigskip

\item[4)] Let $\so^{\star}(2m):=\{f \in \mathfrak{su}(m,m):[f,I^1]=0\}$. Here 
$I^1(v_1,v_2)=(-v_2,v_1)$ with respect to $\bbR^{4m}=\bbR^{2m} \oplus \bbR^{2m}$. 
Note that $(I_{2m,2m},I^1,I^2:=I_{2m,2m}\circ I^1)$ is a quaternionic triple. This description of $\so^{\star}(2m)$ agrees with the duality principle in Proposition \ref{symp1-c} in the next section, which also ensures this representation is admissible.

\bigskip

\item[5)] Let $\so(2,n) \subseteq \gl(\bbR^{n+2})$ consist of endomorphisms  which are skew-symmetric with respect to $g_2-g_n$. Consider the splitting $\bbR^{n+2}=\bbR^2 \oplus \bbR^n$ where $\bbR^2$ is spanned by $e_1,e_2$ with $g_2(e_i,e_j)=\delta_{ij}$. With respect to this splitting we have $\so(2,n)=\Kk \oplus \H$ as follows. The isotropy algebra $\Kk=\bbR z \oplus \so(n)$ where $z(e_1)=e_2, z(e_2)=-e_1$ and $z$ acts trivially on $\bbR^n$; moreover $\so(n)$ preserves $\bbR^n$ and acts trivially on $\bbR^2$. We identify $\bbR^{2n}$ and $\H$ via 
$$(x,y) \mapsto g_n(x, \cdot) \otimes e_1+g_2(e_1, \cdot) \otimes x+ g_n(y, \cdot) \otimes e_2+g_2(e_2, \cdot) \otimes y.$$
The complex structure $J_{\H}=(\ad_z)_{\vert_{\H}}$ thus reads $J_{\H}(x,y)=(-y,x)$ and the Lie bracket within $\H$ is given by 
$$ [(x_1,y_1),(x_2,y_2)]=x_2 \wedge x_1+y_2 \wedge y_1+\left (g_n(x_2,y_1)-g_n(x_1,y_2) \right) z.
$$
%
\end{itemize}
Below we briefly outline the construction of admissible representations of $\so(2,n)$. We show those are in bijective correspondence 
with Clifford algebra representations $\rho_n:Cl_{0,n} \to \GL(\Sigma)$, thus satisfying 
$\rho_n(x) \circ \rho_n(x)=g_n(x,x)1_{\Sigma}$ for all $x \in \bbR^{n}$.
\begin{lemma} \label{spin-explt}
Consider the symmetric pair $(\so(2,n),\bbR z \oplus \so(n) )$.
Any admissible representation $(\rho,\Sigma)$ of this pair is uniquely determined from a Clifford algebra representation $\rho_n:Cl_{0,n} \to \GL(\Sigma)$ such that 
\begin{equation} \label{spin-pp}
\rho_n(x)J_{\Sigma}+J_{\Sigma} \rho_n(x)=0 \ \mathrm{for \ all \ } \ x \in \bbR^n 
\end{equation}
where $J_{\Sigma}$ is a linear complex structure on $\Sigma$. Explicitly 
\begin{equation*}
\begin{split}
& \rho(x,y)=\frac{1}{2}(\rho_n(x)+\rho_n(y)J_{\Sigma}), \ 
\rho(z)=-\frac{1}{2}J_{\Sigma}, \ \rho(x \wedge y)=-\frac{1}{2}\rho_n(x \wedge y)
\end{split}
\end{equation*}
whenever $(x,y) \in \H=\bbR^n \oplus \bbR^n$. 
\end{lemma}
\begin{proof}
This has been also proved in \cite{Sat1}, below we give a simple direct proof. Assume that $\rho$ is admissible with respect to 
$J_{\Sigma}$ and  pick $x \in \bbR^n$. Since $J_{\H}(x,0)=(0,x)$, from 
$[J_{\H}(x,0), (x,0)]=g_n(x,x)z$ and \eqref{jordan} we get 
\begin{equation} \label{cliff-m}
\rho_n(x) \circ \rho_n(x)=g_n(x,x)2J_{\Sigma}\rho(z)
\end{equation}
after setting $\rho_n(x)=2\rho(x,0)$. Now consider the spaces $\VV_{\lambda}, \lambda \in \bbR$, see Lemma \ref{rho-z} and its proof. If 
$\lambda=\frac{1}{2}$ we have $\VV_{\lambda}=0$ by Lemma \ref{rho-z}; otherwise the map $\rho_n : \VV_{\lambda} \to \VV_{-\lambda}$ is injective by \eqref{cliff-m}, so $\dim \VV_{\lambda}=\dim \VV_{-\lambda}$. Then Lemma \ref{rho-z} ensures that $\lambda=0$, hence $\rho(z)+\tfrac{1}{2}J_{\Sigma}=0$. Furthermore, we also get $\rho(0,x)=\rho J_{\H}(x,0)=\tfrac{1}{2}\rho_n(x)J_{\Sigma}$, because $\rho$ is admissible. 
Therefore $\rho_n$ extends to a representation of $Cl_{0,n}$ satisfying \eqref{spin-pp}. The expression of $\rho$ on $\so(n)$ thus follows from the Lie bracket structure in $\so(2,n)$.
\end{proof}
See \cite{LM} or \cite{Sat1} for details concerning Clifford algebra representations satisfying \eqref{spin-pp}.
Now consider the following table. 
{\tiny{
 \begin{table}[h!]
\centering
\begin{tabular}{||c | c | c | c||} 
 \hline
 L & $\Kk$ & $\rho$ & $\VV$ \\ 
 [0.5ex] 
 \hline\hline
 $\mathfrak{sp}(m,\mathbb{R})$  & $\mathfrak{u}(m)$ & $d \rho^{\spr}$ & $d \bbR^{2m}$\\ 
 $\mathfrak{su}(p,q), p \neq 1$ or $q \neq 1$ & $\mathbb{R}z \oplus \mathfrak{su}(p) \oplus \mathfrak{su}(q)$ & $d  \rho_1$& $d  \bbR^{2p+2q}$\\
 $\mathfrak{su}(p,1), p\geq 2$ & $\mathbb{R}z\oplus \mathfrak{su}(p)$  & $\oplus_{i=1}^{l}(d_i  \rho_i)$& $\oplus_{i=1}^l d_i  \lambda^i\bbR^{2p+2}$  \\
 $\so^{\star}(2m)$ & $\mathfrak{u}(m) $ & $d  \rho^{\so^{\star}}$ & $d  \bbR^{4m}$\\
$\mathfrak{so}(2,n)$ & $\mathfrak{so}(2)\oplus \so(n)$ & representations of $Cl_{0,n}$ as in \eqref{spin-pp}
& $\Sigma$\\
 [1ex]
    \hline
\end{tabular}
\caption{Admissible representations with $\mL$ of non-compact type.}
\label{table1}
\end{table}
}}
\begin{thm} \label{claas-nc}\cite{Sat1}
Let  $\mL$ be simple of non-compact type.  Then $(\mL, \rho,\VV)$ is an admissible representation if, and only if, it is one of the entries in Table \ref{table1}.
\end{thm}
\subsection{When $\mL$ has compact type} \label{type1-reps}
We finish this section by observing that describing admissible representations with $\mL$ of compact type 
is equivalent, via the duality principle for symmetric spaces, with the non-compact type situation(see also Remark \ref{dual0} below). We proceed in a slightly different way, in order to keep track on the representation. If $\mL=\Kk \oplus \mathcal{H}$ is a symmetric Lie algebra recall that its dual Lie algebra $(\mL^{\star}, [\cdot, \cdot]^{\star}$ is defined by requiring $\mL^{\star}=\mL$ as vector spaces and the Lie bracket 
\begin{equation*}
\begin{split}
&[F,G]^{\star}=[F,G], \ [F,x]^{\star}=[F,x]\\
&[x,y]^{\star}=-[x,y].
\end{split}
\end{equation*}
Straightforward algebraic computation leads to the proof of the following  
\begin{propn} \label{symp1-c} 
Let $(\mL,\rho,\VV)$ be a linear representation where $\mL=\Kk \oplus \mathcal{H}$ is a simple symmetric Lie algebra. The following hold
\begin{itemize}
\item[(i)] the linear map $\rho^{\star}: \mL^{\star} \to \gl_{I}(\VV \oplus \VV)$ given by 
\begin{equation*}
\rho^{\star}(F)=\biggr ( \begin{array}{lr} 
\rho(F) & 0 \\ 
0 & \rho(F)
\end{array} \biggl ) \ \mbox{if} \ F \in \Kk \ \mbox{respectively} \  \rho^{\star}(x)=\biggr ( \begin{array}{lr} 
0 & \rho(x) \\ 
-\rho(x) & 0
\end{array} \biggl ) \ \mbox{if} \ x \in \mathcal{H}
\end{equation*}
is a well defined Lie algebra representation where $I(v_1,v_2)=(-v_2,v_1)$
\item[(ii)] if $h_{\VV}$ is a background metric for $(\mL,\rho,\VV)$ then $h_{\VV} \oplus h_{\VV}$ is a background metric for 
$(\mL^{\star},\rho^{\star},\VV \oplus \VV)$
\item[(iii)] if $(\mL,\rho,\VV)$ is admissible with respect to $J_{\VV}$ then $(\mL^{\star}, \rho^{\star},\VV \oplus \VV)$ is admissible with respect to $J_{\VV} \oplus J_{\VV}$.
\end{itemize}
\end{propn}
\begin{rem} \label{dual0}
What is really underpinned by the explicit construction for $\rho^{\star}$ is the Lie algebra embedding $\so^{\star}(2m) \to \mathfrak{sp}(m,\bbR)$.
\end{rem}

\begin{rem} \label{typeI-int}
From the point of view of submanifold geometry admissible and irreducible representations of simple Lie algebras of compact type correspond to totally geodesic, complex submanifolds of the Hermitian symmetric space $\frac{SO(2m)}{U(m)}$. Such objects have been classified by Ihara 
in \cite{Iha,Iha2}, although those references do not exhibit an explicit list.
\end{rem}
We recall now the construction of the following symmetric pairs of compact type.
\begin{itemize}
\item[$\bullet$] the pair $(\mathfrak{sp}(m),\mathfrak{u}(m))$. Here $\mathfrak{sp}(m)$ is the Lie subalgebra of $\so(4m)$ preserving the quaternion structure $\{ I, I_1(v_1,v_2)=(J_nv_2,J_nv_1), I_2=II_1\}$ constructed with respect to the splitting $\bbR^{4m}=\bbR^{2m} \oplus \bbR^{2m}$. The isotropy algebra $\mathfrak{u}(m)$ is realised as the subalgebra of $\mathfrak{sp}(m)$ preserving $J(v_1,v_2) \:=(J_nv_1,J_nv_2)$. The matrix representation $(\mathfrak{sp}(m),\bbR^{4m})$ is then admissible with respect to $J$.
\item[$\bullet$] the pair $(\so(n+2), \Kk=\so(n) \oplus \so(2))$. This is constructed in complete analogy with example 5 in subsection \ref{sat}. With respect to the splitting $\bbR^{n+2}=\bbR^n \oplus \bbR^2$ we have $\so(n+2)=\bbR z \oplus \so(n) \oplus \H$ with 
$\H \cong \bbR^{2n}$. The complex structure reads $J_{\H}(x,y)=(-y,x)$ whilst the Lie bracket 
$$ [(x_1,y_1), (x_2,y_2)]=x_1 \wedge y_1+x_2 \wedge y_2+\left (g_n(x_1,y_2)-g_n(x_2,y_1) \right )z.
$$
\end{itemize}

Now consider the following table. 
{\tiny{
 \begin{table}[h!]
\centering
\begin{tabular}{||c | c | c | c||} 
 \hline
 L & $\Kk$ & $\rho$ & $\VV$ \\ 
 [0.5ex] 
 \hline\hline
 $\mathfrak{sp}(m)$  & $\mathfrak{u}(m)$ & $d \rho^{\spr}$ & $d \bbR^{4m}$\\ 
 $\mathfrak{su}(p+q), p \neq 1$ or $q \neq 1$ & $\mathbb{R}z \oplus \mathfrak{su}(p) \oplus \mathfrak{su}(q)$ & $d  \rho_1$& $d  \bbR^{2p+2q}$\\
 $\mathfrak{su}(p+1), p\geq 2$ & $\mathbb{R}z\oplus \mathfrak{su}(p)$  & $\oplus_{i=1}^{l}(d_i  \rho_i)$& $\oplus_{i=1}^l d_i  \lambda^i\bbR^{2p+2}$  \\
 $\so(2m)$ & $\mathfrak{u}(m) $ & $d  \rho^{\so}$ & $d  \bbR^{2m}$\\
$\mathfrak{so}(n+2)$ & $\mathfrak{so}(2)\oplus \so(n)$ & representations of $Cl_{n}$ as in \eqref{spin-pp}
& $\Sigma$\\
 [1ex]
    \hline
\end{tabular}
\caption{Admissible representations with $\mL$ of compact type.}
\label{table2}
\end{table}
}}
\begin{rem} \label{rep-com}
The skew-symmetric power representation of $\mathfrak{su}(p+1)$ are constructed in complete analogy with those in Example 3 from subsection \ref{sat} and still denoted by $\rho_i$ in Table 2. The representations $\rho^{\mathfrak{sp}}, \rho_1, \rho^{\so}$ 
indicate the matrix algebra representations of $\mathfrak{sp}(m), \mathfrak{su}(p+q), \so(2m)$.

Analogously to Lemma \ref{spin-explt} admissible representation $\rho:\so(n+2) \to \gl(\Sigma)$ are constructed from Clifford algebra representations $\rho_n:Cl_n \to \GL(\Sigma)$ satisfying \eqref{spin-pp} via 
$$ \rho(x,y)=\frac{1}{2} \left (\rho_n(x)+\rho_n(y)J_{\Sigma} \right ), \ \rho(z)=-\frac{1}{2}J_{\Sigma}, \ \rho(x \wedge y)=\frac{1}{2}
\rho_n(x \wedge y).
$$
\end{rem}
\begin{thm} \label{claas-c}
Let  $\mL$ be simple of compact type.  Then $(\mL, \rho,\VV)$ is an admissible representation if, and only if, it is one of the entries in Table \ref{table2}.
\end{thm}
\begin{proof}
This follows directly from the duality principle in Proposition \ref{symp1-c} and Theorem \ref{claas-nc}.
\end{proof}

\section{Computing ${}^g\og$ \label{der-tv} and the proof of Theorem \ref{t1.5}}

We continue to follow the convention that  $(\g,\sigma,g)$ is Riemannian  $3$-symmetric.
Theorem \ref{last-split} clarifies the structure of the transvection Lie algebra $\ug$. To complete the picture we also need to describe 
$\g$ in the terms of the data $\ug$ is build from. 

Our first objective is the computation of the Lie algebra $\oh$ defined in \eqref{sandwich} together with the Nomizu algebra defined in \eqref{Nom}. To compute ${}^g\oh$ intrinsically in terms of $\ug$ we consider the subalgebra 
\begin{equation*}
\Der_{\uh}(\ug):=\{f \in \Der(\ug): f(\uh) \subseteq \uh\}
\end{equation*} 
of $\Der(\ug)$. 

Identify $\uh$ with a subalgebra of $\so(\bbV)$ by means of the isotropy representation 
$(\iota,\uh,\bbV)$. Whenever $F \in {}^g\oh$ consider the linear map $\varphi_F:\ug \to \ug, \varphi_F=[F,\cdot]$ where the Lie bracket is taken within $ {}^g\og$. Because $\ug$ is an ideal in $\og$ and $\uh \subseteq \so(\bbV,g)$ it follows that $\varphi_F 
\in \Der_{\uh}(\ug) \cap \so(\ug,\ung)$. Here $\ung$ is the inner product on $\ug$ given by $\ung=-\underline{\BK}_{\vert \uh}+g$. Also note that, explicitly, 
$\varphi_Fv=Fv$ and $ \varphi_F(G)=F \circ G-G \circ F$.
\begin{propn} \label{oh-gen} 
The map $ {}^g\oh \to \Der_{\uh}(\ug) \cap \so(\ug,\ung)$ given by 
$\varphi \mapsto \varphi_F$ is a Lie algebra isomorphism.
\end{propn}
\begin{proof}
We know that $\Der(\ug) \subseteq \so(\ug,\underline{\BK})$, that $\underline{B}$ is non-degenerate on $\uh$ and also that $\underline{\BK}(\uh,\bbV)=0$; it follows that any 
$\varphi \in \Der_{\uh}(\ug)$ satisfies $\varphi\bbV \subseteq \bbV$ as well. Writing $f:=\varphi_{\vert \bbV} \in \so(\bbV,g)$ we have 
$$ \varphi[G,v]=[\varphi G,v]+[G,\varphi v]$$
whenever $G \in \uh, v \in \bbV$. It follows that 
\begin{equation} \label{dG}
\varphi G=f \circ G-G \circ f.
\end{equation}
Identifying the components of the equation 
$\varphi [v_1,v_2]=[\varphi v_1,v_2]+[v_1,\varphi v_2]$ on $\uh$, respectively $\bbV$ we find 
\begin{equation*}
\begin{split}
\varphi R^D(v_1,v_2)=&R^D(fv_1,v_2)+R^D(v_1,fv_2)\\
f\tau^D(v_1,v_2)=&\tau^D(fv_1,v_2)+\tau^D(v_1,fv_2)
\end{split}
\end{equation*}
The second equation means that $f$ preserves $\tau^D$; the first, combined with 
\eqref{dG},yields that $f$ preserves $R^D$ as well, thus $f \in \oh$.
\end{proof}
We also record the following elementary(see e.g. \cite{sato}, Lemma 4)
\begin{lemma} \label{der-p}
Let $\g_1,\g_2$ be Lie algebra such that $\z(\g_2)=0$ and $[\g_2,\g_2]=\g_2$. Then 
$\Der(\g_1 \oplus \g_2)=\Der(\g_1) \oplus \Der(\g_2)$, a direct sum of Lie algebras.
\end{lemma}  
The structure of the Nomizu algebra $\oh$ is now a direct consequence of these facts and of the splitting Theorem \ref{last-split}. We indicate with  $\Der(\n,g_{\n}):=\Der(\n) \cap \so(\n,g_{\n})$ as well as $\mathfrak{c}(\mL,\rho,g_{\VV}):=\mathfrak{c}(\mL,\rho) \cap \so(\VV,g_{\VV})$.
\begin{propn} \label{nomizu-1}
Let $(\g,\sigma,g)$ be Riemannian $3$-symmetric and split 
\begin{equation*}
\ug=\mL_0 \oplus \n \oplus (\VV \rtimes_{\rho}\mL).
\end{equation*}
We have 
\begin{equation} \label{nom-1}
 {}^g\oh=\uh \oplus \Der(\n,g_{\n}) \oplus \mathfrak{c}(\mL,\rho,g_{\VV})
\end{equation} as well as 
\begin{equation} \label{nom-2}
 {}^g\og \cong \mL_0 \oplus (\n \rtimes \Der(\n,g_{\n})) \oplus (\VV \rtimes_{\rho}(\mL \oplus \mathfrak{c}(\mL,\VV,g_{\VV})))
\end{equation}
as direct sums of Lie algebras. 

\end{propn}
\begin{proof}

The direct product Lie algebra $\g_2:=\mL_0 \oplus (\VV \rtimes_{\rho} \mL)$ satisfies 
$\z(\g_2)=0$ and $[\g_2,\g_2]=\g_2$ by taking Lemma \ref{cen-rad-p},(i) into account. Using Lemma \ref{der-p} this fact leads to 
\begin{equation*}
\begin{split}
\Der(\ug)=&\Der(\n \oplus \g_2)=\Der(\mL_0) \oplus \Der(\n) \oplus \Der(\VV \rtimes_{\rho} \mL),
\end{split}
\end{equation*}
a direct sum of Lie algebras. Because $\uh=\Kk_0 \oplus \Kk$ is a direct sum of Lie algebras and $\uh$ acts trivially on $\n$ we get further that 
\begin{equation*}
\Der_{\uh}(\ug)=\Der_{\Kk_0}(\mL_0) \oplus \Der(\n) \oplus \Der_{\Kk}(\VV \rtimes_{\rho} \mL).
\end{equation*}
It is a simple exercise to see that $\ad:\Kk_0 \to \Der_{\Kk_0}(\mL_0)$ is a Lie algebra isomorphism. From \eqref{der-gen2} combined with $\VV^{\Kk}=\{0\}$ (see proof of Lemma \ref{fix-2}) it follows that the map $\Kk \oplus \mathfrak{c}(\mL,\rho) \to \Der_{\Kk}(\VV \rtimes_{\rho} \mL)$ 
given by $(k,g) \mapsto D$, where $D_{\vert \VV}=\rho(k)+g, \ D_{\vert \mL}=ad_{k}$ is a Lie algebra isomorphism. Thus we have a Lie algebra 
isomorphism $\Der_{\uh}(\ug) \cong \Kk_0 \oplus \Der(\n) \oplus (\Kk \oplus \mathfrak{c}(\mL,\rho))$. The claim in \eqref{nom-1} is now proved by intersecting with $\so(\ug,\ung)$. Equation \eqref{nom-2} follows from the definition of the Lie algebra structure in $\og=\oh \oplus \bbV$ and \eqref{nom-1}. 

\end{proof}
The isotropy algebra $\h$ of a Riemannian $3$-symmetric Lie algebra $(\g,\sigma,g) $ satisfies 
\begin{equation}\label{complexnomizu}
\uh \subseteq \h \subseteq  {}^g\oh_J:=\oh \cap \gl_J(\bbV).
\end{equation}
The subalgebra $ {}^g\og_J:= {}^g\oh_J \oplus \bbV$ of the Nomizu algebra $ {}^g\og$  inherits a natural $3$-symmetric structure from $\ug$; however $ {}^g\og$ does not, unless $ {}^g\oh \cap \gl_J(\bbV)=\oh $.

The next step in order to determine the freedom in choosing $\h$ is to 
compute the intersection $ {}^g\oh \cap \gl_{J}(\bbV)$. This is done below by considering the nilpotent and the semidirect product situations separately. 

\begin{lemma} \label{cent-cpx}
Let $(\mL,\rho,\VV)$ be admissible, where $\mL$ is simple.  
\begin{itemize}
\item[(i)] consider the operator $\mathscr{C}:=\sum \limits_{k}\rho(e_k) \circ \rho(e_k) :\VV \to \VV$ where $\{e_k\}$ is some $g_{\mathcal{H}}$-orthonormal 
basis in $\mathcal{H}$. Then 
\begin{equation*}
\mathscr{C}=\frac{\Lambda}{2}J_{\VV} \rho(z)
\end{equation*}
where the $\Lambda \in \bbR$ is determined from $\BK^{\mL}_{\vert \mathcal{H}}=\Lambda g_{\mathcal{H}}$
\item[(ii)] both $\mathscr{C}$ and $\rho(z)$ are invertible
\item[(iii)] we have $\mathfrak{c}(\mL,\rho,\VV) \subseteq \gl_{J_{\VV}}(\VV)$.
\end{itemize}
\end{lemma}
\begin{proof}
(i) We first note that $\mathscr{C}$ can be viewed as a trace-type operator with respect to 
the metric $g_{\mathcal{H}}$, so its definition is basis independent. By \eqref{jordan} we have 
$$2\rho(e_k) \circ \rho(e_k)=-J_V \rho([e_k,Je_k]).$$
At the same time 
it is easy to see 
that $\sum \limits_{k}[e_k,Je_k]=-\Lambda z$ and the claim follows.\\
(ii) 
let $h_{\VV}$ be a background metric for $(\mL,\rho,\VV)$; since $\rho(\mathcal{H}) \subseteq 
\so(\VV,h_{\VV})$ when $\mL$ has compact type, respectively $\rho(\mathcal{H}) \subseteq \Sym^2(\VV,h_{\VV})$ when $\mL$ has non-compact type, a positivity argument shows that 
$\ker(\mathscr{C})=\{v \in \VV:\rho(\mathcal{H})v=0\}$; the latter space clearly vanishes as $[\mathcal{H},\mathcal{H}]=\Kk$ and $\VV^{\mL}=\{0\}$. Thus $\mathscr{C}$ is invertible and so is $\rho(z)$ by (i).\\
(iii) pick $F \in \mathfrak{c}(\mL,\rho,\VV)$; certainly $[F,\rho(z)]=0$. Since 
$[F,\rho(x)]=0, x \in \mathcal{H}$ we get that $[F, \mathscr{C}]=0$. By using (i) this yields $[F,\rho(z)J_{\VV}]=0$ and further $\rho(z)[F,J_{\VV}]=0$. But $\rho(z)$ is invertible by (ii), hence $[F,J_{\VV}]=0$.
\end{proof}

Whenever $\n$ is a $3$-symmetric nilpotent Lie algebra(hence $[[\n,\n],\n]=0$) we let $\Der(\n,g_{\n},J_{\n}):=\Der(\n) \cap \un(\n,g_{\n},J_{\n})$, whenever $g_{\n}$ is a scalar product on $\n$, compatible with $J_{\n}$.
\begin{thm} \label{thm-gen}
Let $(\g,\sigma,g)$ be Riemannian $3$-symmetric. Then 
\begin{equation*}
\g \cong \ug \rtimes \mathfrak{s}
\end{equation*}
for some subalgebra $\mathfrak{s} \subseteq \Der(\n,g_{\n},J_{\n}) \oplus 
\mathfrak{c}(\mL,\rho,g_{\VV})$.
\end{thm}
\begin{proof}
Consider the bilinear form $b:\oh_J \times \oh_J \to \bbR, b(A,B):=-\Tr(A \circ B)$. 
This is positive definite since $\oh_J \subseteq \mathfrak{u}(\bbV,g,J)$.
Split $\h=\uh \oplus \mathfrak{s}$ orthogonally with respect to $b$. Because $\uh$ is an ideal in $\h$ and $b$ is $\oh_J$-invariant we get $[\uh,\mathfrak{s}]=0$ and $[\mathfrak{s},\mathfrak{s}] \subseteq \mathfrak{s}$. To determine $\oh_J$ we recall 
that 
\begin{equation*}
\Der(\n,g_{\n})=\Der(\n,g_{\n},J_{\n}) \oplus \{f \in \mathfrak{u}^{\perp}(\n,g_{\n},J_{\n}):f[\n,\n]=0, f(\n)\subseteq \z(\n)\}.
\end{equation*}
as granted by Lemma \ref{DER}. Using Lemma \ref{cent-cpx}, (iii) and \eqref{nom-1} thus yields the equality 
$\oh_J=\uh \oplus \Der(\n,g_{\n},J_{\n}) \oplus 
\mathfrak{c}(\mL,\rho,g_{\VV})$. Since this splitting is orthogonal with respect to the form $b$ it follows that $\mathfrak{s} \subseteq \Der(\n,g_{\n},J_{\n}) \oplus \mathfrak{c}(\mL,\rho,g_V)$. Then $\g=\ug \oplus \mathfrak{s}$ with $\ug$ an ideal and $\mathfrak{s}$ a subalgebra therefore 
$\g=\ug \rtimes_{r} \mathfrak{s}$ where the representation $r:\mathfrak{s} \to \Der(\ug)$ is the restriction of the canonical inclusion 
$\Der(\n,g_{\n},J_{\n}) \oplus \mathfrak{c}(\mL,\rho,g_{\VV}) \subseteq \Der(\ug)$.
\end{proof}
Note that $\mL_0$ splits off $ {}^g\ug \rtimes \mathfrak{s}$ as a direct product factor.
For consistency, also note that the maximal choice $\mathfrak{s}=\Der(\n,g_{\n},J_{\n}) \oplus \mathfrak{c}(\mL,\rho,g_{\VV})$ corresponds to the decomposition in \eqref{nom-2}.

\subsection {Proof of Theorem \ref{t1.5}}
The precise structure of $\g$ is determined in Theorem \ref{thm-gen}, specializing to the case of spaces of type III. The fact we can always obtain a corresponding Lie group action  has been shown in Section \ref{reginfpf1} 
for the transvection algebra and in Section 2 for the general case. To show we have computed all possible presentations as a Riemannian 3-symmetric space, we also need that ${}^g\overline{\g}_{J}$ is the full isometry algebra $\aut(M,g)$ of the Riemannian metric. This will be established in Theorem \ref{split-iso} \endproof
\section{Isotropy invariant metrics} \label{prep-1}
Let $(\mL,\rho,\VV)$ be an admissible representation in the sense of Definition \ref{adm-ff}.We consider the connected Lie group $G$ with algebra $\mL$ and connected compact subgroup $K$. We also also exponentiate the representation $\rho$ to $\pi:G \to \GL(\VV)$ such 
$(d \pi)_e=\rho$ and consider the homogeneous manifold $M:=(\VV \rtimes_{\pi} G) \slash K$.

The aim of this section is to determine the moduli space of $\VV \rtimes_{\pi} G$-invariant Riemannian metrics on $M$. From the general theory of homogeneous Riemannian metrics, such metrics are fully determined by elements in $\me_{K}(\bbV)$. In turn, the latter space is fully described (see Proposition \ref{fix-2}, (iii)) by the set 
$\me_{\Kk}(\VV)$ of isotropy invariant metrics on $\VV$ via the bijection
\begin{equation*}
\me_{K}(\VV) \times \bbR_{>0}\to \me_{K}(\bbV), \ (g_{\VV},t) \mapsto g. 
\end{equation*}
Here $g$ is the Riemannian metric on $M$ obtained from the $\Ad(K)$-invariant metric $g_{\VV}+tg_{\mathcal{H}}$ on $\bbV=\VV \oplus \mathcal{H}$, with respect to the reductive decomposition $\VV \rtimes_{\rho} \mL=\Kk \oplus \bbV$. Also recall that $g_{\mathcal{H}}=\BK^{\mL}|_{\mathcal{H}}$; we set the parameter $t$ to 1 in what follows. This is done without loss of generality since  we will determine all compatible metrics $g_{\VV}$ and we work up to homothety. 

We fix once and for all a background metric $h_{\VV}$ on $\VV$, by using Proposition \ref{exist-b}, and record that the map 
\begin{equation*}
g_{\VV} \in \me_{K}(\VV)\mapsto h_{\VV}^{-1} \circ g_{\VV} \in \{S \in \Sym^2_{\Kk}(\VV,h_{\VV}): SJ_{\VV}=J_{\VV}S, S>0\}
\end{equation*}
is bijective. The first step towards parametrising $\me_{K}(\VV)$ 
is thus the computation of the $\Kk$-module $\Sym^2_{\Kk}(\VV)$. We denote by
$$\mathfrak{c}^{+}(\mL,\rho,h_V)=\mathfrak{c}(\mL,\rho,h_{\VV}) \cap \Sym^2(\VV,h_{\VV}), \ \mathfrak{c}^{-}(\mL,\rho,h_{\VV})=\mathfrak{c}(\mL,\rho,h_{\VV}) \cap \so(\VV,h_{\VV})$$
the symmetric respectively the skew-symmetric parts of the centraliser $\mathfrak{c}(\mL,\rho,\VV)$, computed with respect to the background metric $h_{\VV}$.
\begin{propn} \label{iso-invm}
Let $(\mL,\rho,\VV)$ be an admissible representation where $\mL$ is simple. Then 
$$\Sym^2_{\Kk}(\VV,h_{\VV}) \cap \gl_{J_{\VV}}(\VV)=\{A^{+}+A^{-}J_{\VV}: A^{\pm} \in 
\mathfrak{c}^{\pm}(\mL,\rho,\VV) \}.$$
\end{propn}
\begin{proof} 
Pick $S \in \Sym^2_{\Kk}(\VV,h_{\VV})$ such that $SJ_{\VV}=J_{\VV}S$. From \eqref{jordan} we get 
\begin{equation} \label{git-s}
\rho(x)\rho(y)S=S\rho(x)\rho(y)
\end{equation}
for all $x,y \in \mathcal{H}$. Consider the operator $S_1:\VV \to \VV$ given by 
\begin{equation*}
S_1:=\sum \limits_{i} \rho(e_i) \circ S \circ \rho(e_i)
\end{equation*}
where $\{e_i\}$ is a $g_{\mathcal{H}}$-orthonormal basis in $\mathcal{H}$. It can be viewed as a trace-type operator by using the metric $g_{\mathcal{H}}$ thus its definition is independent of basis choice. We have $S_1 \in \Sym^2(\VV,h_{\VV}) \cap \gl_{J_{\VV}}(\VV)$
by taking into account that $\rho(\mathcal{H}) \subseteq \Sym^2(\VV,h_{\VV}) \cap \gl^{\perp}_{J_{\VV}}(\VV)$ respectively $\rho(\mathcal{H}) \subseteq \so(\VV,h_{\VV}) \cap \gl^{\perp}_{J_{\VV}}(\VV)$ when $\mL$ is non-compact respectively compact. Take 
$x=e_i$ in \eqref{git-s} and compose to the left with $\rho(e_i)$; after summation 
\begin{equation} \label{git-s1}
\mathscr{C}\rho(y)S=S_1 \rho(y).
\end{equation}  
Take $y=e_i$ and compose to the right with $\rho(e_i)$; after summation it follows that 
\begin{equation} \label{git-s2}
\mathscr{C}S_1=S_1\mathscr{C}.
\end{equation}
Since $\mathscr{C}$ is invertible(see Lemma \ref{cent-cpx}, (ii)) the operator $S_2=\mathscr{C}^{-1}S_1$ 
is symmetric by \eqref{git-s2}, commutes with $J_{\VV}$, and satisfies 
\begin{equation*}
\rho(y)S=S_2 \rho(y)
\end{equation*}
by \eqref{git-s1}. Taking the transpose we also get $\rho(y)S_2=S \rho(y)$. Taking linear combinations and using that $\rho(y)J_{\VV}+J_{\VV}\rho(y)=0$ it follows that the endomorphisms given by 
$A^{+}=\frac{1}{2}(S+S_2)$ and $A^{-}=-\frac{1}{2}(S-S_2)J$ satisfy $[\rho(\mathcal{H}),A^{+}]=[\rho(\mathcal{H}),A^{-}]=0$. Thus $[A^{+},\rho(\mL)]=[\rho(\mL),A^{-}]=0$; finally 
$S=A^{+}+A^{-}J_{\VV}$ clearly fully solves the initial constraint \eqref{git-s}.
\end{proof}
The advantage of the result above is to reduce the computation of the 
$\Kk$-module
$\Sym^2_{\Kk}(\VV,h_{\VV}) \cap \gl_{J_{\VV}}(\VV)$ to the 
computation of specific parts of the centraliser $\mathfrak{c}(\mL,\rho,\VV)$. The latter is computed in the standard way from 
the decomposition of $\VV$ into irreducible pieces.

\begin{defn} \label{moduli-def}$\mathcal{M}(\mL,\rho,\VV)$ is the moduli space of isometry classes of  the set    $\me_{K}(\bbV)$ of Riemannian 3-symmetric metrics on a  type III space,  with $\mL$ understood to be simple and of non-compact type. 
\end{defn} 

Below we use Proposition \ref{iso-invm} to describe $\mathcal{M}(\mL,\rho,\VV)$; in section \ref{loc-red} of the paper the same result will be needed to determine which metrics in 
$\me(M,J,D)$ are deRham irreducible.

Consider the centraliser of $\pi(G)$ in $\GL(\VV)$ defined according to \eqref{cent-grp-d}. 
Because $G$ is simply connected 
we have 
\begin{equation*}
C(G,\pi,\VV)=\{f \in \mathfrak{c}(\mL,\rho,\VV) : \det(f) \neq 0\}.
\end{equation*} 
We also have a right action 
\begin{equation*}
(g_{\VV},f) \in \me_{K}\VV\times C(G,\pi,\VV) \mapsto  f^{\star}g_{\VV} \in \me_{K}\VV
\end{equation*}
with respect to which we form the quotient $\me_{K}\VV \slash C(G,\pi,\VV)$.
Since $ C(G,\pi,\VV) $ acts by isometries, it is immediate that  $\mathcal{M}(\mL,\rho,\VV) \subset  \me_{K}\VV \slash C(G,\pi,\VV)$. When $\rho$ is of real or complex type, we will show that we actually have equality and hence determine the moduli space of compatible Riemannian 3-symmetric metrics. The case where $\rho$ has quaternionic type is left to a future paper. 

\begin{propn} \label{iso-invm1}
Let $(\mL,\rho,\VV)$ be an admissible representation where $\mL$ is simple. 
\begin{itemize}
\item[(i)] if $\mL$ has non-compact type
the map 
\begin{equation*}
\{A^{-} \in \mathfrak{c}^{-}(\mL,\rho,\VV):1_V+A^{-}J_{\VV}>0\} \to \me_{K}(\VV, J_{\VV})\slash C(G,\pi,\VV)
\end{equation*}
induced by $A^{-} \mapsto h_{\VV}((1_{\VV}+A^{-}J_{\VV})\cdot, \cdot)$ is surjective. 
\item[(ii)] if $\mL$ has compact type the map 
\begin{equation*}
\{A^{+} \in \mathfrak{c}^{+}(\mL,\rho,\VV):1_{\VV}+A^{+}>0\} \to \me_{K}(\VV, J_{\VV})\slash C(G,\pi,\VV)
\end{equation*}
induced by $A^{+} \mapsto h_{\VV}((1_{\VV}+A^{+})\cdot, \cdot)$ is surjective. 
\end{itemize}

\end{propn}
\begin{proof}
Pick $g_V \in \me_{\Kk}(\VV)$. By Proposition \ref{iso-invm} we have $h_{\VV}^{-1} \circ g_{\VV}=A^{+}+A^{-}J_{\VV}$ where 
$A^{\pm} \in \mathfrak{c}^{\pm}(\mL,\VV,g_{\VV})$. \\
(i) we first show that $A^{+}$ is positive definite with respect to $h_{\VV}$. By construction $2A^{+}=S+\mathscr{C}^{-1}\circ S_1$ where $S=h_{\VV}^{-1} \circ g_{\VV}>0$. But 
\begin{equation*}
h_{\VV}(S_1v,v)=\sum_{i}h_{\VV}(\rho(e_i)S\rho(e_i)v,v)=\sum_{i}h_{\VV}(S(\rho(e_i)v),\rho(e_i)v)>0.
\end{equation*}
As $\mathscr{C}$ is positive definite and commutes with $S_1$ it follows that 
$A^{+}>0$ as claimed. 
 
Pick $f \in C(G,\pi,\VV)$. Then 
\begin{equation} \label{act-expl}
h_{\VV}^{-1} \circ f^{\star}g_{\VV}=f^TA^{+}f+(f^TA^{-}f)J_{\VV}
\end{equation}
by taking into account that $fJ_{\VV}=J_{\VV}f$. Then $f=(A^{+})^{-\frac{1}{2}}$ belongs to 
$C(G,\rho,\VV)$ and satisfies $h_{\VV}^{-1} \circ f^{\star}g_{\VV}=1_{\VV}+(f^TA^{-}f)J_{\VV}$ 
by \eqref{act-expl}. Using Proposition \ref{R-sem}, (i), we see that the map 
$\varphi \in \Aut(\VV \rtimes_{\pi} G)$ given by $ \varphi(v,g):=(fv,g)$ induces an isometry between the metrics 
$g_{\VV}$ and $h_{\VV}((1+A^{+})\cdot, \cdot)$. This finishes the proof of the claim. \\
(ii) we have $2A^{-}J_{\VV}=S-\mathscr{C}^{-1}\circ S_1$. As above $-S_1 >0$ with respect to $h_{\VV}$ and the claim follows by using entirely analogous arguments.
\end{proof}
The parametrisation 
\begin{equation} \label{can-par}
h_{\VV}^{-1} \circ g_{\VV}=1+S, \ \mathrm{with} \ SJ_{\VV} \in \mathfrak{c}^{-}(\mL,\VV,h_{\VV}) \ \mbox{when} \ \mL \ \mbox{is non-compact}
\end{equation}
respectively 
\begin{equation} \label{can-par}
h_{\VV}^{-1} \circ g_{\VV}=1+S, \ \mathrm{with} \ S \in \mathfrak{c}^{+}(\mL,\VV,h_{\VV})\ \mbox{when} \ \mL \ \mbox{is compact}
\end{equation}
from Proposition \ref{iso-invm} will be referred to as the {\it{canonical parametrisation}} for isotropy invariant metrics. It will be used systematically for computing geometric invariants for such metrics.  

The centraliser $\mathfrak{c}(\mL,\rho,\VV)$ can be easily described from general principles; we briefly recall this description below and use it in combination with 
Proposition \ref{iso-invm} to derive information on the set of isometry classes of $3$-symmetric metrics. For an {\it{irreducible}} representation 
$(\mL,\rho,\VV)$ consider $(\mL,\rho_d, \VV_d:=d \VV), d \geq 1$ where $d\VV$ indicates the direct sum of $d$ copies of $\VV$. Any $f \in \gl(V_d)$ is determined from 
\begin{equation} \label{act-cent1}
fv_i=\sum \limits_{k=1}^d f_{ik}v_i 
\end{equation}
where $f_{ij} \in \gl(\VV), 1 \leq i,j \leq p$. It follows that 
\begin{equation}\label{act-cent10}
(g \circ f)_{ij}=\sum \limits_{k=1}^d g_{kj} \circ f_{ik}
\end{equation}
whenever $f,g \in \gl(\VV_d)$.

We fix a background metric $h_{\VV}$ 
on $\VV$ and consider the background metric $h_d=d h_{\VV}$, the direct sum of $d$-copies of $h_{\VV}$, on $\VV_d$. Then 
$f$ is symmetric, respectively skew-symmetric, with respect to $h_d$ iff
\begin{equation*} 
f_{ji}=f_{ij}^T, \ \mbox{respectively} \ f_{ji}=-f_{ij}^T.
\end{equation*}
In particular, if $f \in \mathfrak{c}(\mL,\rho_d,\VV_d) $ then $f_{ij} \in \mathfrak{c}(\mL,\rho,\VV)$.

Using the structure of centraliser we will obtain simple parametrisations for the 
set of isometry classes of $3$-symmetric metrics on $M=(\VV_d \rtimes_{\pi_d} G)\slash K$ according to the type of $\VV$: real, complex or quaternionic. 
\begin{propn} \label{r-type}
Assume that $(\mL,\rho,\VV)$ is irreducible of real type, with $\mL$  simple and non-compact. 
Any $3$-symmetric metric on $M=(\VV_d \rtimes_{\pi_d} G)\slash K$ is isometric to a background metric.
\end{propn}
\begin{proof}
Since $\VV$ has real type $\mathfrak{c}(\mL,\rho,\VV)=\bbR1_{\VV}$ whence 
\begin{equation*}
\mathfrak{c}(\mL,\rho_d,\VV_d)=\mathfrak{c}^{+}(\mL,\rho_d, \VV_d) \cong 
\gl_d(\bbR) \ \mathrm{and} \ \mathfrak{c}^{-}(\mL,\rho_d,\VV_d)=0.
\end{equation*}
Therefore, by Proposition \ref{iso-invm}, any isotropy invariant metric is a background metric. The claim follows from Proposition \ref{iso-invm1}. 
\end{proof}
Suppose that $\VV$ has complex type, that is $\mathfrak{c}(\mL,\rho,\VV)=
\spa \{1_{\VV},I\}$. After writing $f_{ij}=A_{ij}1_{\VV}+B_{ij}I$ 
we obtain a linear isomorphism 
\begin{equation} \label{act-cent2}
\varphi:\mathfrak{c}(\mL,\rho_d,\VV_d) \to \gl_d(\mathbb{C}), \varphi(f):=\biggl ( \begin{array}{cc} A & B\\
-B & A \end{array} \biggr )^T
\end{equation}
where $A=(A_{ij})_{1 \leq i,j \leq d}, B=(B_{ij})_{1 \leq i,j \leq d}$. A straightforward computation based on \eqref{act-cent10} shows that  
\begin{equation*}
\varphi(f \circ g)=\varphi(f)\varphi(g), \ \varphi(f^T)=(\varphi(f))^T
\end{equation*}
in particular $\varphi$ is a Lie algebra isomorphism.
Further on we have 
\begin{equation} \label{act-cent3}
\varphi(\mathfrak{c}^{-}(\mL,\rho_d,\VV_d)=\mathfrak{u}(d) \ \mathrm{and} \ \varphi(\mathfrak{c}^{+}(\mL,\rho_d,\VV_d)=\Sym^2(\bbR^{2d}) \cap \gl_d(\mathbb{C}).
\end{equation}
Denote $\Delta_d=\{\lambda=(\lambda_1, \ldots, \lambda_d) \in \bbR^d : 0 \leq \lambda_1 \leq \ldots \leq \lambda_d<1 \}$.
For any $\lambda \in \bbR$ satisfying $\vert \lambda \vert<1$ consider the isotropy invariant metric $g_{\lambda}$ on $V$ determined from 
$h_{\VV}^{-1} \circ g_{\lambda}=1+\lambda IJ$; note that having $\vert \lambda \vert <1$ ensures that $1+\lambda IJ>0$. This allows considering metrics of the form 
\begin{equation*}
g_{\lambda}=g_{\lambda_1} \oplus \ldots \oplus g_{\lambda_d}
\end{equation*}
on $\VV_d$ where $\lambda=(\lambda_1, \ldots, \lambda_d) \in \Delta_d$.
\begin{thm} \label{cpx-ty}
Let $M=(\VV_d \rtimes_{\pi_d} G)\slash K$. If $(\mL,\rho,\VV)$ is irreducible of complex type and $\mL$ is non-compact the map
$$\Delta_d \to \mathcal{M}(\mL, \rho, \VV_d)$$ induced by $\lambda \mapsto g_{\lambda}$ is surjective.
\end{thm}
\begin{proof}
Because $I \in \mathfrak{c}(\mL,\rho_d,\VV_d)$ we have an isomorphism $\mathfrak{c}^{-}(\mL,\rho_d,\VV_d) \to 
\mathfrak{c}^{+}(\mL,\rho_d,\VV_d)$ given by $f \mapsto f \circ I$. Therefore 
any isotropy invariant metric $g_{\VV_d}$ can be parametrised according to 
\begin{equation*}
h_d^{-1} \circ g_{\VV_d}=1_{\VV}+B_{+}IJ 
\end{equation*}
where $B_{+}$ belongs to $\mathfrak{c}^{+}(\mL,\rho_d, \VV_d)$. 
Next we consider the matrix $b_{+}=\varphi(B_{+})$ in $\Sym^2(\bbR^{2d}) \cap \gl_d(\bbC)$. Thus 
$b_{+}=U^{-1}D(\lambda)U$ with  $U \in U(d)$, 
where 
$D(\lambda), \lambda \in \bbR^d$ denotes the diagonal matrix 
$\biggl (\begin{array}{cc} \lambda & 0 \\ 0 & \lambda \end{array} \biggr ) \in \gl_d(\bbC)$. Let $f:=\varphi^{-1}(U) \in C(G,\pi_d,\VV_d)$; taking into 
account that $[f,I]=[f,J]=0$ we obtain 
$$h_d^{-1} \circ f^{\star}g_{\VV}=f^Tf+(f^TB_{+}f) \circ IJ=1_{\VV}+\varphi^{-1}(D(\lambda))\circ IJ=g_{\lambda}.
$$
Here the transpose is taken with respect to $h_d$.
To show that $\lambda$ can be chosen to lie in $\Delta_d$ let $S_d$ be the permutation group. Via the map $\varphi$ this is identified to a subgroup of 
$C(G) \cap \mathrm{SO}(\VV_d,h_d)$ acting on metrics of type $g_{\lambda}$ according to $(\sigma,g_{\lambda}) \mapsto g_{\sigma(\lambda)}$ where 
$\sigma(\lambda_1, \ldots, \lambda_q)=(\lambda_{\sigma(1)}, \ldots, \lambda_{\sigma(d)})$. It follows that the components in $\lambda$ can be ordered lexicographically. Similar arguments for the action $\mathbb{Z}_2 \times 
\bbR^d \to \bbR^d, (\varepsilon,\lambda) \mapsto (\varepsilon \lambda_1, \ldots, \lambda_q)$ show that we can choose $\lambda_1 \geq 0$.

\end{proof}
In the next section we will show the map above is also injective; see Theorem \ref{cpx-ty2}.
\section{The canonical polar foliation} \label{it-curv}
We work in the following more general set-up throughout this section. Let $\mL=\Kk \oplus \H$ be the Cartan decomposition, where $\mL$ is simple. Consider the semi-direct product $\g:=\VV \rtimes_{\rho}\mL$ together with the splitting $\g=\Kk \oplus \bbV$ where $\bbV=\VV \oplus \H$. The $\Kk$-invariant linear complex structure on $\bbV$ is $J:=J_{\VV}+J_{\H}$. Finally let $g_{\VV}$ be a $\Kk$-invariant metric on $\VV$. This allows to equip $\bbV$ with the $\Kk$-invariant metric $g=g_{\VV}+g_{\H}$. Whenever $x \in \H$ we split 
$\rho(x)=\rho^{-}(x)+\rho^{+}(x)$ according to $\gl(\VV)=\so(\VV,g_{\VV}) \oplus \Sym^2(\VV,g_{\VV}) $.
\begin{propn} \label{it-semi}
The intrinsic torsion tensor of $\D$ satisfies 
\begin{equation} \label{its-1}
\eta_{\VV}\VV \subseteq \H, \ \eta_{\H}\VV \subseteq \VV, \ \eta_{\VV} \H \subseteq \VV, \ \eta_{\H}\H=0
\end{equation}
as well as 
\begin{equation} \label{its-2}
g(\eta_{v_1}v_2,x)=g_V(\rho^{+}(x)v_1,v_2)
\end{equation}
and 
\begin{equation} \label{its-3}
g(\eta_{x}v_1,v_2)=g_V(\rho^{-}(x)v_1,v_2)
\end{equation}
for all $v_1,v_2 \in V,x \in \H$.
\end{propn}
\begin{proof}
By the definition of the Lie algebra structure in $\g$ it follows that the tensor $\tau$ satisfies $\tau(\VV,\VV)=0$; we also have 
$[x,v]=\rho(x)v$ with $x \in \H, v\in \VV$. Using this in \eqref{it-expr} yields 
$\eta_{\VV}\VV\subseteq \H$ as well as \eqref{its-2}. That $\eta_{\H}\H=0$ follows from the same formula by using that $\tau(\H,\H)=0$ and $\tau(\V,\H)\subseteq V$.
In particular $\eta_{\H}\VV \subseteq V$ by orthogonality. Since $\tau(\VV,\H) \subseteq V$ it follows that $\eta_{\VV} \H \subseteq \VV$ as well. Finally, \eqref{its-3} follows from \eqref{it-expr} in a completely similar way.
\end{proof}
Consider the manifold $M=(\VV \rtimes_{\pi} G)\slash K$. Equip $M$ with distributions $\V$ respectively $\H$ obtained by projecting the $K$-invariant 
subspaces $\VV$ respectively $\H$ of $\bbV$ onto $M$. A key observation  
in this section is the following 
\begin{thm} \label{polar}
We have 
\begin{itemize}
\item[(i)] the $g$-orthogonal splitting 
\begin{equation} \label{pol-1}
TM=\V \oplus \H
\end{equation} defines a polar Riemannian foliation with respect to the metric $g$
\item[(ii)] the Bott connection of the foliation defined by \eqref{pol-1} co\"incides with $\D$.
\end{itemize}
\end{thm}
\begin{proof}
(i) as $\mathfrak{s}$ and $\H$ are $\Kk$-invariant the corresponding 
distributions $\V$ and $\H$ are $D$-parallel. From $D=\nabla^g+\eta$ and $\eta_{\H}=0$ we get $\nabla^g_XY \in \H$ for $X,Y \in \H$, that is $\H$ is totally 
geodesic with respect to $g$. The distribution $\V$ is integrable since it is parallel with respect to $D$ and the restriction of $\eta$ to $\V$ is a symmetric tensor. Thus $\V$ is tangent to the leaves of a polar Riemannian foliation.\\
(ii) from the definitions the Bott connection $\nabla^B$ is given 
by orthogonal projection of $\nabla^g$ onto $TM=\V \oplus \H$. 
\begin{equation*}
\nabla^B_{E}F=(\nabla^g_{E}F)_{\V}+(\nabla^g_{E}F)_{\H}
\end{equation*}
with $E,F \in TM$ where the subscript indicates orthogonal projection onto the distribution. The claim follows now easily from \eqref{its-1}. 
\end{proof}

For further reference we record that the O'Neill tensor $T:\V \times \V \to \H$, the orthogonal projection onto 
$\H$ of $\nabla^g_{v_1}v_2$ where $v_1,v_2 \in \V$, is given by 
\begin{equation} \label{ON-1}
T_{v_1}v_2=-\eta_{v_1}v_2.
\end{equation}
This follows from combining the information that $D$ preserves $\mathcal{V}$ and the fact that $\eta_{\mathcal{V}}\mathcal{V}\subseteq \mathcal{H}$.

\subsection{Complex and symplectic structures} \label{cpx-symp}
For this subsection, we  specialise to consider an admissible representation $(\mL,\rho,\VV)$ with respect to a complex structure 
$J_{\VV}$ on $\VV$. We investigate the complex structures on $\bbV$, in a differential geometric sense, induced by admissible representations. 
\begin{defn} \label{cpx-h}
Let $\g$ be a Lie algebra equipped with a reductive decomposition $\g=\h \oplus \m$. An $\ad_{\h}$-invariant linear complex structure $J:\m \to \m$ is called integrable if its Nijenhuis tensor $N^J:\m \times \m \to \m$ vanishes identically. Here 
\begin{equation} \label{int-hom}
N^J(x,y)=[x,y]_{\m}-[Jx,Jy]_{\m}+J([x,Jy]_{\m}+[Jx,y]_{\m})
\end{equation}
for all $x,y \in \m$ and the subscript indicates projection onto the subspace.
\end{defn}
If $\g$ respectively $\h$ are the Lie algebras of a Lie group $G$ respectively 
of a closed subgroup $H$ in $G$ then any $\Ad(H)$-invariant $J$ satisfying \eqref{int-hom} projects onto an integrable complex structure on the homogeneous manifold $M:=G \slash H$.
\begin{propn} \label{int-s}
Let $\mL=\Kk \oplus \mathcal{H}$ be a Hermitian symmetric Lie algebra and let $\rho :\mL \to \mathfrak{gl}(\VV)$ be a linear representation. Consider the reductive decomposition 
$$ V \rtimes_{\rho} \mL=\Kk \oplus (\VV \oplus \mathcal{H}).
$$
\begin{itemize}
\item[(i)] if $J_{\VV}$ is a linear complex structure on $V$ such that $\rho: \mL \to \mathfrak{gl}(V)$ is admissible with respect to $J_V$ then 
$$\tilde{J}:=-J_{\VV}+J_{\mathcal{H}}:\bbV \to \bbV,$$ 
where $\bbV:=V \oplus \H$, is complex integrable
\item[(ii)] the Nijenhuis tensor of $J=J_{\VV}+J_{\mathcal{H}}$ is non-degenerate; in particular $J$
is never complex integrable
\item[(iii)] if $g_{\VV}$ is an isotropy invariant metric for the admissible representation 
$\rho$ then $(g=g_{\VV}+g_{\mathcal{H}},J)$ is almost-K\"ahler  if and only if $\mL$ has non-compact type 
and $g_{\V}V=h_{\VV}$  where $h_{\VV}$ is a background metric; in this case $(g,\widetilde{J})$ is a K\"ahler structure on $M$.
\end{itemize}
\end{propn} 
\begin{proof}
(i) We use the definition of the Nijenhuis tensor in \eqref{int-hom} with $\m=\V \oplus \H$ together with 
the structure of the 
Lie bracket in $V \rtimes_{\rho} \mL$. Thus
\begin{equation*}
\begin{split}
& N^{\tilde{J}}(v,w)=N^{\tilde{J}}(x,y)=0 \\
&N^{\tilde{J}}(v,x)=-\rho(x)v-\rho(J_{\H}x)J_vv+J_V\rho(J_{\H}x)v-J_V\rho(x)J_Vv)=0
\end{split}
\end{equation*}
for all $v,w \in V$ and $x,y \in \H$. The last equality follows from having $\rho$ admissible, that is $\{\rho(x),J_V\}=0$ and 
$\rho(J_{\H}x)=\rho(x)J_V$.\\
(ii) as in (i), from the definition of the Lie bracket in $V \rtimes_{\rho} \mL$ we get 
\begin{equation*}
\begin{split}
& N^J(v,w)=N^{J}(x,y)=0 \\
&N^{J}(v,x)=-\rho(x)v+\rho(J_{\H}x)J_vv-J_V(\rho(J_{\H}x)v+\rho(x)J_Vv)=-4\rho(x)v
\end{split}
\end{equation*}
for all $v,w \in V$ and $x,y \in \H$. In the last equality above, we have taken again into account that the representation $\rho$ is admissible. These relations ensure that having $N^J(v_0+x_0),\cdot)=0$ is equivalent with $\rho(x_0)=0$ and $\rho(\H)v_0=0$. Because 
$\rho$ is faithful, $[\H,\H]=\mathfrak{k}$ and $\rho$ acts without fixed points it follows that $v_0=x_0=0$ that is $N^J$ is non-degenerate. 
Note that $N^J$ does not have maximal rank since $N^J(\bbV,\bbV)=\rho(\H)V=V$; see also subsection \ref{el-qK} for the relevant definitions.\\ 
(iii) the exterior derivative of $\omega=g(J \cdot, \cdot)$ corresponds to 
\begin{equation*}
(\di\!\omega)(w_1,w_2,w_3)=-\omega(\tau(w_1,w_2),w_3)-\omega(\tau(w_2,w_3),w_1)
-\omega(\tau(w_3,w_1),w_2)
\end{equation*}
whenever $w_i, 1 \leq i \leq 3$ belong to $\bbV$. Take into account that the torsion tensor satisfies $\tau(\H,\H)=\tau(V,V)=0, \tau(V,\H) \subseteq V$ and also use that $\omega(V,\H)=0$. It follows that the components of $\di\!\omega$ on $\Lambda^3V, \Lambda^1V \wedge \Lambda^2\H$ respectively $\Lambda^3\H$ vanish. Furthermore, from the definition 
of $\tau$ we obtain 
\begin{equation*}
(\di\!\omega)(v_1,v_2,x)=\omega(\rho(x)v_2,v_1)-\omega(\rho(x)v_1,v_2)=
2g_V(\rho^{-}(x)J_Vv_1,v_2).
\end{equation*}
Thus $\di\!\omega=0$ iff $\rho(\mathcal{H}) \subseteq \Sym^2(\VV,g_{\VV})$. Thus $g_{\VV}$ is a background metric and $\mL$ must be of non-compact type; the proof of this fact is entirely similar to that of part (i) in Lemma \ref{cent-cpx}. To see that $(g,\widetilde{J})$ is K\"ahler let $\widetilde{\omega}=g(\tilde{J} \cdot, \cdot)$; because $\V$ is a polar foliation, we have $\di\!\widetilde{\omega}+\di\!\omega=0$ and the claim is proved.
\end{proof}
In particular the $3$-symmetric structure on $\VV \times M$ is induced by $(g,J)$. 
\subsection{Curvature structure} \label{curv-s}
In this section, we return to our  general setting where  $\mL=\Kk \oplus \H$ is  the Cartan decomposition with $\mL$ simple. 
We now move on to determine the Riemann curvature tensor $R^g$ of $g$. Recall that at infinitesimal model level $R^g$ is determined from 
\begin{equation} \label{comp-f}
R^D(U_1,U_2)=R^g(U_1,U_2)+[\eta_{U_1},\eta_{U_2}]-\eta_{\tau(U_1,U_2)}
\end{equation}
whenever $U_1,U_2 \in \bbV$. For clarity's sake the explicit computation of $R^g$ is performed in two steps as follows. Using the formulas 
from \cite{Bes} for the curvature tensor of a homogeneous Riemannian space yields in our setting 
\begin{propn} \label{curv-1}
We have 
\begin{itemize}
\item[(i)] $R^g(v_1,v_2,v_3,x)=R^g(x_1,x_2,x_3,v)=0$
\item[(ii)] $R^g(x_1,x_2,x_3,x_4)=g_{\H}([[x_1,x_2],x_3],x_4)$
\item[(iii)] $R^g(v_1,x_1,v_2,x_2)=-g(\rho^{+}(x_1)v_1,\rho^{+}(x_2)v_2)+
g([\rho^{-}(x_1),\rho^{+}(x_2)]v_1,v_2)$
\item[(iv)] $R^g(v_1,v_2,v_3,v_4)=g_{\VV}(\eta_{v_2}v_3,\eta_{v_1}v_4)-g_{\VV}(\eta_{v_1}v_3,\eta_{v_2}v_4)$.
\end{itemize}
\end{propn}
\begin{proof} This relatively straightforward computation follows from Equation (\ref{comp-f}) by expanding the $R^D$ terms as  Lie brackets  and applying Proposition \ref{it-semi}.
\end{proof}
\begin{rem} \label{sigm-R}
The restriction $R^g_{\vert \H}$ has a sign, either positive or negative according 
to the type of $\mL$:compact or non-compact. This is a standard fact which can be explained as follows. Because the isotropy representation $(\Kk,\H)$ is irreducible 
we must have $\BK^{\mL}_{\vert \H}=\Lambda g_{\H}$ with $\Lambda \in \mathbb{R}, \Lambda \neq 0$. The invariance of $\BK^{\mL}$ grants 
the explicit expression $R^g(x_1,x_2,x_3,x_4)=\frac{1}{\Lambda}\BK([x_1,x_2],[x_3,x_4])$. Taking into account that $-\BK^{\mL}_{\vert \Kk}$ is positive definite it follows that $R^g$ is positive definite when $\Lambda<0$($\mL$ has compact type) and negative definite when $\Lambda>0$($\mL$ is of non-compact type).
\end{rem}


Indicate with $\ric^{\mathcal{H}}$ denote the Ricci form of the restriction of $R^g$ to 
$\mathcal{H}$, that is $\ric^{\mathcal{H}}(x_1,x_2)=\sum \limits_{i} R^g(x_1,e_i,x_2,e_i)$ where 
$\{e_i\}$ is some $g_{\mathcal{H}}$-orthonormal basis in $\mathcal{H}$. Define $Q:\VV \to \VV$ by 
\begin{equation*}
Q:=\sum \limits_{i}[\rho^{-}(e_i),\rho^{+}(e_i)]
\end{equation*}
where $\{e_i\}$ is some $g_{\mathcal{H}}$-orthonormal basis in $\mathcal{H}$. Since $Q$ is a trace-type operator with respect to $g_{\mathcal{H}}$ it is independent of basis choice. This operator turns up as part of the Ricci tensor of the infinitesimal model.
\begin{propn} \label{Ricci}
We have 
\begin{itemize}
\item[(i)] $\ric^g(v_1,v_2)=g_V(Qv_1, v_2)$
\item[(ii)] $\ric^g(V,\mathcal{H})=0$
\item[(iii)] $\ric^g(x_1,x_2)=-\Tr(\rho^{+}(x_1) \circ \rho^{+}(x_2))+\ric^{\H}(x_1,x_2)$.
\end{itemize}
\end{propn}
\begin{proof}
Since $\mL$ is semisimple the map $l \in \mL \mapsto \Tr(\rho(l))\in \bbR$ vanishes identically. In particular $\Tr(\rho^{+}x)=0$ for all $x \in \mathcal{H}$. Taking this into account the claim follows either directly from Proposition \ref{curv-1} or by applying the general formulas from \cite{Bes}. 
\end{proof}
\begin{thm} \label{sol-1}
Let $\mL$ be a simple Lie algebra of non-compact type with Cartan decomposition 
$\mL=\Kk \oplus \H$. Let $\rho:\mL \to \gl(\VV)$ be a faithful linear representation. 
With respect to the reductive decomposition 
\begin{equation*}
\VV \rtimes_{\rho}\mL=\Kk \oplus \bbV \ \mbox{where} \ \bbV=\VV \oplus \H
\end{equation*}
metrics of the type 
$h=h_{\VV}+t(B^{\mL})_{\vert \H}, t>0$, where $h_V$ is a background metric for $\rho$, are expanding Ricci soliton metrics.
\end{thm}
\begin{proof}
It is enough to show that $h$ is an algebraic Ricci soliton in the sense of 
\cite{LaLa}. That is the Ricci form of $h$ satisfies
\begin{equation*} \label{alg-sol}
\ric^h=\Lambda h+h(\zeta_{\bbV} \cdot, \cdot)
\end{equation*}
for some $\Lambda \in \bbR$, where $\zeta \in \Der(\g), \g=\VV \rtimes_{\rho} \mL$ and $\zeta_{\bbV}$ indicates the piece in $\zeta$ acting solely  
on $\bbV$, with respect to the reductive splitting $\g=\Kk \oplus \bbV$. Proving this is a straightforward consequence of Proposition \ref{Ricci}. Indeed, since $(\Kk,\H)$ 
is irreducible we have $g_{\H}=t(\BK^{\mL})_{\H}$ for some $t>0$. Because $\mL$ is simple and $\rho$ is faithful the trace form of $\rho$ satisfies $t_{\rho}=\mu(\rho)\BK^{\mL}$.

Computing $t_{\rho}$ by means of the background metric $h_{\VV}$ will show that $\mu(\rho)>0$. More specifically, choose $X\in \H$. Then have $\BK^{\mL}(X,X)>0$. At the same time, computing the trace form with respect to a basis $\lbrace e_i \rbrace$ of $V$ yields
$$
t_{\rho} (X,X) = \sum_{e_i} h_{\VV}(\rho(X)\circ \rho(X)e_i, e_i).
$$
This is independent of basis. As $X\in \H$, $\rho(X)\in \Sym^2(\VV, h_{\VV})$ since we have chosen to compute with the background metric. Taking now our basis to be the eigenvectors of $\rho(X)$ immediately implies that $t_{\rho}(X,X)>0$, and the claim follows. 

  As $h_V$ is a background metric $Q=0$ and $\rho^{+}=\rho$ in Proposition \ref{Ricci} which thus yields 
\begin{equation*}
\ric^h=-(\mu(\rho)+\frac{1}{2})(\BK^{\mL})_{\vert \H}=-t^{-1}(\mu(\rho)+\frac{1}{2})h+
t^{-1}(\mu(\rho)+\frac{1}{2})h(1_{V} \cdot, \cdot)
\end{equation*}
after re-arranging terms. Note that we have used the equality $\ric^{\H}=-
\frac{1}{2}(\BK^{\mL})_{\H}$, a standard fact in the theory of orthogonally symmetric Lie algebras. Taking into account that $1_{V} \in \mathfrak{c}(\mL,\rho,\VV)\subseteq \Der(\g)$(see Proposition \ref{cen-rad-p},(iii) for the last inclusion) thus proves the claim.
\end{proof}
This proof can be viewed as a refined version of Theorem \ref{t3} from the introduction, and so Theorem \ref{t3} is now established.  It is worth noting the degree of generality in this construction; no assumption 
is made on the representation $\rho$, as the existence of background metrics is 
granted by Mostow. Theorem \ref{sol-1} provides a systematic way of constructing homogeneous Ricci solitons from arbitrary representations of non-compact semisimple Lie groups. 
\begin{rem} \label{ric-back} 
For the special instance when $g_{\VV}=h_{\VV}$ with $h_{\VV}$ a background metric on $\VV$ we have $Q=0$. Thus 
\begin{equation*}
\ric^h=2\BK^{\mL}_{\vert \H}, \ \mbox{in particular} \ \ric^h(\VV,\bbV)=0.
\end{equation*}
\end{rem}

Our next task is to fully compute the Ricci tensor for $3$-symmetric metrics. We only treat $3$-symmetric spaces of type $III$, when 
$\mL$ has non-compact type; for spaces of type $IV$ these considerations are not needed since any $3$-symmetric metric on such spaces is a Riemannian product (see Theorem \ref{irred-Rc}).
\begin{propn} \label{ric-full}
Let $(\mL,\rho,\VV)$ be an admissible representation where $\mL$ is simple of 
non-compact type. Let $g=g_{\VV}+g_{\H}$ be a $3$-symmetric metric of the form 
$$ h_{\VV}^{-1} \circ g_{\VV}=1_{\VV}+S,
$$
where $SJ \in \mathfrak{c}^{-}(\mL,\rho).$ We have 
\begin{itemize}
\item[(i)]
$Q=2S(1-S^2)^{-1}\mathscr{C}$
\item[(ii)] $\Tr(\rho^{+}(x) \circ \rho^{+}(x))=\Tr((1-S^2)^{-1} \circ \rho(x) \circ \rho(x))$, whenever $x \in \H$.
\end{itemize}
\end{propn}
\begin{proof}
In the standard parametrisation of the metric $g_{\VV}$ above we have 
\begin{equation} \label{st-p}
2\rho^{+}(x)=\rho(x)+(1+S)^{-1}\rho(x)(1+S) \ \mbox{and} \ 2\rho^{-}(x)=\rho(x)-(1+S)^{-1}\rho(x)(1+S)
\end{equation}
for all $x \in \H$. Because $[SJ,\rho(x)]=0$ we have $(1+S)^{-1}\rho(x)=\rho(x)(1-S)^{-1}$ a fact which we will use systematically in the computations below.
(i) we compute 
\begin{equation*}
\begin{split}
4Q=&\sum \limits_{i}^{}[\rho^{-}(e_i),\rho^{+}(e_i)]=2\sum \limits_{i}^{}[\rho(e_i),(1+S)^{-1}\rho(e_i)(1+S)]\\
=&2\sum \limits_{i}^{}(\rho(e_i)(1+S)^{-1}\rho(e_i)(1+S)-(1+S)^{-1}\rho(e_i)(1+S)
\rho(e_i))\\
=&2((1-S)^{-1}(1+S)-(1+S)^{-1}(1-S))\mathcal{C}=8S(1-S^2)^{-1}\mathscr{C}
\end{split}
\end{equation*}
and the claim follows.\\
(ii)we have 
\begin{equation*}
\begin{split}
4\rho^{+}(x) \circ \rho^{+}(x)=&\rho(x) \circ \rho(x)+\rho(x)(1+S)^{-1}\rho(x)
(1+S)+(1+S)^{-1}\rho(x)(1+S)\rho(x)\\
&+(1+S)^{-1}\rho(x)\circ \rho(x)(1+S)\\
=&2\rho(x) \circ \rho(x)+((1+S)(1-S)^{-1}+(1-S)(1+S)^{-1})\rho(x) \circ \rho(x)\\
=&2(1+(1+S^2)(1-S^2)^{-1})\rho(x) \circ \rho(x)=4(1-S^2)^{-1}\rho(x) \circ \rho(x)
\end{split}
\end{equation*}
and the claim is proved.
\end{proof}

Theorem \ref{sol-1}  arose out of our study of  Riemannian $3$-symmetric spaces of type III, which are of course special cases of the construction just outlined. In that setting we have the following result. 
\begin{thm} \label{sol-2}
Let $(M,g,J)$ be a Riemannian $3$-symmetric space of Type III. The 
metric $g$ is a Ricci soliton if and only if $(g,J)$ is almost-K\"ahler. 
\end{thm} 
\begin{proof}
Suppose firstly $(g, J)$ is almost-K\"ahler. Then by Proposition \ref{int-s} $g$ is a metric of the form considered in Theorem \ref{sol-1}, and hence it is an expanding Ricci soliton. 

Conversely,  if $g$ is a Ricci soliton metric then it is an algebraic Ricci soliton. To see this one uses the homogeneous structure on $M$ and recent work of Jablonski \cite{jab}. This states that 
$\ric^g = \lambda g +  g(\zeta_{\vert \bbV}\cdot, \cdot)$
for $\zeta \in \Der_{\Kk}(\g)$ and $\lambda \in \bbR$. Proposition \ref{Ricci}(i) yields that
$$
g_{\VV}(Q v_1, v_2) = \lambda g_{\VV}(v_1,v_2)+ g_{\VV}(\zeta v_1, v_2)
$$
which implies $Q\in \Der(\g)$.  From Equation \ref{der-gen2} it is apparent that $Q\in \mathfrak{c}(\mL,\rho, \VV)$. Next, we claim that this implies that $S=0$, in other words $g_{\VV}=h_{\VV}$. 
By Lemma \ref{cent-cpx} we have $Q=m \circ \rho(z)$ where $m=\Lambda(SJ)(1-S^2)$. As $m$ belongs to the centraliser 
$\mathfrak{c}(\mL,\rho,\VV)$, having 
$Q$ therein entails that $0=[Q,\rho(x)]=m \circ [\rho(z), \rho(x)]=m \circ \rho(J_{\H}x)$ for all $x \in \mathcal{H}$. This leads easily to $m=0$ and the claim follows. 
\end{proof}
Proposition \ref{ric-full} can now be immediately put to work in order to explicitly determine the set of isometry classes of $3$-symmetric metrics on some classes of spaces of type $III$.
\begin{thm} \label{cpx-ty2}
Let $M=(\VV_d \rtimes_{\pi_d} G)\slash K$, where $\mL$ is simple.  If $(\mL,\rho,\VV)$ is irreducible of complex type the map 
$$\Delta_d \to \mathcal{M}(\mL, \rho, \VV)$$ induced by $\lambda \mapsto g_{\lambda}$ is 
bijective.
\end{thm}
\begin{proof}
Since $(\mL,\rho,\VV)$ is irreducible the Casimir operator 
$\mathscr{C}$, acting on $\VV$, has at most two distinct eigenvalues $d_{\pm}>0$ with eigenspaces $\ker(IJ \pm 1_{\VV})$. It follows that the non-negative eigenvalues of 
$\Ric^{g_{\lambda}}$ are $\{\psi(\lambda_k)d_{-}, 1 \leq k \leq d\}$ where $\psi:[0,1) \to \bbR$ is given by $ \psi(x)=\frac{2x}{1-x^2}$. If $g_{\mu}$ is isometric to $g_{\lambda}$ the operators $\Ric^{g_{\lambda}}$ and $\Ric^{g_{\mu}}$ must have the same spectrum, in particular they must have the same positive spectrum. Because 
$\psi$ is increasing it follows that $\lambda=\mu$. The claim follows now from 
Theorem \ref{cpx-ty}.
\end{proof}

\section{ The isometry group  of a Riemannian 3-symmetric metric} \label{iso-fin}
Let $(M,g,\D)$ be a Riemannian Ambrose-Singer space. We first explain how to compute, from first principles, the isometry group $\Aut(M,g)$ in terms of data involving only the connection $\D$. When $(M^{2m},g,J)$ is Riemannian $3$-symmetric these general observations will also enable us to decide explicitly which Riemannian metrics $g$ in $\me(M,J)$ satisfy $\Aut^0(M,g) \subseteq \Aut(M,J)$. 

\subsection{Isometries of Ambrose-Singer spaces} \label{g-obs}
Consider a Riemannian Ambrose-Singer space $(M,g,\D)$ not necessarily $3$-symmetric. We begin with the following general observation which is slightly rephrasing a well-known result in \cite{Nom1}.
Let $\bbV:=T_xM$ for some $x \in M$. For ease of reference the Riemannian curvature $R^g$, the intrisinc torsion tensor $\eta^g$ and the metric $g$ evaluated at $x$ will be denoted with the same symbols. Consider the Lie algebra action of $\mathfrak{so}(\bbV,g)$ on the space of algebraic Riemann curvature tensors which we identify,via the metric $g$, with linear maps from $\Lambda^2\bbV$ to $\Lambda^2\bbV$. Explicitly 
\begin{equation*}
\begin{split}
[F,R^g](v_1,v_2)=R^g(Fv_1,v_2)+R^g(v_1,Fv_2)+[R^g(v_1,v_2),F].
\end{split}
\end{equation*}
This enters the definition of the stabiliser 
$$\stg_{\so(\bbV,g)}:=\{F \in \so(\bbV,g) : [F,R^g]=0\}$$ 
which is actually a Lie subalgebra in $\so(\bbV,g)$. Furthermore, we consider the operators 
$l_v:\so(\bbV,g) \to \so(\bbV,g), v \in \bbV$ given by 
$$F \mapsto l_v(F):=[F,\eta^g_v]-\eta^g_{Fv}$$
Define 
\begin{equation*}
\begin{split}
&\i^0:=\stg_{\so(\bbV,g)}, \i^{k+1}:=\{F \in \i^k : l_v(F) \in \i^k \ \mbox{for all } \ v \in \bbV\}, k \geq 0\\
& \i:=\bigcap \limits_{k \geq 0}^{} \i^k.
\end{split}
\end{equation*} 
\begin{rem} \label{gen-no}In fact the spaces $\i^k, k \geq 1$ also do not depend on the Ambrose-Singer connection, for $\i^k=\mathfrak{stab}_{\so(\bbV,g)} (\nabla^g)^kR^g$ at the point $x$. The proof of this fact which hinges on having $(\nabla^g)^kR^g$ parallel with respect to $D$ is very similar to \cite{Tri2} and will be omitted.
\end{rem}
Re-interpreting the Lie algebra $\i$ in a purely abstract way in terms of the operators $l_v$ as above has the practical benefit 
to allow equipping the vector space $\i \oplus \bbV$ with an {\it{explicit}} Lie algebra structure. Indeed, a straightforward algebraic computation shows that 
\begin{propn} \label{a-int}
The following hold
\begin{itemize}
\item[(i)] $\i$ is a subalgebra in $\stg_{\so(\bbV,g)}$ such that 
\begin{equation} \label{inv-ia}
l_v(\i) \subseteq \i, v \in \bbV.
\end{equation}
\item[(ii)] the direct sum of vector space $\g_b:=\i \oplus \bbV$ is a Lie algebra with respect to the Lie bracket 
\begin{equation*}
\begin{split}
&[F,v]^b=Fv+l_v(F), \ [F,G]^b=[F,G]\\
&[v_1,v_2]^b=\tau(v_1,v_2)+R^{\D}(v_1,v_2)
\end{split}
\end{equation*}
\item[(iii)] the Nomizu algebra $\og^g$ sits canonically in $\g_b$ as a subalgebra such that 
\begin{equation} \label{oni}
\g_b=\og^g+\i.
\end{equation}
\end{itemize}
\end{propn}
\begin{proof}
(i) follows from either remark \ref{gen-no} or directly from the identity 
\begin{equation} \label{br-b1}
l_v[F,G]=l_{Gv}F-l_{Fv}G+[F,l_vG]-[G,l_vF]\\
\end{equation}
which is valid for any $F,G \in \so(\bbV,g)$ and $v \in \bbV$.\\
(ii) Letting $v_1,v_2,v_3$ belong to $\bbV$ we have $[[v_1,v_2]^b,v_3]^b=[[v_1,v_2]^{\og},v_3]^{\og}+
l_{v_3}R^D(v_1,v_2)$, where $[\cdot, \cdot]^{\og}$ denotes here the Lie bracket in $\og$. Because $\oh$ preserves $\eta$ we have 
\begin{equation} \label{inkl}
\l_v \oh=0 \ \mathrm{for \ all} \ v \in \bbV,
\end{equation}
in particular $l_{v_3}R^D(v_1,v_2)=0$ since $\uh \subseteq \oh$. The vanishing of $\mathfrak{S}_{abc}[[v_a,v_b]^b,v_c]^b$ thus follows from the Jacobi identity in $\og$. The rest of the Jacobi identity 
for the bracket $[\cdot, \cdot]^b$ follows from \eqref{br-b1} together with the purely algebraic 
identities 
\begin{equation} \label{br-b2}
(l_{v_1}F)v_2-(l_{v_2}F)v_1=\tau(v_2,Fv_1)-\tau(v_1,Fv_2)+F\tau(v_1,v_2)
\end{equation}
\begin{equation} \label{br-b3}
[l_{v_1},l_{v_2}]F+l_{\tau(v_1,v_2)}F=[F,R^D](v_1,v_2) - [F,R^g](v_1,v_2)
\end{equation}
where the action $[F,R^{\D}]$ is defined in complete analogy to $[F,R^g]$. Both of these are valid for arbitrary $v_1,v_2 \in \bbV$ and $F \in \so(\bbV,g)$. The first formula is proved by direct computation, only based on the definition of the operators 
$l_v$. The second follows from the comparaison formula \eqref{comp-f} and is in fact used here for $F \in \i$, when $[F,R^g]=0$.\\
(iii) follows from \eqref{inkl}.
\end{proof}
\begin{rem} \label{semi-On}
When $\g_b$ is simple of compact type \eqref{oni} can be used directly to extract geometric information on the space $(M,g)$ by means of 
Onischik classification results in \cite{on}. 
\end{rem}
In fact the construction of $\g_b$ is a purely algebraic way to describe the algebra of Killing generators as introduced by Nomizu \cite{Nom1}. We explain how this works below.
\begin{lemma} \label{der-11}
Assuming that $K \in \aut(M,g)$ we have 
$$\D_U(\nabla^gK+\eta_K)=R^{\D}(K,U)-l_U (\nabla^gK+\eta_K) \ \mathrm{ for \ all} \ U \in TM.$$
\end{lemma}
\begin{proof}
For simplicity of notation consider the skew-symmetric endomorphism given by $F:=\nabla^gK$. As Killing fields are parallel with respect to Kostant's connection on $TM\oplus \so(TM)$ it is well known that $\nabla^g_UF=R^g(K,U)$. Using this and the curvature comparaison formula \eqref{comp-f} we compute 
\begin{equation*}
\begin{split}
\nabla^g_UF+\eta_{\eta_UK}=&R^g(K,U)+\eta_{\tau(U,K)}+\eta_{\eta_KU}=R^{\D}(K,U)-[\eta_K,\eta_U]+\eta_{\eta_KU}\\
=&R^{\D}(K,U)-l_U(\eta_K)
\end{split}
\end{equation*}
after also taking into account the definition of the operators $l_U$. Furthermore, 
since $\D=\nabla^g+\eta$ we have 
$ \D_UK=FU+\eta_UK \ \mathrm{and} \ \D_UF=\nabla^g_UF+[\eta_U,F].$
As $\D\!\eta=0$ these entail 
\begin{equation*}
\begin{split}
\D_U(F+\eta_K)=&\nabla^g_UF+[\eta_U,F]+\eta_{D_UK}=\nabla^g_UF+[\eta_U,F]+\eta_{FU+\eta_UK}\\
=&\nabla^g_UF+\eta_{\eta_UK}-l_U(F)=R^{D}(K,U)-l_U(F+\eta_K)
\end{split}
\end{equation*}
and the claim is proved.
\end{proof}
This elementary observation has far reaching consequences in that it shows that $\aut(M,g)$ is isomorphic to the Lie algebra $\g_b$; in particular the Lie bracket in $\aut(M,g)$ can be fully described in terms of the Ambrose-Singer connection $\D$. We proceed in several steps as follows.

Consider the linear connection on the bundle $\mathrm{E}:=TM \oplus \so(TM)$ given by 
$$ \widetilde{D}_U(X,F):=(\D_UX+\tau(X,U)-FU, \D_UF+l_UF-R^{\D}(X,U))
$$
Lemma \ref{der-11} then says that sections of $\mathrm{E}$ which are parallel with respect to $\widetilde{D}$ correspond to Killing vector fields. To get a purely algebraic interpretation of those we compute below the holonomy algebra $\hol(\widetilde{D})$ of the connection $\widetilde{D}$. The curvature tensor of the latter is defined acccording to $\widetilde{R}(U_1,U_2):=-[\widetilde{D}_{U_1},\widetilde{D}_{U_2}]+\widetilde{D}_{[U_1,U_2]}.$ 

\begin{propn} \label{iso-eps}
At a given point $x \in M$ we have 
$$\{s \in E_x : \hol(\widetilde{D}) s=0\}=\i_x \oplus T_xM.$$
\end{propn}
\begin{proof}
To simplify notation write $\widetilde{R}(U_1,U_2)(U_3,F)=
(A_{U_1,U_2}(U_3,F), B_{U_1,U_2}(U_3,F))$. Using only that $\D\!\tau=0$, a formal computation following the definitions shows that 
\begin{equation*}
\begin{split}
A_{U_1,U_2}(U_3,F)=&\mathfrak{S}_{abc}(R^{\D}(U_1,U_2)U_3-\tau(U_a,\tau(U_b,U_c)))\\
&-\bigl ((l_{U_1}F)U_2-(l_{U_2}F)U_1+\tau(FU_1,U_2)+\tau(U_1,FU_2)-F\tau(U_1,U_2) \bigr ).
\end{split}
\end{equation*}
We conclude that $A=0$ by using the algebraic Bianchi identity for $\D$ in \eqref{B1} and \eqref{br-b2}. Similarly, using $\D\!l=0, \D\!R^{\D}=0$ and $l_{U_1}R^{\D}(X,U_2)=0$ leads to 
\begin{equation*}
\begin{split}
-B_{U_1,U_2}(U_3,F)=-\mathfrak{S}_{abc}R^{\D}(U_a,\tau(U_b,U_c))-[F,R^{\D}](U_1,U_2)+[l_{U_1},l_{U_2}]F+
l_{\tau(U_1,U_2)}F.
\end{split}
\end{equation*}
Using the differential Bianchi identity in \eqref{B2} and \eqref{br-b3} we conclude that 
\begin{equation} \label{Rt}
\widetilde{R}(U_1,U_2)(U_3,F)=(0, -[F,R^g](U_1,U_2)).
\end{equation}
Now consider higher order derivatives of the curvature tensor of $\widetilde{D}$, when coupling the latter, 
which is a connection on $\mathrm{E}=TM \oplus \so(TM)$, with the connection $\D$ on $TM$. On tensors 
$q:\otimes^pTM \times \mathrm{E} \to \mathrm{E}$ we thus have 
\begin{equation*}
\begin{split}
(\widetilde{D}_Uq)(X_1,\ldots, X_p,s)=&\widetilde{D}_U(q(X_1, \ldots, X_p,s))-q(X_1, \ldots, X_p,\widetilde{D}_Us)\\
&-\sum_{i=1}^p q(X_1,\ldots,\D_U\!X_i, \ldots, X_p,s).
\end{split}
\end{equation*}
Extend now $\D$ to $\mathrm{E}$ as a product connection and write $\widetilde{D}=\D+\zeta$ where the tensor $\zeta : TM \to \gl(\mathrm{E})$. It follows that if $\D\!q=0$ we have 
\begin{equation*}
\begin{split}
(\widetilde{D}_Uq)(X_1,\ldots, X_p,s)=&\zeta_U(q(X_1, \ldots, X_p,s))-q(X_1, \ldots, X_p,\zeta_Us)
\end{split}
\end{equation*}
Explicitly, the tensor $\zeta$ acts according to 
$\zeta_U(X,F)=(\tau(X,U)-FU,l_UF-R^D(X,U))$ and thus satisfies $\D\!\zeta=0$. Since \eqref{Rt} entails that also $\D\!\widetilde{R}=0$ taking $q=\widetilde{D}^{p-1}\widetilde{R}$ where $p \geq 1$ we get 
$$ (\widetilde{D}^{p}_{X_1, \ldots, X_p} \widetilde{R})(U_2,U_3)=(\ad_{\zeta_{X_1}} \circ \ldots \circ \ad_{\zeta_{X_p}}) \widetilde{R}(U_2,U_3)
$$
where $\ad$ denotes the Lie bracket in $\gl(\mathrm{E})$. By induction it is easy to see that the set of vectors in $E_x$ which are annihilated by the holonomy 
algebra of $\widetilde{D}$ is 
$$ \mathscr{F}=\{s \in E_x : (\widetilde{R}(U_2,U_3) \circ \zeta_{x_1} \circ \ldots \circ \zeta_{x_p})s=0 \} 
$$
for all $x_1, \ldots x_p \in T_xM$ and for all $p \geq 0$. Since the operators $\zeta_v, v \in \bbV$ preserve the subspace $\bbV \oplus \uh \subseteq \mathrm{E}_x$ an induction argument shows that  
$$ \zeta_{x_1} \circ \ldots \circ \zeta_{x_p}(X,F)=(l_{x_1} \circ \ldots \circ l_{x_p})F \ \mathrm{mod} \ 
\bbV \oplus \uh.$$
This observation, together with $\widetilde{R}(U_1,U_2)(\bbV \oplus \uh)=0$ shows that 
$$\mathscr{F}=\bbV \oplus \{F \in \so(\bbV,g) : (\widetilde{R}(U_2,U_3) \circ l_{x_1} \circ \ldots \circ l_{x_p})F=0\ \mathrm{for} \ x_1, \ldots, x_p \in \bbV \ \mathrm{and} \ p \geq 0\}.$$
It follows that $\mathscr{F}=\bbV \oplus \i$ by using \eqref{Rt} and the definition of $\i$. 
\end{proof}
\begin{thm} \label{embedd}
The following hold at any point $x \in M$
\begin{itemize}
\item[(i)] for any $K\in \aut(M,g)$ we have $(\nabla^gK+\eta_K)_x \in \i$
\item[(ii)] assuming that $M$ is simply connected we have a Lie algebra isomorphism 
$$\varepsilon : \aut(M,g) \to 
(\g_b, [\cdot, \cdot]^b)\  \mathrm{given \ by}\ \varepsilon(K):=-(\nabla^gK+\eta_K)_x+K_x.
$$
\end{itemize}

\end{thm}
\begin{proof}
(i) According to Lemma \ref{der-11} we have $\widetilde{D}(K,\nabla^gK+\eta_K)=0$, in particular $(K_x,(\nabla^gK+\eta_K)_x)$ is acted on trivially by $\hol(\widetilde{D})$ and the claim follows from 
proposition \ref{iso-eps}.\\
(ii)Pick $K_1,K_2 \in \aut(M,g)$ and write, for simplicity $F_a=\nabla^gK_a$ with $a=1,2$. A direct computation using only the definitions of the Lie bracket $[\cdot, \cdot]_b$ and that of the operators $l_v$ shows that 
\begin{equation*}
\begin{split}
[\varepsilon(K_1), \varepsilon(K_2)]_b=&[F_1,F_2]+(F_2K_1-F_1K_2)-[\eta_{K_1},\eta_{K_2}]+
\eta_{F_1K_2-F_2K_1+\tau(K_1,K_2)}\\
&+R^D(K_1,K_2).
\end{split}
\end{equation*}
Since $F_2K_1-F_1K_2=[K_1,K_2]$ using the comparaison formula \eqref{comp-f} shows that 
\begin{equation*}
\begin{split}
[\varepsilon(K_1), \varepsilon(K_2)]_b=&[F_1,F_2]-\eta_{[K_1,K_2]}+R^g(K_1,K_2)+[K_1,K_2].
\end{split}
\end{equation*}
As it is well known(see e.g.\cite{Nom1}), the Killing fields $K_1$ and $K_2$ satisfy the identity $[\nabla^gK_1,
\nabla^gK_2]+R^g(K_1,K_2)=-\nabla^g[K_1,K_2]$ hence $\varepsilon$ is a Lie algebra morphism. The latter 
is injective since any Killing field $K$ is determined by the values of $K$ and $\nabla^gK$ at the point $x$. There remains to prove that $\varepsilon$ is surjective.

By proposition \ref{iso-eps} any element of $\bbV \oplus \i$ extends,by parallel transport, to a $\widetilde{D}$-parallel section of $\mathrm{E}$. As such sections are necessarily of the form $(K,\nabla^gK+\eta_K)$ with $K \in \aut(M,g)$ it follows that for any pair $(v,F) \in \bbV \oplus \i$ there exists a pair $(K^1,K^2)$ of Killing vector fields such that 
$$K^1_x=v, \ (\nabla^gK^1+\eta_{K^1})_x=0, \ K^2_x=0, \ (\nabla^gK^2+\eta_{K^2})_x=F.
$$
Then $\varepsilon(K^1-K^2)=(v,F)$ thus $\varepsilon$ is surjective.
\end{proof}
This generalises the well known fact ( see e.g.\cite{Nom1}) that at any point $x \in M$ the isotropy algebra $\aut_x(M,g):=\{X \in \aut(M,g) :X_x=0\}$ is isomorphic to $\i_x$, via $X \mapsto (\nabla^gX)_x$. For subsequent use we also record the following 
\begin{rem} \label{sym-cor}
Let the Lie algebra $\Kk \oplus \H$ be Hermitian symmetric with corresponding connected symmetric space $(M,g)$; as it is well 
the isotropy algebra $\aut_{o}(M,g)$ at the origin is canonically isomorphic to $\Kk$. 
On the other hand, consider the Riemann curvature tensor
$R^0(x_1,x_2,x_3,x_4):=g_{\H}([[x_1,x_2],x_3],x_4)$ on $\H$. Since $\eta^g=0$ we see that the Lie algebra $\i$ reduces to 
$\i=\mathfrak{stab}_{\so(\H,g_{\H})}R^0$. Then Theorem \ref{embedd} together with the previous observation 
ensures that $h \in \Kk \mapsto \ad_h \in \i$ is a linear isomorphism.

\end{rem}
\subsection{The dichotomy for $3$-symmetric spaces} \label{alt}
Letting $(M^{2m},g,J)$ be Riemannian $3$-symmetric we start studying the Lie algebras $\i$ 
introduced previously. The objective is to determine which instances satisfy $\Aut^0(M,g)=\Aut^0(M,g,J)$. We fix a point $x \in M$ and observe that 
\begin{lemma} \label{3symm-1}
The following hold at the point $x$
\begin{itemize}
\item[(i)] we have 
$$\gamma(\i) \subseteq \i \ \mathrm{and} \ \ad_J(\i) \subseteq \i $$ 
where $\gamma:\mathfrak{so}(\bbV) \to \mathfrak{so}(\bbV)$ is given by $\gamma(F)=-JFJ$
\item[(ii)] $\mathfrak{i}=\mathfrak{i}^{+} \oplus \mathfrak{i}^{-}$ where $\i^{+}=\{F \in \i : FJ=JF\}$ and $\i^{-}=\{F \in \i : FJ+JF=0\}$
\item[(iii)] denoting $\bbV^{-}:=\spa \{Fv : F \in \i^{-}, v \in \bbV\}$ we have 
\begin{equation} \label{iso21}
\eta_v \in \i^{-} \ \mbox{for all} \ v \in \bbV^{-}
\end{equation}
as well as 
\begin{equation} \label{iso22}
[\i^{-},\eta_v] \subseteq \i^{+} \ \mbox{for all} \ v \in \bbV.
\end{equation}
\end{itemize}
\end{lemma}
\begin{proof}
Throughout the proof we denote with $F^{\pm}$ the components of $F \in \so(\bbV,g)$ with respect to $\so(\bbV,g)=\mathfrak{u}(\bbV,g,J) \oplus \mathfrak{u}^{\perp}(\bbV,g,J)$. Similarly, if the subspace
$ \mathfrak{s} \subseteq \so(\bbV,g)$ satisfies $\gamma(\mathfrak{s}) \subseteq \mathfrak{s}$ we split 
$\mathfrak{s}=\mathfrak{s}^{+} \oplus \mathfrak{s}^{-}$ where $\mathfrak{s}^{\pm}:=\mathfrak{s} \cap \ker(\gamma \pm 1_{\bbV})$.\\
To prove the claim in (i) we show, by induction on $k$, that $\gamma(\i^k) \subseteq \i^k, \ad_J(\i^k) 
\subseteq \i^k$ for $ k \geq 0$. 
When $k=0$ this follows from 
$R^g(J \cdot, J \cdot, J\cdot, J \cdot)=R^g$ and $J \in \i^0$, facts which are granted by Proposition \ref{G2}, (i) and (ii).

Assume that $\gamma(\i^k) \subseteq \i^k, \ad_J(\i^k) \subseteq \i^k$ and pick $F \in \i^{k+1}$. 
From the definition of $\i^{k+1}$ we have $l_vF \in \i^k$; since the latter is preserved by $\gamma$ it follows that 
\begin{equation*}
\begin{split}
&(l_vF)^{-}=[F^{+},\eta_v]-\eta_{Fv} \in (\i^k)^{-}\ \mathrm{and} \ (l_vF)^{+}=[F^{-},\eta_v] \in (\i^k)^{+}
\end{split}
\end{equation*}
for all $v \in \bbV$. Since $\eta_{Jv}=\eta_v \circ J$ we have 
$$ [F^{+},\eta_{Jv}]-\eta_{FJv}=[F^{+},\eta_{v}]J-\eta_{F^{+}v-F^{-}v}J=l_v(F^{+})J+\eta_{F^{-}v}J.
$$
Because $\i^{k}$ is preserved by $\ad_J$ it follows that 
\begin{equation} \label{bra-21}
l_v(F^{+})+\eta_{F^{-}v} \in (\i^k)^{-}.
\end{equation}
As $[F^{+},\eta_{v}]-\eta_{Fv}=l_v(F^{+})-\eta_{F^{-}v}$  comparaison with \eqref{bra-21} above yields 
$$ l_v(F^{+}) \in (\i^k)^{-} \ \mathrm{and} \ \eta_{F^{-}v} \in (\i^k)^{-} $$
for all $v \in \bbV$. In particular $l_v(F^{+}),l_v(F^{-})$ belong 
to $\i^k$, showing that $\gamma(F) \in \i^{k+1}$. Finally 
$$ l_v(\ad_JF)=2l_v(JF^{-})=2[F^{-},\eta_{Jv}]-2\eta_{JF^{-}v}
$$
belongs to $\i^k$ as well and the proof is finished.\\
(ii) follows from (i).\\
(iii) has been proven during the proof of (i).
\end{proof}
Standard orthogonality arguments arguments with respect to the Killing form of $\so(\bbV,g)$ which is negative definite 
ensure that $\i^{+}=\mathfrak{a} \oplus [\i^{-}, \i^{-}]$, a direct sum of ideals. Thus $\i=\mathfrak{a} \oplus \mathfrak{j}$, a direct sum of ideals, where $\mathfrak{j}:=\mathfrak{i}^{-} \oplus [\i^{-}, \i^{-}]$. Furthermore split 
$$\bbV=\bbV^{+} \oplus \bbV^{-}$$
orthogonally with respect to $g$.
\begin{lemma} \label{split1-iso}
Let $(M,g,J)$ be $3$-symmetric. At a given point $x \in M$ 
\begin{itemize}
\item[(i)] the spaces $\bbV^{+}$ and $\bbV^{-}$ are $\mathfrak{j}$-invariant 
\item[(ii)] we have $\mathfrak{j}\bbV^{+}=0$ and moreover $\mathfrak{j}$ acts 
on $\bbV^{-}$ without zero vectors, $(\bbV^{-})^{\mathfrak{j}}=0$
\item[(iii)] $\bbV^{+}$ and $\bbV^{-}$ are $\i$-invariant
\item[(iv)] we have $ \eta_{\bbV^{+}}\bbV^{+} \subseteq \bbV^{+}$
\item[(v)] we have $\eta_{\bbV^{-}}\bbV^{+}=0$ and $\eta_{\bbV^{-}}\bbV^{-} \subseteq \bbV^{-}$
\item[(vi)] the representation $(\mathfrak{j},\bbV^{-})$ is admissible.
\end{itemize}
\end{lemma}
\begin{proof}
(i) from the definition, $\i^{-}\bbV^{-} \subseteq 
\bbV^{-}$. Taking commutators leads to $\mathfrak{j}\bbV^{-} \subseteq \bbV^{-}$ hence $\mathfrak{j}$ preserves $\bbV^{+}$ as well, by orthogonality.\\
(ii) from the definition of $\bbV^{-}$ it follows that 
$\i^{-}\bbV^{+}=0$, by orthogonality. That $\mathfrak{j}\bbV^{+}=0$ follows by taking commutators. If $v \in \bbV^{-}$ is such 
that $\mathfrak{j}v=0$ we have, in particular, $\i^{-}v=0$. Then, with respect to $g$, we have  
$v \perp \i^{-}\bbV=\bbV^{-}$ thus $v=0$.\\
(iii) combining $[\mathfrak{a}, \mathfrak{j}]=0$ and $\mathfrak{j}\bbV^{+}=0$ yields $\mathfrak{j}(\mathfrak{a} \bbV^{+})=0$. It follows that 
$\mathfrak{a} \bbV^{+} \subseteq \bbV^{+} \oplus (\bbV^{-})^{\mathfrak{j}}$ thus $\bbV^{+}$ is $\mathfrak{a}$-invariant by (ii) and the claim follows easily.\\
(iv) pick $v \in \bbV^{+}$. By \eqref{iso22} we have $[\i^{-},\eta_v] \subseteq \i^{+}$. However $[\i^{-},\eta_v] \perp \mathfrak{a}$ since $[\i^{-},\mathfrak{a}]=0$. It follows that 
$$ [\i^{-},\eta_v] \subseteq \mathfrak{j}.
$$
Because $\mathfrak{j}$ acts trivially on $\bbV^{+}$ we get 
$0=[\mathfrak{j},\eta_v]w=\mathfrak{j}(\eta_vw)$ for all $w \in \bbV^{+}$. In other words $\eta_{\bbV^{+}}\bbV^{+} \subseteq \bbV^{+} \oplus (\bbV^{-})^{\mathfrak{j}}$. The claim follows by using (ii).\\
(v) by \eqref{iso21} we have $\eta_v \in \mathfrak{j}$ for all $v \in \bbV^{-}$. Again because $\mathfrak{j}$ acts trivially on $\bbV^{+}$ we get 
$\eta_{\bbV^{-}}\bbV^{+}=0$ and the last part of the claim is proved by orthogonality.\\
(vi) $\mathfrak{j}$ is semisimple of compact type by construction, actually Hermitian symmetric and the claim essentially follows from the vanishing of $(\bbV^{-})^{\mathfrak{j}}$.
\end{proof}

Below we prove the main dichotomy-type result in this section; the Lie algebras ${}^g\og_J,{}^g\oh_J$ are build out of the infinitesimal model for the canonical connection $D$.
\begin{thm} \label{split-iso}
Let $(M^{2m},g,J)$ be a simply connected Riemannian $3$-symmetric space such that $g$ is irreducible. Then either
\begin{itemize}
\item[(i)]  $\i=\oh^g_J$ that is $\Aut^0(M,g)=\Aut^0(M,g,J)$
\item[or] 
\item[(ii)] $\Aut^0(M,g) \neq \Aut^0(M,g,J)$ in which case $(M,g)$ is a Riemannian symmetric space.
\end{itemize}
\end{thm}
\begin{proof}
The splitting $\bbV=\bbV^{+} \oplus \bbV^{-}$ is invariant under the intrinsic torsion tensor $\eta$ by Lemma
\ref{split1-iso}, (iii) and (iv). It is also $\underline{\h}$-invariant, since the latter is a subalgebra of $\i$, which preserves 
$\bbV^{+}$ respectively $\bbV^{-}$ by (iii) in Lemma \ref{split1-iso}. The splitting as  a Riemannian product is thus clear. 

In order to understand its properties 
we examine the behaviour of the transvection Lie algebra $\underline{\g}$ of 
$(M,g,J)$. Since we only know that $\eta_{\bbV^{+}}\bbV^{-} \subseteq \bbV^{-}$
(but we do not know that $\eta_{\bbV^{+}}\bbV^{-}=0$) a priori $\underline{\g}$ is not a direct product Lie algebra.
There are two cases to distinguish.\\
{\it{Case 1.$\bbV^{+}=0$}} at some point $x \in M$. In this situation the only information left in Lemma \ref{3symm-1} and \ref{split1-iso} is having 
$\eta_v \in \i^{-}, v \in \bbV$. Since $\i^{-} \subseteq \stg$ it follows that $[\eta_v,R^g]=0$. Thus by taking into account that $\D\!R^g=0$ we also obtain that $(\nabla^gR^g)_x=0$. Having $\D\!(\nabla^gR^g)=0$ ensures that $\nabla^gR^g$ has constant norm and thus enables to conclude that $\nabla^gR^g=0$ over $M$.\\
{\it{Case 2. $(M,g)$ is not locally symmetric. }} Then $\bbV^{-}=0$ at every point in $M$; indeed, because the spaces $\bbV^{\pm}$ are holonomy invariant and the metric is irreducible, at a point where $\bbV^{-} \neq 0$ we must have $\bbV^{+}=0$, thus $g$ is symmetric by Case 1, a contradiction.

Further on, from $\bbV^{-}=0$ we deduce that $\i^{-}=0$ thus $\i \subseteq \un(\bbV,g,J)$. As $\eta_v \in \un^{\perp}(\bbV,g,J)
$ by using \eqref{inv-ia} we get $l_v(\i)=0, v \in \bbV$ thus $\i \subseteq \oh^g_J$. Equality follows from $\oh^g_J \subseteq \i$. As explained above the latter equality is valid at any point in $M$.

To conclude, let $K \in \aut(M,g)$ and record the identity $-\mathscr{L}_KJ=[\nabla K+\eta_K,J]$.  
Since $\nabla K+\eta_K \in \mathfrak{i}$ we get $\mathscr{L}_KJ=0$ at any point in $M$. It follows that $\Aut^0(M,g) \subseteq \Aut^0(M,g,J)$ and the claim is proved.
\end{proof}
As a consequence we can determine immediately isometry groups for spaces of type $III$ as follows. The idea is to 
use the dichotomy previously proved together with the abstract description of $\mathfrak{aut}(M,g)$ in Theorem \ref{embedd}.
\begin{thm} \label{isoIII} Let $M=(\VV \rtimes_{\pi} G) \slash K$ be $3$-symmetric of type $III$, equipped with its canonical 
almost complex structure $J$. For any $g \in \me(M,J)$ we have 
\begin{equation*}
\Aut^0(M,g) \cong{}^g \overline{G}=\VV \rtimes_{\pi} (G \times {}^g C^0(G,\pi,\VV)).
\end{equation*}
\end{thm}
\begin{proof}
Because the Lie algebra of $G$ is simple of non-compact type the metric $g$ is irreducible. This will be proved later on, in Theorem \ref{irred-R}. The metric $g$ cannot be symmetric; assuming $g$ symmetric implies it is Einstein. Then, part (i) in Proposition \ref{Ricci}, yields, since $\Tr(Q)=0$, that $g$ is Ricci flat which contradicts that $g$ is irreducible. 

Therefore the assumptions in Theorem \ref{split-iso} are satisfied and we conclude that at a given point $\mathfrak{i}=\oh^g_J$ and also $\Aut^0(M,g)=\Aut^0(M,g,J)$. 

From now on work at a point $x \in M$ where $\bbV=\VV \oplus \H$ and pick $F \in \mathfrak{i}$. Because $F$ preserves 
$R^D$ and the nullity of the latter is precisely $V$ it follows that $F\VV \subseteq \VV$ thus $F\H \subseteq \H$ as well since $F$ is skew-symmetric. In particular, 
$F$ is block diagonal with respect to the splitting $\bbV=\VV \oplus \H$, which allows writing $F=F_{11}+F_{22}$.  
Since $F \in \mathfrak{stab}(R^g)$ we have, in particular, $[F,R^g](\H,\H,\H,\H)=0$; then by Proposition \ref{curv-1},(ii) it follows that $F_{22}$ preserves the Riemann curvature tensor of the symmetric Lie algebra $\Kk \oplus \H$, namely $(x_1,x_2,x_3,x_4) \mapsto g_{\H}([[x_1,x_2],x_3],x_4)$ on $\H$.  According to Remark \ref{sym-cor} we have  $F_{22}x=[h,x]$ for some $h \in \Kk$. Furthermore, since $F$ preserves $\eta^g$ it also preserves the torsion tensor $\tau$, in particular 
$ F\tau(x,v)=\tau(Fx,v)+\tau(x,Fv)$
for all $v \in V$ and $x \in \H$. Due to having $\tau(x,v)=\rho(x)v$ it follows that 
\begin{equation*}
\begin{split}
F_{11}\rho(x)v=&\rho(F_{22}x)v+\rho(x)F_{11}v=\rho([h,x])v+\rho(x)F_{11}v\\
=&[\rho(h),\rho(x)]v+\rho(x)F_{11}v.
\end{split}
\end{equation*}
It follows that 
$F_{11}=\rho(h)+A$ where 
$A \in \mathfrak{u}(V,g,J)$ satisfies $A\rho(x)=\rho(x)A$ for all $x \in \H$. Thus, $A \in \mathfrak{c}_J(\mL,\rho,|VV)$ and we have showed that 
\begin{equation*}
F=A+\rho(h)+\ad_h.
\end{equation*}
Due to Lemma \ref{cent-cpx},(iii) we have $\mathfrak{c}(\mL,\rho,\VV)=\mathfrak{c}_J(\mL,\rho,\VV)$ hence we have an isomorphism $ \Kk \oplus \mathfrak{c}_J(\mL,\rho,\VV) \to \mathfrak{i}$. Based on this fact, there remains to spell out the structure of the Lie algebra $\g_b$ as given in Proposition \ref{a-int}, (ii). Because $l_v\mathfrak{i}=0$ whenever $v \in 
\bbV$ the Lie bracket structure in $\g_b$ is exactly the same as in $\og$, thus as in the semi-direct product $\VV \rtimes (\mL \oplus \mathfrak{c}(\rho,\VV,g_{\VV}) )$. The Lie algebra isomorphism 
$$\aut(M,g) \cong \VV \rtimes \left (\mL \oplus \mathfrak{c}(\rho,\VV,g_{\VV}) \right )$$ 
follows now from Theorem \ref{embedd}. Since $\VV \rtimes_{\pi} \left (G \times {}^g C^0(G,\pi,\VV) \right ) \subseteq \Aut^0(M,g)$ by Proposition \ref{R-sem}, (iii), this Lie algebra isomorphism entails equality and the claim is proved. 

\end{proof}

\section{Local irreducibility} \label{loc-red}
The aim in this section is to show that Riemannian $3$-symmetric spaces 
of the form $((\VV \rtimes_{\pi} G)\slash H,g)$, where $G$ is a simple Lie group, have irreducible Riemannian holonomy. 
In order to do so we need a $1:1$ correspondence between a Riemannian holonomy reduction and some type of additional algebraic structure at infinitesimal model level. 

When $(M,g)$ is a connected Riemannian manifold with curvature tensor $R^g$ the correspondence 
$$m\in M \to \{X \in T_mM :R^g(X,\cdot)=0 \}$$ is referred to as the Riemann curvature nullity of the metric $g$. We denote by $\V_R$ the corresponding distribution, which may not have constant rank over $M$. The curvature nullity is invariant under the isometry group of $M$; in particular if $M$ is homogeneous 
$\V_R$ will have constant rank over $M$. 
\begin{rem}
The Riemann nullity has been studied for various classes of metrics, such as Riemannian submersions \cite{MoSe} or 
homogeneous spaces with reductive group of isometries \cite{DiSc3}. Structure results for arbitrary Riemannian metrics seem unavailable so far.
\end{rem}
We start with the following general observation in the case of Ambrose-Singer spaces; in case the Riemann nullity vanishes identically it provides a simple algebraic criterion 
for deciding when the Ambrose-Singer has reducible Riemannian holonomy. This is needed for deciding local irreducibility 
for type III spaces.
\begin{propn} \label{holR-red}
Let $(M,g,D)$ be a simply connected Ambrose-Singer manifold. Assume that 
$\V_{R}=0$ and consider the minimal infinitesimal model $\ug=\uh \oplus \bbV$ at a fixed point $m \in M$. The following are equivalent 
\begin{itemize}
\item[(i)] $(M,g)$ is reducible
\item[(ii)] there exists a $g$-orthogonal splitting $\bbV=\E_1 \oplus \E_2$ such that the spaces $\E_i, i=1,2$ are
\begin{itemize}
\item[(a)] invariant under the action of $\uh$
\item[(b)] invariant under the intrinsic torsion $\eta$ of $D$(at the point $m$).
\end{itemize}
\end{itemize}
\end{propn}
\begin{proof}
Assume that $g$ is reducible. Then $TM=\mathscr{D}_1 \oplus \mathscr{D}_2 $ orthogonally with respect to $g$, where the distributions $\mathscr{D}_1, \mathscr{D}_2$ are parallel with respect to the Levi-Civita connection $\nabla^g$ of $g$. Those distributions are therefore invariant under the Riemann curvature tensor, that is 
\begin{equation} \label{R-inv}
R^g(TM,TM) \mathscr{D}_1 \subseteq \mathscr{D}_1 \ \mathrm{and}  \ R^g(TM,TM) \mathscr{D}_2 \subseteq \mathscr{D}_2.
\end{equation}
Because $\D\!\eta=0$ and $\D\!R^{\D}=0$ it follows that $\D\!R^g=0$; the  expanded version of this reads 
\begin{equation*}
(\nabla^g_{U_1}R^g)(U_2,U_3)=R^g(\eta_{U_1}U_2,U_3)+R^g(U_2,\eta_{U_1}U_2)+
[R^g(U_2,U_3),\eta_{U_1}]
\end{equation*}
whenever $U_k \in TM, 1 \leq k \leq 3$.  As $\mathscr{D}_i, i=1,2$ are $\nabla^g$-parallel they must also be invariant under the action of the covariant derivative $\nabla^gR^g$. Combining the above formula with \eqref{R-inv} yields 
\begin{equation} \label{R-inv2}
[R^g(U_2,U_3),\eta_{U_1}]\mathscr{D}_1 \subseteq \mathscr{D}_1, \ [R^g(U_2,U_3),\eta_{U_1}]\mathscr{D}_2 \subseteq \mathscr{D}_2
\end{equation}
for all $U_k \in TM, 1 \leq k \leq 3$. The symmetry in pairs of the Riemann curvature and \eqref{R-inv} yield 
$R^g(\mathscr{D}_1,\mathscr{D}_2)=0$. In particular $R^g(\mathscr{D}_2,\mathscr{D}_2)\mathscr{D}_1=0$ after using the algebraic Bianchi identity for $R^g$. Thus taking $U_2,U_3 \in \mathscr{D}_2$ in the first part of \eqref{R-inv2} leads to 
\begin{equation*}
R^g(\mathscr{D}_2,\mathscr{D}_2,\eta_{TM}\mathscr{D}_1,\mathscr{D}_2)=0.
\end{equation*}
In other words $\eta_{TM}\mathscr{D}_1 \subseteq \mathscr{D}_1 \oplus (\V_{R} \cap \mathscr{D}_2)$. As the latter summand vanishes $\mathcal{D}_1$ is invariant under 
$\eta$ and by orthogonality so is $\mathcal{D}_2$ hence one direction of the claim is proved. \\
To prove the claim in the other direction we note that $\E_1,\E_2$ are invariant 
under the action of the holonomy group of $D$. Hence we obtain smooth distributions $\mathcal{D}_1,\mathcal{D}_2$ by parallel transport of $\E_1,\E_2$ to $M$ with respect to 
$D$. Since $\eta$ is $\uh$-invariant, it is invariant under the holonomy group 
of $D$ thus $\eta_{TM}\mathcal{D}_k \subseteq \mathcal{D}_k, k=1,2$. Therefore 
the comparaison formula $D=\nabla^g+\eta$ shows that $\mathcal{D}_k, k=1,2$ are parallel with respect to $\nabla^g$. 
\end{proof}
\begin{rem} \label{non-split}
Local reducibility does not entail decomposability for the transvection 
algebra $\ug$ of the connection $\D$.
\end{rem}
\subsection{Invariant subspaces} \label{sub-str}
In this section $(\mL,\rho,\VV)$ denotes an admissible representation equipped with a background metric $g_{\VV}$. We assume that $\mL$ is simple, unless otherwise stated. First we record the following 
\begin{propn} \label{i-sub}
We have 
\begin{itemize}
\item[(i)]$\Sym^2_{\Kk}(\VV \oplus \H)=\mathrm{Sym}^2_{\Kk}(\VV) \oplus \mathbb{R}1_{\H}$
\item[(ii)] if $\mathscr{S}$ is a $\Kk$-invariant subspace of $\VV \oplus \H$ then 
either $\mathscr{S} \subseteq \VV$ or $\H \subseteq \mathscr{S}$.
\end{itemize}
\end{propn}
\begin{proof}
(i)Clearly $\mathrm{Sym}^2_{\Kk}(\VV \oplus \H)$ is isomorphic to $\mathrm{Sym}^2_{\Kk}(\VV) \oplus 
\mathrm{Hom}_{\Kk}(\H,\VV) \oplus \mathrm{Sym}^2_{\Kk}(\H)$. The middle space vanishes by Proposition 
\ref{met-mix} whilst $\mathrm{Sym}^2_{\Kk}(\H)=\mathbb{R}1_{\H}$ due to the irreducibility of the isotropy representation of $\Kk$ on $\H$.\\
(ii) let $S \in \mathrm{Sym}^2_{\Kk}(\VV \oplus \H)$ be given by $S=1_{\mathscr{S}}-1_{\mathscr{S}^{\perp}}$ where the orthogonal complement is taken with respect to the metric $g$. By (i) we have $S=S_1+\lambda1_{\H}$ where $S_1 \in \mathrm{Sym}^2_{\Kk}(\VV)$ and $\lambda \in \mathbb{R}$. If $\lambda \neq 1$ then $\mathscr{S}=\ker(S-1) \subseteq \VV$. If $\lambda=1$ then $\H \subseteq \mathscr{S}$.
\end{proof}
In the rest of this section we consider $3$-symmetric spaces $M=(\VV \rtimes_{\pi} G) \slash K$ where $G$ is the simply 
connected Lie group with Lie algebra $\mL$ and $\pi$ is the Lie group representation tangent to $\rho$; these spaces will 
be consider equipped with a right invariant metric $g$, as constructed 
previously. 

\begin{propn} \label{null-start}The following hold 
\begin{itemize}
\item[(i)]we have $\V_R=\{v \in \VV: \rho^{+}(\H)v=0\}$
\item[(ii)] $\V_R$ is invariant under the intrinsic torsion tensor $\eta^g$, that is $\eta^g_{\bbV}\V_R \subseteq \V_R$
\item[(iii)] we have $\rho(\mL)\V_R \subseteq \V_R$.
\end{itemize}
\end{propn}
\begin{proof}
(i)by part (ii) in Proposition \ref{curv-1} the restriction $R^g_{\vert \H}$ is either positive or negative definite, according to the type of $\mL$: compact or non compact(see Remark \ref{sigm-R}). This leads immediately to 
$\V_{R} \cap \H=\{0\}$. Because $\V_R$ is $\Kk$-invariant it follows that 
$\V_R \subseteq \VV$ by Proposition \ref{i-sub}, (ii). Since $R^{\D}(\VV,\bbV)=0$ and 
$\tau(\VV,\VV)=0$ formula \eqref{comp-f} reads $[\eta_{v_0}^g,\eta_v^g]=0$ for all 
$v_0 \in \V_R$ respectively $v \in \VV$. Substitute $v \mapsto Jv$; from the invariance property $\eta_{Jv}^g=\eta_v^gJ=-J\eta_v^g$ (see e.g.\eqref{iti}) it follows that 
$\eta_{v_0}^g\eta_v=\eta_v^g \eta_{v_0}^g=0$. In particular $(\eta_{v_0}^g)^2=0$ thus 
$\eta_{v_0}^g=0$ since $\eta_{v_0}^g \in \mathfrak{so}(\bbV,g)$. We have shown that the nullity space 
$\V_R=\{v_{0} \in V:\eta_{v_0}^g=0\}$ and the claim follows from \eqref{its-2}.\\
(ii)$\&$(iii) are proved at the same time. We expand $\rho(x_i)=\rho^{+}(x_i)+\rho^{-}(x_i)$ whenever $x_i \in \H, i=1,2$, with respect to the metric $g_V$. Then 
$$[\rho^{+}(x_1)+\rho^{-}(x_1), \rho^{+}(x_2)+\rho^{-}(x_2)]=\rho[x_1,x_2].$$
Because $\rho(\Kk) \subseteq \so(\VV,g_{\VV})$, the symmetric component 
in the latter identity reads 
$$[\rho^{-}(x_1),\rho^{+}(x_2)]+[\rho^{+}(x_1),\rho^{-}(x_2)]=0.$$
Since $\rho^{\pm}(J_{\H}x)=\rho^{\pm}(x)J_V=-J_V \rho^{\pm}(x)$ for all $x \in \H$, operating the variable change $x_1 \mapsto J_{\H}x_1$ leads easily to 
$$ \rho^{-}(x_1)\rho^{+}(x_2)+\rho^{+}(x_1)\rho^{-}(x_2)=0.
$$  
It follows that $\rho^{+}(x)\rho^{-}(y)\V_R=0$ for all 
$x,y \in \H$ hence  
$\rho^{-}(\H)\V_R \subseteq \V_R$. Using \eqref{its-3} it follows that $\eta_{\H}\V_R \subseteq \V_R$ thus $\V_R$ is invariant under $\eta$. At the same time, since $\rho^{+}(\H)\V_R=0$ by (i), we obtain that $\rho(\H) \V_R \subseteq \V_R$. Taking 
commutators shows that $\rho(\mL) \V_R \subseteq \V_R$ and the claim in (iii) is proved as well.
\end{proof}
Geometrically, Proposition \ref{null-start} ensures that $\V_R$ splits off as a flat Riemannian factor. A more precise result is the following 
\begin{thm} \label{irred-R}
Assume that $\mL$ is simple of non-compact type. Then
\begin{itemize}
\item[(i)] we have $\V_R=0$
\item[(ii)] metrics on $\bbV$ of the form 
$g_V \oplus g_{\H}$ where $g_V$ is isotropy invariant correspond to irreducible Riemannian metrics on $M$.
\end{itemize}
\end{thm}
\begin{proof}
(i) restricting $\rho$ to the $\mL$-invariant space $\V_R$ yields a Lie algebra representation 
$\rho_R:\mL \to \gl(\V_R)$. By Proposition \ref{null-start},(i) we must have $\rho_R(\mL) \subseteq \mathfrak{so}(\V_R)$, in particular the trace form $t_{\rho_R} \leq 0$. As $\mL$ is simple of non-compact type this leads easily to $\rho_R=0$. This 
entails the vanishing of $\V_R$ as $\VV^{\mL}=0$.\\
(ii) since the Riemann nullity vanishes it is enough to determine subspaces $\mathcal{L}$
of $\bbV$ satisyfing (a) and (b) in Proposition \ref{holR-red}.  
As in the proof of Proposition \ref{i-sub} such a subspace corresponds to a tensor $S \in \Sym^2_{\Kk}(\VV \oplus \H) $ 
which moreover satisfies the requirement $S\circ \eta_U=\eta_U \circ S$ for all $U \in \bbV$. By part (i) in Proposition \ref{i-sub} we have $S=S_1+\lambda1_{\H}$ 
where $S_1 \in \Sym^2_{\Kk}(\VV)$. From $S(\eta_vw)=\eta_v(Sw)$ for all $v,w \in \VV$ and $\eta_VV \subseteq \H$(see Proposition \ref{it-semi}) it follows that $\eta_v(S_1-\lambda)w=0$. Combining \eqref{its-2} and the description of $\V_R$ from (i) in Proposition \ref{null-start} yields then $\mathrm{Im}(S_1-\lambda) \subseteq \V_R$ thus $S_1=\lambda1_V$ and further $S=\lambda 1_{\bbV}$. Equivalently either $\mathcal{L}=0$ or 
$\mathcal{L}=\bbV$ and the proof is finished.  
\end{proof}
We also treat below the case when $\mL$ has compact type where the outcome is entirely different.
\begin{thm} \label{irred-Rc}
Assume that $\mL$ is simple of compact type. Then any $3$-symmetric metric on $(\VV \rtimes G)\slash K$ is isometric to 
the product metric $(\VV \times M, h_{\VV} \times g_M)$, where $h_{\VV}$ is a background metric on $\VV$ and $g_M$ is a Hermitian symmetric metric on $M=G \slash K$.
\end{thm}
\begin{proof}
Let $h_V$ be a background metric on $V$; according to Proposition \ref{iso-invm1} any isotropy invariant metric $g_{\VV}$ on $\VV$ satisfies 
$g_{\VV}=f^{\star}\tilde{g}_{\VV}$ where $f \in C(G,\pi,\VV)$ and $h_{\VV}^{-1}\tilde{g}_{\VV}=1+A^{+}$ with $A^{+} \in \mathfrak{c}^{+}(\mL,\rho,\VV)$.
The $3$-symmetric metrics induced by $g_{\H} \oplus \tilde{g}_{\VV}$ and $g_{\H} \oplus g_{\VV}$ are isometric via $F:\VV \times M \to 
\VV \times M, F(v,x):=(fv,x)$; this follows from the explicit form for these metrics in Proposition \ref{R-sem}, together 
with having $[f,\pi(g)]=0$ for all $g \in G$. To see that $\tilde{g}_{\VV}$ induces a product metric we proceed as follows. 
Because $G$ is simply connected we have 
that $A^{+}$ commutes with $\pi(g)$ and also that $\pi(g) \in \mathrm{O}(h_{\VV},\VV)$ for all $g \in G$; see Definition \ref{bgr-m} for the last property, which is entailed by $\mL$ having compact type. Using this facts yields $\pi(g) \in \mathrm{O}(\VV,g_{\VV})$ hence the description of $\VV \rtimes _{\pi} G$ invariant metrics in Proposition \ref{R-sem} ensures that the $3$-symmetric metric 
is the Riemannian product $g_{\VV} \times g_M$.
\end{proof}


\begin{theindex}

\item Admissible Representations, Definition \ref{adm-ff}
\bigskip
\item Almost Hermitian structures
\subitem quasi-K\"ahler, Definition \ref{d-qK}
\subitem almost-K\"ahler, Section \ref{el-qK}
\subitem 3-symmetric, Definition \ref{dd1}, discussion after Proposition \ref{p-311}
\subitem Riemannian 3-symmetric, Definition \ref{R3s}
\bigskip

\item Automorphism groups, see Section \ref{intr-defn}
\subitem of connection $\Aut(M,D)$
\subitem of almost complex structure $\Aut(M,J)$
\subitem of Riemannian metric $\Aut(M,g)$
\subitem of almost Hermitian pair $\Aut(M,g,J)$, see also \eqref{aut-grD}
\bigskip 

\item Background metric,
\subitem w.r.t. a faithful $\rho$,  Definition \ref{bgr-m}
\subitem w.r.t. an admissible $\rho$, Definition \ref{bgr-madm}
\bigskip

\item Casimir-type operator $\mathscr{C}$, Lemma \ref{cent-cpx}
\bigskip

\item Centralizers
\subitem of Lie group representation $C(G,\pi,\VV)$, \eqref{cent-grp-d}  
\subitem of Lie algebra representation $\mathfrak{c}(\mL,\rho,\VV)$ Section \ref{semi-Lie}.
\subitem metric centralizer ${}^{g_V}C(G,\pi,\VV)$, \eqref{met-cent}
 \bigskip

\item Classes of Riemannian $3$-symmetric spaces 
\subitem Type I Spaces, Definition \ref{typeI}
\subitem Type II Spaces, Definition \ref{nil-def}.
\subitem Type III and IV Spaces, Definition \ref{3+4}
\bigskip

\item Foliation, see Proposition \ref{tams-int} 
\subitem polar
\subitem Riemannian  
\subitem holomorphic
\bigskip 

\item Lie algebras
\subitem isotropy $\h$, Equation \ref{can-red}
\subitem transvection, $\mathrm{t}(\bbV)$, \eqref{tans-v}
\subitem transvection of $3$-symmetric algebra, $\underline{\g}$, Section \ref{el-pp}
\subitem holonomy, $\underline{\h}$, Section \ref{el-pp}
\subitem Nomizu, ${}^g\overline{\g}$, \eqref{sandwich}, Equation \ref{Nom}
\subitem Complex Nomizu ${}^g\overline{\g}_J$,  Equation \ref{complexnomizu}
\subitem Lie algebra of Killing generators $\g_b$, Equation \ref{oni}
\bigskip

\item Lie algebra actions
\subitem explicit examples, Section \ref{sat},
\bigskip

\item Infinitesimal model, Proposition \ref{pass2local}
\bigskip

\item Intrinsic torsion $\eta^g$, Definition \ref{d-qK}
\bigskip 

\item Moduli spaces of metrics
\subitem isotropy-invariant metric on 3-symmetric space, Definition \ref{met-c}
\subitem $\mathrm{Met}_{\mathrm{H}} (\bbV)$,  Definition \ref{met-not1}
\subitem $\mathrm{Met}_{\mathfrak{k}}(\VV)$ Def \ref{def-iso-m}
\subitem $\mathcal{M}(\rho, \mL, \VV)$, Definition \ref{moduli-def}
\bigskip

\item Radicals
\subitem of $\mathfrak{\g}$, Section \ref{not-conv}
\subitem of the Killing form, $W$,  Lemma \ref{rad-10}
\bigskip 
\item Riemann curvature nullity $\V_R$, Section \ref{loc-red}
\bigskip

\item Semidirect products
\subitem of Lie algebras, \eqref{semi-1}
\subitem of Lie groups, Section \ref{grp} 
\end{theindex}



\end{document}